\documentclass[a4paper]{amsart}
\usepackage{amsmath,amsthm,amsfonts}
\usepackage[utf8]{inputenc}
\usepackage{amssymb}
\DeclareMathAlphabet\mathbfcal{OMS}{cmsy}{b}{n}
\usepackage[foot]{amsaddr}
\usepackage{stmaryrd}
\usepackage[new]{old-arrows}
\usepackage{enumitem}
\usepackage{tikz-cd}
\usepackage{quiver}
\usepackage{url}
\usepackage[doi=false, backend=biber, style=alphabetic, sorting=nty, maxbibnames=5]{biblatex}
\DeclareFieldFormat[article]{title}{#1}
\DeclareFieldFormat{postnote}{#1}
\DeclareFieldFormat{multipostnote}{#1}
\usepackage[colorlinks, pdfpagelabels, urlcolor = blue, pdfstartview = FitH, bookmarksopen = true, bookmarksnumbered = true, linkcolor = black, plainpages = false, hypertexnames = false, citecolor = blue]{hyperref}
\usepackage{cleveref}
\usepackage{orcidlink}

\makeatletter
\newcommand*{\longdashrightarrow}{%
    \mathrel{\dabar@\dabar@\dabar@\mathchar"0\hexnumber@\symAMSa 4B}}
\makeatother

\newtheorem{theorem}{Theorem}[section]
\newtheorem{corollary}[theorem]{Corollary}
\newtheorem{lemma}[theorem]{Lemma}
\newtheorem{proposition}[theorem]{Proposition}

\newtheorem{claim}{Claim}

\theoremstyle{definition}
\newtheorem{definition}[theorem]{Definition}
\newtheorem{construction}[theorem]{Construction}

\newtheorem{assumption}[theorem]{Assumption}

\theoremstyle{remark}
\newtheorem{remark}[theorem]{Remark}

\newcommand{\Z}{\mathbb{Z}}
\newcommand{\C}{\mathbb{C}}
\newcommand{\Q}{\mathbb{Q}}
\newcommand{\OO}{\mathcal{O}}
\newcommand{\pic}{\mathrm{Pic}}
\newcommand{\im}{\mathrm{im}}

\newcommand{\id}{\mathrm{id}}

\title{Holomorphic 1-forms without zeros on K\"ahler manifolds}
\author{Simon Pietig}
\address{Leibniz Universit\"at Hannover, Institut f\"ur Algebraische Geometrie, Welfengarten 1, 30167 Hannover}
\address{\href{mailto:pietig@math.uni-hannover.de}{pietig@math.uni-hannover.de}}
\address{\href{https://orcid.org/0009-0001-1800-6799}{Orcid iD}}

\begin{document}
\begin{abstract}
    Assuming the existence of good minimal models, we give a bimeromorphic classification of compact connected K\"ahler manifolds admitting a holomorphic 1-form without zeros.
\end{abstract}
\maketitle
\setcounter{tocdepth}{1}
\tableofcontents
\section{Introduction}
Let $X$ be a smooth projective variety over $\C$ with $g$ everywhere linearly independent holomorphic 1-forms $\omega_1,\ldots,\omega_g\in H^0(X,\Omega_X^1)$. In \cite{Popa_Schnell_Kodaira_dimension_zeros_1-forms}, Popa and Schnell showed that the Kodaira dimension of $X$ is bounded by $\kappa(X)\leq\dim X-g$.
\begin{theorem}\label{Popa-Schnell_introduction}\cite[][Thm. 2.1, Sec. 4]{Popa_Schnell_Kodaira_dimension_zeros_1-forms}
    Let $X$ be a smooth projective variety, and let $f\colon X\rightarrow A$ be a morphism to an abelian variety. If there is an ample line bundle $\mathcal{L}$ on $A$ and a positive integer $d$ such that
    \begin{equation*}
        H^0(X,\mathcal{K}_X^d\otimes f^*\mathcal{L}^\vee)\neq 0,
    \end{equation*}
    e.g. if $f$ factors birationally through the Iitaka fibration, then the zero locus $Z(f^*\omega)\neq\varnothing$ is non-empty for all $\omega\in H^0(A,\Omega_A^1)$. In particular, if $X$ admits $g$ everywhere linearly independent holomorphic 1-forms $\omega_1,\ldots,\omega_g$, if $X$ admits a good minimal model $X^{\mathrm{min}}$, and if $f\colon X^{\mathrm{min}}\rightarrow Y$ is the Iitaka fibration, then a general fiber of $f$ is mapped to a translate of a subtorus $A\subset\mathrm{Alb}(X)$ of $\dim A\geq g$.
\end{theorem}
To prove this result, Popa and Schnell applied the theory of mixed Hodge modules. Subsequently, Hao gave the following geometric explanation of this behavior for $\kappa(X)=\dim X-1$, i.e. $g=1$ \cite[][Thm. 1.3]{Hao_Nowhere_Vanishing_Holomorphic_One-Forms_on_Varieties_of_Kodaira_Codimension_One}. Recent work of Church \cite{church_gmm} and of Hao, Wang and Zhang \cite{Hao_Wang_Zhang_gmm} generalized Hao's result to general $g$ assuming the abundance conjecture.
\begin{theorem}\label{structure_dim_g_1-forms_projective}\cite[][Thm. 3.6]{church_gmm}\cite[][Thm. 1.3]{Hao_Wang_Zhang_gmm}
    Let $X$ be a smooth projective variety over $\C$ with $g$ everywhere linearly independent holomorphic 1-forms $\omega_1,\ldots,\omega_g\in H^0(X,\Omega_X^1)$. Suppose that $X$ admits a good minimal model. Then there is a finite \'etale cover $X'\rightarrow X$ such that $X'$ is birational to a product $W\times A$, where $A$ is an abelian variety of $\dim A\geq g$.
\end{theorem}
The main results of our paper are the following generalizations of \Cref{Popa-Schnell_introduction} and of \Cref{structure_dim_g_1-forms_projective} to the K\"ahler case. For more precise versions, see \Cref{Popa_Schnell_Kaehler_case}, \Cref{precise_main_result_higher_codim} and \Cref{corollary_precise_main_result_higher_codim} down below.
\begin{theorem}\label{Popa-Schnell_Kaehler_introduction}
    Let $X$ be a compact connected K\"ahler manifold, and let $f\colon X\rightarrow A$ be a morphism to an abelian variety. If there is an ample line bundle $\mathcal{L}$ on $A$ and a positive integer $d$ such that
    \begin{equation*}
        H^0(X,\mathcal{K}_X^d\otimes f^*\mathcal{L}^\vee)\neq 0,
    \end{equation*}
    e.g. if $f$ factors bimeromorphically through the Iitaka fibration, then the zero locus $Z(f^*\omega)\neq\varnothing$ is non-empty for all $\omega\in H^0(A,\Omega_A^1)$. In particular, if $X$ admits $g$ everywhere linearly independent holomorphic 1-forms $\omega_1,\ldots,\omega_g$, if $X$ admits a good minimal model $X^{\mathrm{min}}$, and if $f\colon X^{\mathrm{min}}\rightarrow Y$ is the Iitaka fibration, then a general fiber of $f$ is mapped to a translate of a subtorus $A\subset\mathrm{Alb}(X)$ of $\dim A\geq g$.
\end{theorem}
\begin{theorem}\label{main_result_higher_codim}
    Let $X$ be a compact connected K\"ahler manifold, which admits a good minimal model $X^{\mathrm{min}}$. Let $f\colon X^{\mathrm{min}}\rightarrow Y$ be the Iitaka fibration, and suppose that a general fiber of $f$ is mapped to a translate of a subtorus $A\subset\mathrm{Alb}(X)$ of $\dim A=g>0$. Then there is a finite \'etale cover $X'\rightarrow X$, such that $X'$ is bimeromorphic to a locally trivial torus fiber bundle $\pi\colon Z\rightarrow W$ with trivial monodromy over a compact connected K\"ahler manifold $W$. The fibers are of $\dim\pi^{-1}(y)\geq g$.
\end{theorem}
\begin{corollary}
    Let $X$ be a compact connected K\"ahler manifold, which admits $g$ everywhere linearly independent holomorphic 1-forms $\omega_1,\ldots,\omega_g\in H^0(X,\Omega_X^1)$. Suppose that $X$ admits a good minimal model. Then there is a finite \'etale cover $X'\rightarrow X$, such that $X'$ is bimeromorphic to a locally trivial torus fiber bundle $\pi\colon Z\rightarrow W$ with trivial monodromy over a compact connected K\"ahler manifold $W$. The fibers are of $\dim\pi^{-1}(y)\geq g$.
\end{corollary}
\begin{remark}
    In contrast to the projective situation, one cannot expect that the finite \'etale covering space $X'$ of $X$ in \Cref{main_result_higher_codim} is bimeromorphic to a product $Z\cong W\times A$, see \cite[][Rem. 2.10]{Pietig_1-forms_Kaehler_threefolds}.
\end{remark}
\begin{remark}
    We suspect that the conclusion of \Cref{Popa-Schnell_introduction} holds true under a weaker (e.g. topological) assumption.\par
    Let $X$ be a compact connected K\"ahler manifold. Recall Kotschick's conjecture \cite{Kotschick_1-forms_fibrations_char-numbers}, which claims that the following conditions are equivalent:
    \begin{enumerate}[label=\textit{(\Alph*)}]
        \item\label{conditionA} $X$ admits a holomorphic one-form without zeros;
        \item\label{conditionB} $X$ admits a submersion to $S^1$.
    \end{enumerate}
    The implication $\ref{conditionA}\Rightarrow\ref{conditionB}$ was proven by Kotschick \cite[][Prop. 7]{Kotschick_1-forms_fibrations_char-numbers}. Corr{\^e}a, Koll{\'a}r, Schreieder and Wang constructed a counterexample to the implication $\ref{conditionB}\Rightarrow\ref{conditionA}$ \cite[][Thm. 26]{Correa_Kollar_Schreieder_Wang}. Their counterexample is a smooth complex projective variety $X$ of dimension 5 such that the Albanese morphism $\mathrm{alb}_X\colon X\rightarrow\mathrm{Alb}(X)$ is homotopy equivalent to a submersion, and every holomorphic 1-form $\omega\in H^0(X,\Omega^1_X)$ has a zero. The variety $X$ is bimeromorphic to a product $\mathbb{P}^4\times E$, where $E$ is an elliptic curve. In \cite[][Rem. 27]{Correa_Kollar_Schreieder_Wang}, the authors found for each $n\geq 7$ smooth complex projective varieties $X_{-\infty},X_0,\ldots X_{n-1}$ of Kodaira dimension $\kappa(X_k)=k$ such that the Albanese morphism $\mathrm{alb}_{X_k}\colon X_k\rightarrow\mathrm{Alb}(X_k)$ is homotopy equivalent to a submersion, and every holomorphic 1-form $\omega\in H^0(X_k,\Omega^1_{X_k})$ has a zero for all $k\in\{-\infty,0,\ldots,n-1\}$. As before the varieties $X_k$ are bimeromorphic to a product $Y\times E$, where $E$ is an elliptic curve.\par
    Note that even though the varieties $X_k$ do not admit holomorphic 1-forms without zeros, the conclusion in \Cref{Popa-Schnell_introduction} still holds true for the varieties $X_k$. This leads us to the following question: Let $X$ be a compact connected K\"ahler manifold which admits a submersion $f\colon X\rightarrow (S^1)^{2g}$, where $g>0$ is a positive integer. Suppose that $X$ admits a good minimal model $X^{\mathrm{min}}$. Let $f\colon X^{\mathrm{min}}\rightarrow Y$ be the Iitaka fibration. Is then the general fiber of $f$ mapped to a translate of a subtorus $A\subset\mathrm{Alb}(X)$ of $\dim A\geq g$?
\end{remark}
\begin{remark}
    OpenAI's recent results \cite[][Thm. 1.1]{OAI:Orbifold-and-logarithmic-Iitaka-subadditivity-September-26-2026}, \cite[][Thm. 1.2]{OAI:Log-abundance-for-compact-Kahler-spaces-under-logarithmic-Iitaka-subadditivity-October-4-2026} prove the abundance conjecture. We thus only need to assume the existence of a minimal model in \Cref{Popa-Schnell_Kaehler_introduction} and in \Cref{main_result_higher_codim}. However, since, to the author's knowledge, OpenAI's result have not yet passed the review process, we nonetheless assume the existence of a good minimal model in \Cref{Popa-Schnell_Kaehler_introduction} and in \Cref{main_result_higher_codim}.
\end{remark}
\subsection{Outline of the argument}
To prove \Cref{Popa-Schnell_Kaehler_introduction}, we follow Villadsen's proof \cite{Villadsen_proof_Popa-Schnell} of Popa's and Schnell's theorem \cite[][Thm. 2.1]{Popa_Schnell_Kodaira_dimension_zeros_1-forms}. The main technical tool is Simpson's structure theorem of the cohomology jump loci
\begin{equation*}
    \Sigma^k_m(X):=\{(\mathcal{L},\omega)\in M_{\mathrm{Dol}}(X)\mid\dim H_{\mathrm{Dol}}^k(X,\mathcal{L},\omega)\geq m\}
\end{equation*}
of Higgs line bundles \cite[][Thm. 4.2]{Simpson_moduli_rank_1_local_systems}. Simpson shows that in the projective case, the irreducible components of $\Sigma^k_m(X)$ are translates of images of pullback maps
\begin{equation*}
    g_{\mathrm{Dol}}\colon M_{\mathrm{Dol}}(T)\longrightarrow M_{\mathrm{Dol}}(X),
\end{equation*}
where $g\colon X\rightarrow T$ is a morphism to a torus. This structure result also holds if $X$ is a compact connected K\"ahler manifold. In particular, Villadsen's proof remains valid in the K\"ahler case.\par
We briefly recall the proof of \Cref{structure_dim_g_1-forms_projective}: Let $X$ be a smooth projective variety over $\C$ with $g$ everywhere linearly independent holomorphic 1-forms $\omega_1,\ldots,\omega_g\in H^0(X,\Omega_X^1)$. Let $X^{\mathrm{min}}$ be a good minimal model of $X$, and consider the Iitaka fibration $f\colon X^{\mathrm{min}}\rightarrow Y$. By \cite[][Thm. 2.1]{Popa_Schnell_Kodaira_dimension_zeros_1-forms}, the holomorphic 1-forms $\omega_1,\ldots,\omega_g$ cannot come from $Y$. In particular, the general fiber of $f$ is mapped to a translate of a fixed subtorus $A\subset \mathrm{Alb}(X)$ of $\dim A\geq g$. As $\mathrm{Alb}(X)$ is projective, we can dualize the inclusion $A\subset \mathrm{Alb}(X)$ to get a morphism
\begin{equation*}
    \varphi\colon X^{\mathrm{min}}\longrightarrow \mathrm{Alb}(X)\longrightarrow A.
\end{equation*}
The variety $W$ in \Cref{structure_dim_g_1-forms_projective}, as well as the splitting $X'\cong W\times A$ are then constructed from a general fiber of $\varphi\colon X^{\mathrm{min}}\rightarrow A$. In \cite{Hao_Schreieder_3-folds}, Hao and Schreieder use a similar argument to show that a smooth projective threefold of Kodaira codimension 1 over $\C$, that fibers over the circle $S^1$, admits a nowhere vanishing holomorphic 1-form.\par
Now, let $X$ be a compact connected K\"ahler manifold with $g$ everywhere linearly independent holomorphic 1-forms $\omega_1,\ldots,\omega_g\in H^0(X,\Omega^1_X)$. Note that $\mathrm{Alb}(X)$ need not be projective, and therefore, the morphism $\varphi$ does not exist in general.\par
To overcome this obstacle, suppose for simplicity that the base of the Iitaka fibration $f\colon X^{\mathrm{min}}\rightarrow Y$ is smooth. Consider the fiber product $\pi\colon Z:=\mathrm{Alb}(X)\times_{\mathrm{Alb}(Y)}Y\rightarrow Y$. We then obtain the commutative diagram
\[\begin{tikzcd}
	{X^{\mathrm{min}}} && Z \\
	& Y & {.}
	\arrow["h", from=1-1, to=1-3]
	\arrow["f"', from=1-1, to=2-2]
	\arrow["\pi", from=1-3, to=2-2]
\end{tikzcd}\]
We first show that $\pi\colon Z\rightarrow Y$ is a locally trivial torus fiber bundle with trivial monodromy and fiber $A$. This follows from the fact that $\mathrm{Alb}(X)\rightarrow\mathrm{Alb}(Y)$ is surjective with connected kernel, which in turn is true since $f$ is a fibration. In the second step, we prove that after replacing $X$ by a finite \'etale cover if necessary, we can assume that $h$ has connected fibers. In the third step, we prove that after replacing $X$ by a further finite \'etale cover, there is an $A$-action on $X^{\mathrm{min}}$ such that $h$ is an equivariant morphism. The $A$-action on $Z$ is given by translation on the fibers of $\pi$. The compact connected K\"ahler manifold $W$ is then constructed from the quotient of $X^{\mathrm{min}}/A$.\par
In every step, we use that $\pi$ is locally trivial, i.e. we can find an open cover $\{U_i\}_{i\in I}$ of $Y$ such that $\pi^{-1}(U_i)$ is isomorphic to $U_i\times A$ over $U_i$. In particular, the above diagram is locally of the form
\[\begin{tikzcd}
	{X_i:=f^{-1}(U_i)} && {U_i\times A} & \\
	& {U_i} && {A.}
	\arrow["{h_i}", from=1-1, to=1-3]
	\arrow["{f_i}"', from=1-1, to=2-2]
	\arrow["{\pi_i}", from=1-3, to=2-2]
	\arrow["{p_i}"', from=1-3, to=2-4]
\end{tikzcd}\]
The composition $X_i\rightarrow U_i\times A\rightarrow A$ can be viewed as the local equivalent of
\begin{equation*}
    \varphi\colon X^{\mathrm{min}}\longrightarrow \mathrm{Alb}(X)\longrightarrow A
\end{equation*}
in the projective case. We can thus perform the constructions in the projective case locally, and glue them to obtain the global analogue.\par
Let us elaborate on steps 2-4. In the second step, we consider for each $a\in A$ the Stein factorization of $h_i^{-1}(U_i\times\{a\})\rightarrow U_i$
\begin{equation*}
    h_i^{-1}(U_i\times\{a\})\longrightarrow V_i(a)\longrightarrow U_i.
\end{equation*}
We prove that the finite morphism $V_i(a)\rightarrow U_i$ is independent of $a\in A$, and these local patches glue to a finite morphism $V\rightarrow Y$. The morphism $h_V\colon X^{\mathrm{min}}\times_YV\rightarrow V$ has then connected fibers. Moreover, a local analysis shows that the base change $X^{\mathrm{min}}\times_YV\rightarrow X^{\mathrm{min}}$ induces a finite \'etale cover $X'\rightarrow X$. This concludes the second step.\par
In the third step, suppose for simplicity that $X^{\mathrm{min}}$ is smooth, and that $f$ is smooth over an affine open subvariety $Y^*\subset Y$. Set $X^*:=f^{-1}(Y^*)$. The fibers of $f$ over $Y^*$ are hence compact connected K\"ahler manifolds with vanishing first real Chern class. The Beauville--Bogomolov decomposition thus shows that for each $y\in Y^*$ the map
\begin{equation*}
    \mathrm{Aut}^\circ(f^{-1}(y))\longrightarrow\mathrm{Alb}(f^{-1}(y))
\end{equation*}
is an isogeny. We construct the relative Albanese of $f^*\colon X^*\rightarrow Y^*$, which is a torus fiber bundle $\xi\colon\mathcal{A}lb\rightarrow Y^*$. In general, we cannot assume that $\xi$ has a global section, since $f^*$ does not necessarily admit a global section. Note that the relative Albanese fits into the composition
\begin{equation*}
    X^*\longrightarrow\mathcal{A}lb\longrightarrow Z.
\end{equation*}
Let $\mathcal{S}$ be the sheaf of local sections of $\xi\colon\mathcal{A}lb\rightarrow Y^*$. One can then show that the induced map
\begin{equation*}
    \mathrm{Aut}^\circ(X^*/Y^*)\longrightarrow H^0(Y^*,\mathcal{S})^\circ
\end{equation*}
is surjective with finite kernel, where $H^0(Y^*,\mathcal{S})^\circ\subset H^0(Y^*,\mathcal{S})$ is the identity component. This is the relative version of the isogeny $\mathrm{Aut}^\circ(f^{-1}(y))\rightarrow\mathrm{Alb}(f^{-1}(y))$. Using Saito's decomposition theorem for K\"ahler morphisms \cite[][Thm. 0.5]{Saito_Kaehler_decomposition}, we prove that there is an isogeny $A'\rightarrow A$ such that
\begin{equation*}
    A'\subset H^0(Y^*,\mathcal{S})^\circ.
\end{equation*}
In particular, if $A''\subset\mathrm{Aut}^\circ(X^*/Y^*)$ is defined as the pre-image of $A'$ along $\mathrm{Aut}^\circ(X^*/Y^*)\rightarrow H^0(Y^*,\mathcal{S})^\circ$, then the composition $A''\rightarrow A'\rightarrow A$ is an isogeny. We construct a finite \'etale cover $Z'\rightarrow Z$, which is on each fiber over $Y$ the isogeny $A''\rightarrow A$. We show that the $A''$-action on $X^*$ can be extended to an $A''$-action on $X^{\mathrm{min}}$. Then $A''$ acts on the base change $X^{\mathrm{min}}\times_ZZ'$, and on $Z'$ such that $X^{\mathrm{min}}\times_ZZ'\rightarrow Z'$ is $A''$-equivariant. This concludes the third step.\par
It remains to construct the compact connected K\"ahler manifold $W$. Note that since $f$ and $h$ have connected fibers by the second step, there is for a fixed $a_0\in A$ one irreducible component of $h_i^{-1}(U_i\times\{a_0\})$ that dominates $U_i$. Denote this component by
\begin{equation*}
    W_i(a_0)\subset h_i^{-1}(U_i\times\{a\}).
\end{equation*}
Denote the $A$-action from step three by $\mathrm{tr}(a)\colon X^{\mathrm{min}}\rightarrow X^{\mathrm{min}}$ for $a\in A$. This allows to glue the local patches $W_i(a_0)\rightarrow U_i$ to the global $W\rightarrow Y$. We can then define the local meromorphic maps
\begin{align*}
    X_i&\longdashrightarrow W_i(a_0)\\
    x&\longmapsto\mathrm{tr}(a_0-p_i(h_i(x))),
\end{align*}
which can be viewed as local patches of the quotient map $X^{\mathrm{min}}\rightarrow X^{\mathrm{min}}/A$. These local meromorphic maps glue to a global meromorphic map $X^{\mathrm{min}}\longdashrightarrow W$. We hence obtain the meromorphic map
\begin{equation*}
    X^{\mathrm{min}}\longdashrightarrow W\times_YZ,
\end{equation*}
which is bimeromorphic. As the projection $W\times_YZ\rightarrow W$ is a locally trivial torus fiber bundle with trivial monodromy, this proves \Cref{main_result_higher_codim}.
\subsection{Conventions and notation}
By a complex analytic variety, we mean a Hausdorff, second-countable, irreducible, and reduced complex space. A complex analytic variety is called K\"ahler if it admits a K\"ahler metric in the sense of \cite[][II. 1]{Varouchas_Kaehler_spaces_proper_open_morphisms}. On a normal complex analytic variety $X$, we define the canonical sheaf as $\mathcal{K}_X:=i_*\mathcal{K}_{X^{\mathrm{sm}}}$, and the $m$-pluricanonical sheaf as $\mathcal{K}_X^m:=i_*\mathcal{K}_{X^{\mathrm{sm}}}^m$ for all $m\geq 0$, where $i\colon X^{\mathrm{sm}}\hookrightarrow X$ is the inclusion of the smooth locus. Similarly, we define the sheaf of reflexive 1-forms as $\Omega_X^{[1]}:=i_*\Omega_{X^{\mathrm{sm}}}^1$. A compact complex analytic K\"ahler variety $X$ is called $\Q$-factorial, if every Weil divisor is $\Q$-Cartier and if there is a positive integer $m$ such that $\mathcal{K}_X^m$ is invertible. We call $X$ minimal if it is normal, $\Q$-factorial, has terminal singularities, and the first Chern class $c_1^{BC}(\mathcal{K}_X)\in H_{BC}^{1,1}(X)$ of the canonical sheaf $\mathcal{K}_X$ is nef. It is furthermore called a good minimal model, if a positive multiple of $\mathcal{K}_X$ is globally generated. We work in the analytic topology rather than the Zariski topology. In particular, local properties of objects or morphisms are properties which hold true over analytically open subsets. Consider morphisms $X\rightarrow S\leftarrow Y$ between normal complex analytic varieties. If $X\times_SY$ is irreducible, we denote by $X\times_S^{\mathrm{norm}}Y$ the normalization of $X\times_SY$.
\subsection*{Acknowledgments}
I would like to thank Stefan Schreieder for suggesting this problem to me, for very helpful discussions, and for comments on this manuscript. The research was conducted in the framework of the DFG-funded research training group RTG 2965: From Geometry to Numbers.
\subsection{AI disclosure}
The mathematical content of this paper is the work of the author. Claude was used to assist with proofreading and to improve the clarity of the manuscript.
\section{Preliminaries}
\subsection{Quasi-\'etale morphisms}
\begin{definition}\cite[][Def. 1.1]{Catanese_quasi_etale}
    Let $X$ and $Y$ be normal complex analytic varieties. A surjective morphism $f\colon X\rightarrow Y$ is called \textit{quasi-\'etale} if there is a closed subvariety $Z\subset X$ of codimension $\geq 2$ such that $f\vert_{X\setminus Z}\colon X\setminus Z\rightarrow Y$ is \'etale.
\end{definition}
\begin{remark}\cite[][Rem. 3.1]{Catanese_quasi_etale}
    Let $f\colon X\rightarrow Y$ be a quasi-\'etale morphism between complex analytic varieties. Suppose that $Y$ is smooth and that $X$ is normal. Then $f$ is \'etale. In particular, the branch locus of a quasi-\'etale morphism $f\colon X\rightarrow Y$ is contained in the singular locus of $Y$.
\end{remark}
We need the following lemma, cf. \cite[][Prop. 3.2]{Catanese_quasi_etale}
\begin{lemma}\label{canonical_sing_plus_invertible_ascends_if_quasi-etale}
    Let $f\colon X\rightarrow Y$ be a quasi-\'etale morphism between normal complex analytic varieties.
    \begin{enumerate}[label=(\roman*)]
        \item If $\mathcal{K}_Y^m$ is invertible for some positive integer $m$, then $\mathcal{K}_X^m\cong f^*\mathcal{K}_Y^m$.
        \item Suppose that $Y$ has canonical singularities. Then $X$ has canonical singularities.
        \item Suppose that $Y$ has terminal singularities. Then $X$ has terminal singularities.
    \end{enumerate}
\end{lemma}
\begin{proof}
    To prove the first assertion, let $Z\subset X$ be a closed subvariety of codimension $\geq 2$ such that $f\vert_{X\setminus Z}\colon X\setminus Z\rightarrow Y$ is \'etale. Define $i\colon X^\circ:=X\setminus Z\hookrightarrow X$, and $j\colon Y^\circ:=Y\setminus f(Z)\hookrightarrow Y$. Hence, the morphism $f^\circ\colon X^\circ\rightarrow Y^\circ$ is surjective \'etale. After removing codimension $\geq 2$ subvarieties from $X^\circ$, and $Y^\circ$, we can furthermore assume that $X^\circ$, and $Y^\circ$ are smooth. Note that the codimensions of $X\setminus X^\circ$ and of $Y\setminus Y^\circ$ is at least two. Since $f^\circ$ is \'etale, $\mathcal{K}_{X^\circ}=(f^\circ)^*\mathcal{K}_{Y^\circ}$. Taking tensor powers shows $\mathcal{K}_{X^\circ}^m=(f^\circ)^*\mathcal{K}_{Y^\circ}^m$ for all positive integers $m$.\par
    By assumption, there is a positive integer $m$, such that $\mathcal{K}_Y^m:=j_*\mathcal{K}_{Y^\circ}^m$ is invertible. Thus, also $f^*\mathcal{K}_Y^m$ is invertible, and hence reflexive. Note that
    \begin{equation*}
        \mathcal{K}_X^m:=i_*\mathcal{K}_{X^\circ}^m=i_*(f^\circ)^*\mathcal{K}_{Y^\circ}^m
    \end{equation*}
    is a reflexive sheaf that is isomorphic to $f^*\mathcal{K}_Y^m$ on $X^\circ$. As $X\setminus X^\circ$ has codimension $\geq 2$, $\mathcal{K}_X^m\cong f^*\mathcal{K}_Y^m$. This proves the first assertion.\par
    To prove the second and third assertion, consider the Stein factorization
    \begin{equation*}
        X\stackrel{\mu}{\longrightarrow}X'\stackrel{\tau}{\longrightarrow}Y
    \end{equation*}
    of $f\colon X\rightarrow Y$. Note that $\tau$ is a finite quasi-\'etale morphism. Moreover, $\mu$ is proper bimeromorphic, and its exceptional locus is of codimension $\geq 2$. From \cite[][Thm. 5.20 (3)]{Kollar_Mori}, it follows that $X'$ has canonical, resp. terminal singularities. Let $\nu\colon\tilde{X}\rightarrow X$ be a resolution of singularities. Since $\mu$ does not contract any divisor, the exceptional loci of $\nu$ and $\mu\circ\nu$ agree, and so do the discrepancies. Thus, $X$ has canonical, resp. terminal singularities.
\end{proof}
\begin{lemma}\label{good_minimal_model_ascend_finite_etale_cover}
    Let $X$ be a compact connected K\"ahler manifold, which admits a good minimal model $X^{\mathrm{min}}$. Let $X'\rightarrow X$ be a finite \'etale cover. The following hold true.
    \begin{enumerate}[label=(\roman*)]
        \item $X'$ admits a good minimal model $X'^{\mathrm{min}}$, which admits a quasi-\'etale morphism $\tau\colon X'^{\mathrm{min}}\rightarrow X^{\mathrm{min}}$.
        \item If we denote the Iitaka fibration of $X^{\mathrm{min}}$ by $f\colon X^{\mathrm{min}}\rightarrow Y$, then the Iitaka fibration of $X'^{\mathrm{min}}$ is given as the morphism with connected fibers $f'\colon X'^{\mathrm{min}}\rightarrow Y'$ in the Stein factorization of $X'^{\mathrm{min}}\rightarrow X^{\mathrm{min}}\rightarrow Y$.
    \end{enumerate}
\end{lemma}
\begin{proof}
    By Hironaka's resolution of indeterminacies, we can choose a compact connected K\"ahler manifold $\tilde{X}$ with proper bimeromorphic morphisms $X\leftarrow\tilde{X}\rightarrow X^{\mathrm{min}}$. The finite \'etale cover $X'\rightarrow X$ induced a finite \'etale cover $\tilde{X}'\rightarrow\tilde{X}$. Consider the Stein factorization
    \[\begin{tikzcd}
	   {\tilde{X}'} & {\tilde{X}} & {X^{\mathrm{min}}} \\
	   & {Z} & {.}
	   \arrow[from=1-1, to=1-2]
	   \arrow[from=1-1, to=2-2]
	   \arrow[from=1-2, to=1-3]
	   \arrow[from=2-2, to=1-3]
    \end{tikzcd}\]
    Note that the morphism $Z\rightarrow X^{\mathrm{min}}$ is finite \'etale outside its singular locus. In particular, it is finite quasi-\'etale. We can hence apply \Cref{canonical_sing_plus_invertible_ascends_if_quasi-etale} to conclude that $Z$ has terminal singularities. By \cite[][Thm. 1.5]{Hacon_Li_Xie_Q-factorialization}, there is a projective bimeromorphic morphism $\mu\colon X'^{\min}\rightarrow Z$, such that $X'^{\mathrm{min}}$ is normal and $\Q$-factorial, and such that the exceptional locus of $\mu$ has codimension $\geq 2$. Define $\tau$ as the composition
    \begin{equation*}
        \tau\colon X'^{\mathrm{min}}\longrightarrow Z\longrightarrow X^{\mathrm{min}}.
    \end{equation*}
    As the exceptional locus of $\mu$ has codimension $\geq 2$, and as $Z\rightarrow X^{\mathrm{min}}$ is finite quasi-\'etale, the composition $\tau$ is a quasi-\'etale morphism. We can hence apply \Cref{canonical_sing_plus_invertible_ascends_if_quasi-etale} again to conclude that $X'^{\mathrm{min}}$ has terminal singularities.\par
    Since $X^{\mathrm{min}}$ is $\Q$-factorial, there is a positive integer, such that $\mathcal{K}_{X^{\mathrm{min}}}^m$ is invertible. Thus, \Cref{canonical_sing_plus_invertible_ascends_if_quasi-etale} implies that $\mathcal{K}_{X'^{\mathrm{min}}}^m\cong\tau^*\mathcal{K}_{X^{\mathrm{min}}}^m$. In particular, as $\mathcal{K}_{X^{\mathrm{min}}}$ is nef, $\mathcal{K}_{X'^{\mathrm{min}}}$ is nef. Since there is a positive integer $n$ such that $\mathcal{K}_{X^{\mathrm{min}}}^n$ is globally generated, also $\tau^*\mathcal{K}_{X^\mathrm{min}}^{mn}\cong\mathcal{K}_{X'^{\mathrm{min}}}^{mn}$ is globally generated. Thus, $X'^{\mathrm{min}}$ is a good minimal model for $X'$. This proves the first assertion.\par
    To prove the second assertion, consider the commutative diagram
    \[\begin{tikzcd}
	   {X'^{\mathrm{min}}} & {\mathbb{P}(H^0(X'^{\mathrm{min}},\mathcal{K}_{X'^{\mathrm{min}}}^{mn})^\vee)} \\
	   {X^{\mathrm{min}}} & {\mathbb{P}(H^0(X^{\mathrm{min}},\mathcal{K}_{X^{\mathrm{min}}}^{mn})^\vee)}
	   \arrow["{\vert\mathcal{K}_{X'^{\mathrm{min}}}^{mn}\vert}", from=1-1, to=1-2]
	   \arrow["\tau", from=1-1, to=2-1]
	   \arrow[dashed, from=1-2, to=2-2]
	   \arrow["{\vert\mathcal{K}_{X^{\mathrm{min}}}^{mn}\vert}", from=2-1, to=2-2]
    \end{tikzcd}\]
    Taking the Stein factorizations of the horizontal arrows induces the commutative diagram
    \[\begin{tikzcd}
	   {X'^{\mathrm{min}}} & Z & {Y'} \\
	   & {X^{\mathrm{min}}} & Y
	   \arrow[from=1-1, to=1-2]
	   \arrow["{f'}", curve={height=-12pt}, from=1-1, to=1-3]
	   \arrow["\tau"', from=1-1, to=2-2]
	   \arrow[from=1-2, to=1-3]
	   \arrow[from=1-2, to=2-2]
	   \arrow[from=1-3, to=2-3]
	   \arrow["f", from=2-2, to=2-3]
    \end{tikzcd}\]
    where the morphism $Y'\rightarrow Y$ is finite, and $f$, resp. $f'$ are the Iitaka fibrations. Thus, the composition through the upper right hand corner is the Stein factorization of the morphism through the lower left hand corner.
\end{proof}
\subsection{Torus fiber bundles}
Most of the following material can be found in \cite{Claudon_tori_fiber_bundles}.
\begin{definition}\label{definition_torus_fiber_bundle}
    A \textit{torus fiber bundle} is defined as a smooth proper surjective morphism $f\colon X\rightarrow Y$ from a connected complex manifold $X$ to a connected complex manifold $Y$ such that every fiber is isomorphic to a positive-dimensional torus.
\end{definition}
Let $A=\C^g/\Lambda$ be a torus. Recall that the Albanese morphism with base point $0\in A$
\begin{equation*}
    A\stackrel{\sim}{\longrightarrow} \mathrm{Alb}(A):=H^0(A,\Omega^1_A)^\vee/H_1(A,\Z)
\end{equation*}
is a canonical isomorphism of tori. Poincar\'e duality yields a canonical isomorphism $H_1(A,\Z)\cong H^{2g-1}(A,\Z)$, and Serre duality yields a canonical isomorphism $H^0(A,\Omega^1_A)^\vee\cong H^{g-1,g}(A)$. Writing furthermore $H^{g-1,g}(A)\cong H^{2g-1}(A,\C)/F^g$, where $F^g$ is the degree $g$-part of the Hodge filtration, gives rise to a canonical isomorphism
\begin{equation*}
    A\stackrel{\sim}{\longrightarrow}\mathrm{Alb}(A)\cong (H^{2g-1}(A,\C)/F^g)/H^{2g-1}(A,\Z).
\end{equation*}
In particular, $A$ is determined by its Hodge structure in degree $2g-1$.
Let $f\colon X\rightarrow Y$ be a torus fiber bundle of fiber dimension $g>0$. Let $B\subset Y$ be an open subset isomorphic to an open ball $B\subset\C^n$ centered at the origin. Denote by $X_b$ the fiber over $b\in B$. Ehresmann's theorem \cite[][Thm. 9.3]{Voisin_HT_1} implies that the restriction maps $H^{2g-1}(f^{-1}(B),\C)\stackrel{\sim}{\rightarrow}H^{2g-1}(X_b,\C)$ are isomorphisms for all $b\in B$. We hence obtain canonical isomorphisms
    \begin{equation*}
        H^{2g-1}(X_b,\C)\cong H^{2g-1}(f^{-1}(B),\C)\cong H^{2g-1}(X_0,\C).
    \end{equation*}
    Denote by $F^g(b)$ the $g$-th part of the Hodge filtration $F^g(b)\subset H^{2g-1}(X_b,\C)$ for $b\in B$. By \cite[][Thm. 10.9]{Voisin_HT_1}, the period map
    \begin{align*}
        B&\longrightarrow \mathrm{Gr}(g, H^{2g-1}(X_0,\C))\\
        b&\longmapsto F^g(b)
    \end{align*}
    to the Grassmannian $\mathrm{Gr}(g, H^{2g-1}(X_0,\C))$ of $g$-dimensional subspaces in $H^{2g-1}(X_0,\C)$ is holomorphic. Thus, the subspaces $F^g(b)\subset H^{2g-1}(X_0,\C)$ vary holomorphically, and hence the quotient
    \begin{equation*}
        (R^{2g-1}f_*\Z\otimes\OO_Y)/(R^{g-1}f_*\Omega^g_{X/Y})
    \end{equation*}
    is a holomorphic vector bundle on $Y$. This gives rise to the following definitions.
\begin{definition}\label{def_variation_HS_Jacobian_fibration}
    Let $f\colon X\rightarrow Y$ be a torus fiber bundle of fiber dimension $g>0$.
    \begin{enumerate}[label=(\roman*)]
        \item The \textit{variation of Hodge structures $H$} associated to $f$ consists of the following data: the local system $H_\Z:=R^{2g-1}f_*\Z$ on $Y$ of free abelian groups of rank $2g$, the associated vector bundle $\mathcal{H}:=H_\Z\otimes\OO_Y$ on $Y$, the holomorphic subbundle $\mathcal{F}^{g}:=R^{g-1}f_*\Omega^g_{X/Y}$, and the quotient $\mathcal{H}^{g-1,g}:=\mathcal{H}/\mathcal{F}^{g}$.
        \item We say that $f$ has \textit{trivial monodromy} if the local system $H_\Z$ is isomorphic to $\Z^{2g}$.
        \item We define the \textit{Jacobian fibration} associated to $f$ as the complex manifold
        \begin{equation*}
            J_H:=\mathcal{H}^{g-1,g}/H_\Z
        \end{equation*}
        obtained by taking the quotient of the vector bundle $\mathcal{H}^{g-1,g}\rightarrow Y$ by the $H_\Z$-action together with the natural morphism $\pi\colon J_H\rightarrow Y$.
        \item The zero section of $\mathcal{H}^{g-1,g}\rightarrow Y$ induces the global section $\sigma_0\colon Y\rightarrow J_H$. The \textit{sheaf of local sections} of $\pi\colon J_H\rightarrow Y$ is a sheaf of abelian groups with addition of local sections and neutral element $\sigma_0$. It is denoted by $\mathcal{J}_H=\mathcal{H}^{g-1,g}/H_\Z$, where we identify the vector bundle $\mathcal{H}^{g-1,g}$ with its sheaf of local sections.
        \item Let $\eta\in H^1(Y,\mathcal{J}_H)$ be a cohomology class. The $\mathcal{J}_H$-torsor associated to $\eta$ is denoted by
        \begin{equation*}
            \pi^\eta\colon J^\eta_H\rightarrow Y.
        \end{equation*}
        \item We denote by $\mathrm{exp}\colon H^1(Y,\mathcal{H}^{g-1,g})\rightarrow H^1(Y,\mathcal{J}_H)$ and $c\colon H^1(Y,\mathcal{J}_H)\rightarrow H^2(Y,H_\Z)$ the maps induced by the short exact sequence $0\rightarrow H_\Z\rightarrow\mathcal{H}^{g-1,g}\rightarrow\mathcal{J}_H\rightarrow 0$, and call them the \textit{exponential}, resp. the \textit{Chern class} map.
    \end{enumerate}
\end{definition}
\begin{proposition}\label{classification_torus_fiber_bundles}\cite[][Prop. 2.1]{Claudon_tori_fiber_bundles}
    Let $f\colon X\rightarrow Y$ be a torus fiber bundle with associated variation of Hodge structures $H$. Then there is a unique cohomology class $\eta\in H^1(Y,\mathcal{J}_H)$ and an isomorphism $\varphi\colon X\stackrel{\sim}{\rightarrow}J_H^\eta$ over $Y$. Moreover, if $f\colon X\rightarrow Y$ admits a global section $\sigma\colon Y\rightarrow X$ then $\eta=0$ and $\varphi\colon X\rightarrow J_H$ can be chosen such that $\varphi\circ\sigma=\sigma_0$. Conversely, any cohomology class $\eta\in H^1(Y,\mathcal{J}_H)$ defines a torus fiber bundle $\pi^\eta\colon J_H^\eta\rightarrow Y$.
\end{proposition}
\begin{corollary}\label{corollary_isotrivial_and_trivial_vhs}
    Let $f\colon X\rightarrow Y$ be a torus fiber bundle of fiber dimension $g$. Suppose that $f$ is isotrivial with fiber $A\cong\C^g/\Lambda$ and that its monodromy is trivial. Then $J_H=Y\times A$ and $\mathcal{J}_H\cong\OO_Y^{\oplus g}/\Lambda$.
\end{corollary}
\begin{definition}\label{definition_isotrivial_and_trivial_vhs}
    In the situation of \Cref{corollary_isotrivial_and_trivial_vhs}, we denote the $\mathcal{J}_H$ torsor associated to $\eta\in H^1(Y,\mathcal{J}_H)$ by $\pi^\eta\colon(Y\times A)^\eta\rightarrow Y$.
\end{definition}
The following proposition builds on \cite[][Thm. 3.20, Prop. 3.23]{Claudon_Hoering_Lin_fundamental_group}, \cite[][Prop. 2.5]{Claudon_tori_fiber_bundles} and \cite[][Prop. 1.33]{Nakayama_elliptic-fibrations}.
\begin{proposition}\label{fundamental_properties_special_torus_bundles}\cite[][Prop. 2.7]{Pietig_1-forms_Kaehler_threefolds}
    Let $Y$ be a compact connected K\"ahler manifold, let $A=\C^g/\Lambda$ be a torus, and let $\eta\in H^1(Y,\OO_Y^{\oplus g}/\Lambda)$.
    \begin{enumerate}[label=(\roman*)]
        \item\label{torus_bundle_kaehler} The complex manifold $(Y\times A)^\eta$ is K\"ahler if and only if $c(\eta)\in H^2(Y,\Lambda)$ is torsion.
        \item\label{1-forms_on_torus_bundles} Suppose that $(Y\times A)^\eta$ is K\"ahler. Let $\pi\colon (Y\times A)^\eta\rightarrow Y$ be the projection. Then for every $y\in Y$ the restriction map
        \begin{equation*}
            H^0((Y\times A)^\eta,\Omega^1_{(Y\times A)^\eta})\longrightarrow H^0(\pi^{-1}(y),\Omega^1_{\pi^{-1}(y)})\cong H^0(A,\Omega^1_A)
        \end{equation*}
        fits into a short exact sequence
        \begin{equation*}
            0\longrightarrow H^0(Y,\Omega^1_Y)\longrightarrow H^0((Y\times A)^\eta,\Omega^1_{(Y\times A)^\eta})\longrightarrow H^0(A,\Omega^1_A)\longrightarrow 0.
        \end{equation*}
        In particular, any holomorphic 1-form $\omega\in H^0((Y\times A)^\eta,\Omega^1_{(Y\times A)^\eta})$ that does not come from $Y$ has no zeros.
        \item\label{torus_bundle_multisection} The following is equivalent:
        \begin{enumerate}[label=(\alph*)]
            \item\label{torus_bundle_multisection_a} The torus fiber bundle $(Y\times A)^\eta\rightarrow Y$ admits a multisection.
            \item\label{torus_bundle_multisection_b} There is a finite \'etale cover $Y'\rightarrow Y$ such that $(Y\times A)^\eta\times_YY'\cong Y'\times A$.
            \item\label{torus_bundle_multisection_c} The cohomology class $\eta\in H^1(Y,\OO_Y^{\oplus g}/\Lambda)$ is torsion.
        \end{enumerate}
    \end{enumerate}
\end{proposition}
\begin{lemma}\label{special_local_trivialization}
    Let $Y$ be a compact connected K\"ahler manifold, let $A=\C^g/\Lambda$ be a torus, and let $\eta\in H^1(Y,\OO_Y^{\oplus g}/\Lambda)$ such that $(Y\times A)^\eta$ is K\"ahler.Then there is an open cover $\{U_i\}_{i\in I}$ of $Y$ which trivializes $\pi^\eta\colon (Y\times A)^\eta\rightarrow Y$, and isomorphisms
    \begin{equation*}
        \varphi_i\colon (\pi^\eta)^{-1}(U_i)\stackrel{\sim}{\longrightarrow}U_i\times A
    \end{equation*}
    over $U_i$ for each $i\in I$, such that the local patches $U_i\times A$, $U_j\times A$ are glued along
    \begin{equation*}
        (U_i\cap U_j)\times A\stackrel{\id\times\mathrm{tr}(a_{ij})}{\longrightarrow} (U_i\cap U_j)\times A,
    \end{equation*}
    where $a_{ij}\in A$.
\end{lemma}
\begin{proof}
    Consider the commutative diagram of short exact sequences
    \[\begin{tikzcd}
	   0 & {\underline{\Lambda}} & {\underline{\mathbb{C}^g}} & {\underline{\mathbb{C}^g/\Lambda}} & 0 \\
	   0 & {\underline{\Lambda}} & {\mathcal{O}_Y^{\oplus g}} & {\mathcal{O}_Y^{\oplus g}/\underline{\Lambda}} & {0.}
	   \arrow[from=1-1, to=1-2]
	   \arrow[from=1-2, to=1-3]
	   \arrow[from=1-2, to=2-2]
	   \arrow[from=1-3, to=1-4]
	   \arrow[from=1-3, to=2-3]
	   \arrow[from=1-4, to=1-5]
	   \arrow[from=1-4, to=2-4]
	   \arrow[from=2-1, to=2-2]
	   \arrow[from=2-2, to=2-3]
	   \arrow[from=2-3, to=2-4]
	   \arrow[from=2-4, to=2-5]
    \end{tikzcd}\]
    Applying cohomology to it induces the following diagram with exact rows
    \[\begin{tikzcd}
	   {H^1(Y,\Lambda)} & {H^1(Y,\mathbb{C}^g)} & {H^1(Y,\mathbb{C}^g/\Lambda)} & {H^2(Y,\Lambda)} & {H^2(Y,\mathbb{C}^g)} \\
	   {H^1(Y,\Lambda)} & {H^1(Y,\mathcal{O}_Y^{\oplus g})} & {H^1(Y,\mathcal{O}_Y^{\oplus g}/\Lambda)} & {H^2(Y,\Lambda)} & {H^2(Y,\mathcal{O}_Y^{\oplus g}).}
	   \arrow[from=1-1, to=1-2]
	   \arrow["\id", from=1-1, to=2-1]
	   \arrow["\exp'", from=1-2, to=1-3]
	   \arrow[two heads, from=1-2, to=2-2]
	   \arrow["c'", from=1-3, to=1-4]
	   \arrow[from=1-3, to=2-3]
	   \arrow[from=1-4, to=1-5]
	   \arrow["\id", from=1-4, to=2-4]
	   \arrow[from=1-5, to=2-5]
	   \arrow[from=2-1, to=2-2]
	   \arrow["\exp", from=2-2, to=2-3]
	   \arrow["c", from=2-3, to=2-4]
	   \arrow[from=2-4, to=2-5]
    \end{tikzcd}\]
    Note that the map $H^1(Y,\C^g)\twoheadrightarrow H^1(Y,\OO_Y^{\oplus g})$ is surjective by the Hodge decomposition. It suffices to show that $\eta$ lifts to a class in $H^1(Y,\C^g/\Lambda)$. Note that since $H^2(Y,\C)\rightarrow H^2(Y,\OO_Y)$ is not injective, we cannot apply the five lemma. Instead we argue by hand.\par
    As $(Y\times A)^\eta$ is K\"ahler, $c(\eta)$ is torsion by \Cref{fundamental_properties_special_torus_bundles}. Since $H^2(Y,\C^g)$ is torsion free, there is a class $\alpha\in H^1(Y,\C^g/\Lambda)$ such that $c'(\alpha)=c(\eta)$.\par
    Denote the image of $\alpha$ in $H^1(Y,\OO_Y^{\oplus g}/\Lambda)$ by $\theta$. Then $c(\eta-\theta)=0$. In particular, there is a class $\beta\in H^1(Y,\OO_Y^{\oplus g})$ such that $\exp(\beta)=\eta-\theta$.\par
    Lift $\beta$ to a class $\gamma\in H^1(Y,\C^g)$. Then $\exp'(\gamma)+\alpha\in H^1(Y,\C^g/\Lambda)$ is a lift of $\eta\in H^1(Y,\OO_Y^{\oplus g}/\Lambda)$.
\end{proof}
\begin{lemma}\label{lemma_extension_of_etale_cover_of_elliptic_curves}
    Let $Y$ be a compact connected K\"ahler manifold and let $\varphi\colon A'=\C^g/\Lambda'\rightarrow A=\C^g/\Lambda$ be a finite \'etale cover of tori. Suppose $\eta\in H^1(Y,\OO_Y^{\oplus g}/\Lambda)$ is a cohomology class with trivial Chern class. Then there exists a non-unique cohomology class $\eta'\in H^1(Y,\OO_Y^{\oplus g}/\Lambda')$ that lifts $\eta$, i.e. the map
    \begin{equation*}
        H^1(Y,\OO_Y^{\oplus g}/\Lambda')\longrightarrow H^1(Y,\OO_Y^{\oplus g}/\Lambda)
    \end{equation*}
    obtained by applying cohomology to the short exact sequence $0\rightarrow\Lambda/\Lambda'\rightarrow\OO_Y^{\oplus g}/\Lambda'\rightarrow\OO_Y^{\oplus g}/\Lambda\rightarrow 0$ sends $\eta'$ to $\eta$. In particular, we can extend $\varphi\colon A'\rightarrow A$ to a finite \'etale cover
    \begin{equation*}
        (Y\times A')^{\eta'}\longrightarrow(Y\times A)^{\eta}.
    \end{equation*}
\end{lemma}
\begin{proof}
    This follows from the commutative diagram
    \[\begin{tikzcd}
	   {H^1(Y,\mathcal{O}_Y^{\oplus g})} & {H^1(Y,\mathcal{O}_Y^{\oplus g}/\Lambda')} & {H^2(Y,\Lambda')} \\
	   {H^1(Y,\mathcal{O}_Y^{\oplus g})} & {H^1(Y,\mathcal{O}_Y^{\oplus g}/\Lambda)} & {H^2(Y,\Lambda).}
	   \arrow["\mathrm{exp}", from=1-1, to=1-2]
	   \arrow[equals, from=1-1, to=2-1]
	   \arrow["c", from=1-2, to=1-3]
	   \arrow[from=1-2, to=2-2]
	   \arrow[from=1-3, to=2-3]
	   \arrow["\mathrm{exp}", from=2-1, to=2-2]
	   \arrow["c", from=2-2, to=2-3]
    \end{tikzcd}\]
\end{proof}
\begin{definition}\label{definition_extension_of_etale_cover_of_elliptic_curves}
    In the situation of \Cref{lemma_extension_of_etale_cover_of_elliptic_curves}, we call $(Y\times A')^{\eta'}\rightarrow(Y\times A)^{\eta}$ an \textit{extension} of $\varphi\colon A'\rightarrow A$.
\end{definition}
\subsection{Numerically \textit{K}-trivial varieties}
\begin{definition}\cite[][Def. 3.1]{Bakker_Guenancia_Lehn}
    A normal compact complex analytic variety $X$ with rational singularities is called \textit{numerically $K$-trivial}, if there is a positive integer $m$ such that $\mathcal{K}_X^m$ is invertible, and if $c_1(\mathcal{K}_X)=0$ in $H^2(X,\Q)$.
\end{definition}
\begin{theorem}\label{decomposition_numerically_K-trivial_Kaehler_variety}\cite[][Thm. A]{Bakker_Guenancia_Lehn}
    Let $X$ be a numerically $K$-trivial compact K\"ahler variety with log terminal singularities. Then there is a finite quasi-\'etale morphism $X'\rightarrow X$ such that
    \begin{equation*}
        X'\cong T\times\prod\limits_iY_i\times\prod\limits_jZ_j,
    \end{equation*}
    where $T$ is a torus, each $Y_i$ is an irreducible Calabi-Yau variety, and each $Z_j$ is an irreducible holomorphic symplectic variety.
\end{theorem}
The following consequence is well-known. However, since we were unable to find a reference, we give a short proof.
\begin{corollary}\label{isogeny_from_Automorphisms_to_Albanese_numerically_K-trivial}
    Let $X$ be a numerically $K$-trivial compact K\"ahler variety with log terminal singularities. Then the following assertions hold true.
    \begin{enumerate}[label=(\roman*)]
        \item The natural pairing
        \begin{equation*}
            H^0(X,\mathcal{T}_X)\times H^0(X,\Omega_X^{[1]})\longrightarrow\C
        \end{equation*}
        is non-degenerate.
        \item The natural map
        \begin{equation*}
            \mathrm{Aut}^\circ(X)\longrightarrow\mathrm{Alb}(X)
        \end{equation*}
        is an isogeny.
    \end{enumerate}
\end{corollary}
\begin{proof}
    To prove the first assertion, let $\pi\colon X'\rightarrow X$ be the finite quasi-\'etale cover from \Cref{decomposition_numerically_K-trivial_Kaehler_variety}, and let $p_1\colon X'\rightarrow T$ be the projection onto the first factor. First note, that since irreducible Calabi-Yau varieties and irreducible holomorphic symplectic varieties do not admit non-trivial reflexive 1-forms or holomorphic vector fields, there are isomorphisms
    \begin{align*}
        H^0(X',\Omega_{X'}^{[1]})&\cong H^0(T,\Omega_T^1),\\
        H^0(X',\mathcal{T}_{X'})&\cong H^0(T,\mathcal{T}_T).
    \end{align*}
    Since the pairing $H^0(T,\mathcal{T}_T)\times H^0(T,\Omega_T^1)\rightarrow\C$ is non-degenerate, it follows that the pairing
    \begin{equation*}
        H^0(X',\mathcal{T}_{X'})\times H^0(X',\Omega_{X'}^{[1]})\longrightarrow\C
    \end{equation*}
    is non-degenerate.\par
    By \cite[][Thm. 1.11]{Kebekus_Schnell_pullback_reflexive_forms}, a reflexive 1-form $\omega$ on $X$ pulls back to a reflexive 1-form $\pi^*\omega$ on $X'$, and as $\pi$ is generically smooth, we obtain an injection
    \begin{equation*}
        \pi^*\colon H^0(X,\Omega_X^{[1]})\longhookrightarrow H^0(X',\Omega_{X'}^1).
    \end{equation*}
    Similarly, we can push holomorphic vector field forward along $\pi$:
    \begin{equation*}
        \pi_*\colon H^0(X',\mathcal{T}_{X'})\longrightarrow H^0(X,\mathcal{T}_X\otimes\pi_*\OO_{X'})\stackrel{\mathrm{trace}}{\longrightarrow}H^0(X,\mathcal{T}_X).
    \end{equation*}
    Let $Z\subset X'$ be a closed subset of codimension $\geq 2$ such that $X'\setminus Z\rightarrow X$ is \'etale. In particular, $\pi^*\Omega_X^{[1]}\vert_{X'\setminus Z}\cong\Omega_{X'}^{[1]}\vert_{X'\setminus Z}$, and $\pi^*\mathcal{T}_X\vert_{X'\setminus Z}\cong\mathcal{T}_{X'}\vert_{X'\setminus Z}$. Consider the map
    \begin{align*}
        H^0(X',\Omega_{X'}^{[1]})&\longrightarrow H^0(X'\setminus Z,\Omega_{X'}^{[1]}\vert_{X'\setminus Z})\cong H^0(X',\setminus Z,\pi^*\Omega_{X}^{[1]}\vert_{X'\setminus Z})\\
        &\longrightarrow H^0(X\setminus\pi(Z),\Omega_X^{[1]}\vert_{X\setminus\pi(Z)}\otimes\pi_*\OO_{X'}\vert_{X\setminus\pi(Z)})\\
        &\stackrel{\mathrm{trace}}{\longrightarrow}H^0(X\setminus\pi(Z),\Omega_X^{[1]}\vert_{X\setminus\pi(Z)}),
    \end{align*}
    resp.
    \begin{equation*}
        H^0(X,\mathcal{T}_X)\stackrel{\pi^*}{\longrightarrow} H^0(X',\pi^*\mathcal{T}_X)\longrightarrow H^0(X'\setminus Z,\pi^*\mathcal{T}_X\vert_{X'\setminus Z})\cong H^0(X'\setminus Z,\mathcal{T}_{X'}\vert_{X'\setminus Z}).
    \end{equation*}
    Since $\Omega_X^{[1]}$, resp. $\mathcal{T}_{X'}:=(\Omega_{X'}^1)^\vee$, is reflexive, and as $\pi(Z)\subset X$, resp. $Z\subset X'$ has codimension $\geq 2$, the restriction map
    \begin{equation*}
        H^0(X,\Omega_X^{[1]})\stackrel{\sim}{\longrightarrow}H^0(X\setminus\pi(Z),\Omega_X^{[1]}\vert_{X\setminus\pi(Z)}),
    \end{equation*}
    resp.
    \begin{equation*}
        H^0(X',\mathcal{T}_{X'})\stackrel{\sim}{\longrightarrow} H^0(X'\setminus Z,\mathcal{T}_{X'}\vert_{X'\setminus Z})
    \end{equation*}
    is an isomorphism. We hence obtain well-defined maps
    \begin{align*}
        \pi_*\colon H^0(X',\Omega_{X'}^{[1]})&\longrightarrow H^0(X,\Omega_X^{[1]})\\
        \pi^*\colon H^0(X,\mathcal{T}_X)&\longrightarrow H^0(X',\mathcal{T}_{X'}).
    \end{align*}
    Note that $\pi_*\circ\pi^*$ equals multiplication by $\deg\pi$.\par
    Let $\omega$ be a non-trivial reflexive 1-form on $X$. Since the pairing on $X'$ is non-degenerate, there exists a holomorphic vector field $v$ on $X'$ such that $\langle v,\pi^*\omega\rangle\neq 0$. Hence,
    \begin{equation*}
        \langle \pi_*v,\omega\rangle=\deg\pi\cdot\langle v,\pi^*\omega\rangle\neq 0.
    \end{equation*}
    Now, let $v$ be a non-trivial holomorphic vector field on $X$. Since the pairing on $X'$ is non-degenerate, there exists a reflexive 1-form $\omega$ on $X'$ such that $\langle\pi^*v,\omega\rangle\neq 0$. Hence,
    \begin{equation*}
        \langle v,\pi_*\omega\rangle=\deg\pi\cdot\langle\pi^*v,\omega\rangle\neq 0.
    \end{equation*}
    This proves the first assertion.\par
    To prove the second assertion, note that both groups are compact complex Lie groups with tangent spaces
    \begin{align*}
        T_{\id}\mathrm{Aut}^\circ(X)&\cong H^0(X,\mathcal{T}_X),\\
        T_{\id}\mathrm{Alb}(X)&\cong H^0(X,\Omega_X^{[1]})^\vee.
    \end{align*}
    By the first assertion, the natural map on tangent spaces
    \begin{equation*}
        T_{\id}\mathrm{Aut}^\circ(X)\longrightarrow T_{\id}\mathrm{Alb}(X)
    \end{equation*}
    is an isomorphism. This suffices to prove the second assertion.
\end{proof}
\section{Popa--Schnell for K\"ahler manifolds}
\begin{theorem}\label{Popa_Schnell_thm}\cite[][Thm. 2.1]{Popa_Schnell_Kodaira_dimension_zeros_1-forms}
    Let $X$ be a smooth projective variety, and let $f\colon X\rightarrow A$ be a morphism to an abelian variety. If there is an ample line bundle $\mathcal{L}$ on $A$ and a positive integer $d$ such that
    \begin{equation*}
        H^0(X,\mathcal{K}_X^d\otimes f^*\mathcal{L}^\vee)\neq 0,
    \end{equation*}
    then the zero locus $Z(f^*\omega)\neq\varnothing$ is non-empty for all $\omega\in H^0(A,\Omega_A^1)$.
\end{theorem}
We show that the result holds for compact connected K\"ahler manifolds.
\begin{theorem}\label{Popa_Schnell_Kaehler_case}
    Let $X$ be a compact connected K\"ahler manifold, and let $f\colon X\rightarrow A$ be a morphism to an abelian variety. If there is an ample line bundle $\mathcal{L}$ on $A$ and a positive integer $d$ such that
    \begin{equation*}
        H^0(X,\mathcal{K}_X^d\otimes f^*\mathcal{L}^\vee)\neq 0,
    \end{equation*}
    then the zero locus $Z(f^*\omega)\neq\varnothing$ is non-empty for all $\omega\in H^0(A,\Omega_A^1)$.
\end{theorem}
We claim that the proof by Villadsen \cite{Villadsen_proof_Popa-Schnell} using Higgs bundles holds true in the K\"ahler setting. For the convenience of the reader, we summarize the argument here.\par
Recall that a Higgs line bundle $(\mathcal{L},\omega)$ on $X$ consists of a line bundle $\mathcal{L}$ on $X$ together with a holomorphic 1-form $\omega\in H^0(X,\Omega_X^1)$. A Higgs line bundle $(\mathcal{L},\omega)$ hence induces the \textit{Dolbeault complex}
\begin{equation*}
    \mathrm{Dol}(\mathcal{L},\omega)_\bullet:=\left[\mathcal{L}\stackrel{\omega\wedge}{\longrightarrow}\mathcal{L}\otimes\Omega_X^1\stackrel{\omega\wedge}{\longrightarrow}\cdots\stackrel{\omega\wedge}{\longrightarrow}\mathcal{L}\otimes\Omega_X^n\right].
\end{equation*}
The \textit{Dolbeault cohomology} of the Higgs line bundle $(\mathcal{L},\omega)$ is defined as the hypercohomology
\begin{equation*}
    H^k_{\mathrm{Dol}}(X,\mathcal{L},\omega):=H^k(R\Gamma(X,\mathrm{Dol}(\mathcal{L},\omega)_\bullet)).
\end{equation*}
Denote by
\begin{equation*}
    \pic^\tau(X):=\ker(\pic(X)\stackrel{c_1}{\longrightarrow}H^2(X,\Z)/\mathrm{torsion})\subset\pic(X)
\end{equation*}
the subgroup of the Picard group consisting of line bundles $\mathcal{L}$ such that $c_1(\mathcal{L})$ is torsion. The \textit{Dolbeault moduli space} $M_{\mathrm{Dol}}(X)$ of Higgs line bundles is thus isomorphic to
\begin{equation*}
    M_{\mathrm{Dol}}(X)\cong\pic^\tau(X)\times H^0(X,\Omega_X^1).
\end{equation*}
For integers $k$, $m$ define the cohomology jump loci
\begin{equation*}
    \Sigma^k_m(X):=\{(\mathcal{L},\omega)\in M_{\mathrm{Dol}}(X)\mid\dim H_{\mathrm{Dol}}^k(X,\mathcal{L},\omega)\geq m\}.
\end{equation*}
Consider a morphism $g\colon X\rightarrow T$ to a torus. We obtain a pullback map
\begin{equation*}
    g_{\mathrm{Dol}}\colon M_{\mathrm{Dol}}(T)\longrightarrow M_{\mathrm{Dol}}(X).
\end{equation*}
A \textit{linear subvariety} $L\subset M_{\mathrm{Dol}}(X)$ is a translate of $\im(g_{\mathrm{Dol}})$, where $g\colon X\rightarrow T$ is a morphism to a torus.
\begin{remark}\label{proper_linear_subvariety}
    The following observation will be useful in the proof of \Cref{Popa_Schnell_Kaehler_case}. Let $L\subset M_{\mathrm{Dol}}(X)$ be a linear subvariety. Then $L$ is a connected component of $M_{\mathrm{Dol}}(X)$ if and only if the forgetful morphism
    \begin{equation*}
        M_{\mathrm{Dol}}(X)\longrightarrow\pic^\tau(X)
    \end{equation*}
    sends $L$ surjectively to a connected component of $\pic^\tau(X)$. Indeed, Write $L=(\mathcal{L},\omega)+\im(g_{\mathrm{Dol}})$, where $g\colon X\rightarrow T$ is a morphism to a torus. Write
    \begin{equation*}
        \im(g_{\mathrm{Dol}})=g^*\pic^0(T)\times g^*H^0(T,\Omega_T^1).
    \end{equation*}
    Then $L$ is a connected component of $M_{\mathrm{Dol}}(X)$ if and only if $g^*\pic^0(T)=\pic^0(X)$ and $g^*H^0(T,\Omega_T^1)=H^0(X,\Omega_X^1)$. Hodge theoretic arguments show that if $g^*\pic^0(T)=\pic^0(X)$ then $g^*H^0(T,\Omega_T^1)=H^0(X,\Omega_X^1)$ is automatic. Thus, $L$ is a connected component of $M_{\mathrm{Dol}}(X)$ if and only if $g^*\pic^0(T)=\pic^0(X)$, and the latter assertion is equivalent to the condition that $L$ is mapped surjectively to a connected component of $\pic^\tau(X)$ by the forgetful morphism
    \begin{equation*}
        M_{\mathrm{Dol}}(X)\longrightarrow\pic^\tau(X).
    \end{equation*}
\end{remark}
The following theorem was proved by Simpson for smooth complex projective varieties. For a proof in the K\"ahler case, see \cite[][Cor. 18.6]{Schnell_generic_vanishing}.
\begin{theorem}\label{structure_cohomology_jump_loci}\cite[][Thm. 4.2]{Simpson_moduli_rank_1_local_systems}
    Every cohomology jump locus $\Sigma^k_m(X)$ is a finite union of linear subvarieties of $M_{\mathrm{Dol}}(X)$.
\end{theorem}
\begin{lemma}\label{Villadsen3}\cite[][Lem. 11.1]{Popa_Schnell_Kodaira_dimension_zeros_1-forms}
    Let $f\colon X\rightarrow A$ be a morphism from a compact connected K\"ahler manifold $X$ to an abelian variety $A$. Suppose there is an ample line bundle $\mathcal{L}$ on $A$ such that
    \begin{equation*}
        H^0(X,\mathcal{K}_X^d\otimes f^*\mathcal{L}^\vee)\neq 0
    \end{equation*}
    for some positive integer $d$. Then there exists an isogeny $\varphi\colon A'\rightarrow A$, and an ample line bundle $\mathcal{L}'$ on $A'$ such that
    \begin{equation*}
        H^0(X',(\mathcal{K}_{X'}\otimes f'^*\mathcal{L}'^\vee)^d)\neq 0,
    \end{equation*}
    where $f'\colon X':=X\times_AA'\rightarrow A'$.
\end{lemma}
\begin{proof}
    One uses the multiplication by $2d$-map on $A$, which does not require projectivity of $X$.
\end{proof}
\begin{construction}\label{construction_Villadsen_2}
    In the situation and notation from \Cref{Villadsen3}, denote the finite \'etale cover by $\tau\colon X'\rightarrow X$. Note that to show $Z(f^*\omega)\neq\varnothing$ for all $\omega\in H^0(A,\Omega_A^1)$, it suffices to show $Z(f'^*\omega)\neq\varnothing$ for all $\omega\in H^0(A',\Omega_{A'}^1)$. We can thus replace $f\colon X\rightarrow A$ by $f'\colon X'\rightarrow A'$, and $\mathcal{L}$ by $\mathcal{L}'$ to assume
    \begin{equation*}
        H^0(X,(\mathcal{K}_{X}\otimes f^*\mathcal{L}^\vee)^d)\neq 0.
    \end{equation*}
    Set $\mathcal{B}:=\mathcal{K}_X\otimes f^*\mathcal{L}^\vee$, and denote by $s$ a non-zero global section of $\mathcal{B}^d$.\par
    Let $\pi\colon X_d\rightarrow X$ be the $d$-fold cyclic cover ramified along $Z(s)$, and let $Y\rightarrow X_d$ be a resolution of singularities.
    \[\begin{tikzcd}
	   Y & {X_d} & X \\
	   && A
	   \arrow[from=1-1, to=1-2]
	   \arrow["\varphi", curve={height=-12pt}, from=1-1, to=1-3]
	   \arrow["h"', from=1-1, to=2-3]
	   \arrow["\pi", from=1-2, to=1-3]
	   \arrow["f", from=1-3, to=2-3]
    \end{tikzcd}\]
    By construction,
    \begin{equation*}
        X_d=\mathrm{Spec}\left(\bigoplus\limits_{i=0}^{d-1}\mathcal{B}^{-i}\right).
    \end{equation*}
    Therefore,
    \begin{equation*}
        \pi_*\pi^*\mathcal{B}=\bigoplus\limits_{i=-1}^{d-2}\mathcal{B}^{-i},
    \end{equation*}
    which has a global section from the $i=0$ summand. The corresponding global section of $\pi^*B$ gives rise to a map
    \begin{equation*}
        \varphi^*\mathcal{B}^\vee\longrightarrow\OO_Y,
    \end{equation*}
    which is an isomorphism away from $Z(s)$. This induces injective maps
    \begin{equation*}
        \varphi^*(\mathcal{B}^\vee\otimes\Omega_X^k)\longhookrightarrow\Omega_Y^k.
    \end{equation*}
    Since $\OO_X\rightarrow\varphi_*\OO_Y$ is injective, the corresponding maps
    \begin{equation}\label{maps_Villadsen4}
        \mathcal{B}^\vee\otimes\Omega_X^k\longrightarrow\varphi_*\Omega_Y^k
    \end{equation}
    are also injective.\par
    Denote by $V:=H^0(A,\Omega_A^1)$ the space of holomorphic 1-forms on $A$, and consider the incidence variety
    \begin{equation*}
        Z_f:=\{(x,\omega)\in X\times V\mid (f^*\omega)_x=0\}\subset X\times V
    \end{equation*}
    and denote by $p\colon X\times V\rightarrow X$, $q\colon X\times V\rightarrow V$ the respective projections. Note that $q$ sends $Z_f$ to the set of holomorphic 1-forms, which have a zero. To prove \Cref{Popa_Schnell_Kaehler_case}, it hence suffices to show that $q(Z_f)=V$.\par
    Define $S:=\mathrm{Sym}^\bullet V^\vee$, and denote by $S[i]$ the graded $S$-module $S$ itself with grading $S[i]_j=\mathrm{Sym}^{i+j}V^\vee$. Set $n:=\dim X$ and $g:=\dim A$. Let $\mathcal{C}_{X,\bullet}$ be the complex of graded $\OO_X\otimes S$-modules
    \begin{equation*}
        \mathcal{C}_{X,\bullet}:=\left[\OO_X\otimes S[-g]\longrightarrow\Omega_X^1\otimes S[-g+1]\longrightarrow\cdots\longrightarrow\Omega_X^n\otimes S[-g+n]\right],
    \end{equation*}
    in degrees $-n$ to $0$. The differentials are induced by
    \begin{align*}
        \OO_X\otimes V&\longrightarrow\Omega_X^1\\
        \varphi\otimes\omega&\longmapsto\varphi f^*\omega.
    \end{align*}
    For $\alpha\in\pic^0(A)$, define $\mathcal{C}_{X,\bullet}^\alpha:=\mathcal{C}_{X,\bullet}\otimes f^*\alpha$.\par
    Similarly, define the complex of graded $\OO_Y\otimes S$-modules $\mathcal{C}_{Y,\bullet}$ as
    \begin{equation*}
        \mathcal{C}_{Y,\bullet}:=\left[\OO_Y\otimes S[-g]\longrightarrow\Omega_Y^1\otimes S[-g+1]\longrightarrow\cdots\longrightarrow\Omega_Y^n\otimes S[-g+n]\right],
    \end{equation*}
    in degrees $-n$ to $0$. For $\alpha\in\pic^0(A)$, define $\mathcal{C}_{Y,\bullet}^\alpha:=\mathcal{C}_{Y,\bullet}\otimes h^*\alpha$.
\end{construction}
\begin{lemma}\label{Villadsen4}\cite[][Lem. 2.4]{Villadsen_proof_Popa-Schnell}
    The maps in \cref{maps_Villadsen4} from \Cref{construction_Villadsen_2} induce a map of complexes of graded $\OO_X\otimes\mathcal{S}$-modules
    \begin{equation*}
        \mathcal{B}^\vee\otimes\mathcal{C}_{X,\bullet}^\alpha\longrightarrow R\varphi_*\mathcal{C}_{Y,\bullet}^\alpha.
    \end{equation*}
\end{lemma}
\begin{proof}
    The maps $\varphi^*(\mathcal{B}^\vee\otimes\Omega_X^k)\rightarrow\Omega_Y^k$ commute with the differentials on $Y$. This gives rise to
    \begin{equation*}
        \varphi^*(\mathcal{B}^\vee\otimes\mathcal{C}_{X,\bullet}^\alpha)\longrightarrow\mathcal{C}_{Y,\bullet}^\alpha.
    \end{equation*}
    Pushing forward along $\varphi$ thus yields
    \begin{equation*}
        R\varphi_*\varphi^*(\mathcal{B}^\vee\otimes\mathcal{C}_{X,\bullet}^\alpha)\longrightarrow R\varphi_*\mathcal{C}_{Y,\bullet}^\alpha.
    \end{equation*}
    The projection formula implies
    \begin{equation*}
        R\varphi_*\varphi^*(\mathcal{B}^\vee\otimes\mathcal{C}_{X,\bullet}^\alpha)\cong (\mathcal{B}^\vee\otimes\mathcal{C}_{X,\bullet}^\alpha)\otimes^LR\varphi_*\OO_Y.
    \end{equation*}
    The canonical map $\OO_X\rightarrow R\varphi_*\OO_Y$ thus gives rise to
    \begin{equation*}
        \mathcal{B}^\vee\otimes\mathcal{C}_{X,\bullet}^\alpha\longrightarrow(\mathcal{B}^\vee\otimes\mathcal{C}_{X,\bullet}^\alpha)\otimes^LR\varphi_*\OO_Y\longrightarrow R\varphi_*\mathcal{C}_{Y,\bullet}^\alpha.
    \end{equation*}
\end{proof}
Define the complexes of graded $S$-modules
\begin{align*}
    M_\bullet^\alpha&:=H^0(R\Gamma(X,\mathcal{B}^\vee\otimes\mathcal{C}_{X,\bullet}^\alpha)),\\
    N_\bullet^\alpha&:=H^0(R\Gamma(X,R\varphi_*\mathcal{C}_{Y,\bullet}^\alpha))=H^0(R\Gamma(Y,\mathcal{C}_{Y,\bullet}^\alpha)),
\end{align*}
where $H^0$ is the degree zero cohomology of the complexes of graded $S$-modules $R\Gamma(X,\mathcal{B}^\vee\otimes\mathcal{C}_{X,\bullet}^\alpha)$, resp. $R\Gamma(X,R\varphi_*\mathcal{C}_{Y,\bullet}^\alpha)$. In particular, $M_\bullet^\alpha$, and $N_\bullet^\alpha$ are graded $S$-modules, and the map from \Cref{Villadsen4} induces a map of graded $S$-modules
\begin{equation*}
    M_\bullet^\alpha\longrightarrow N_\bullet^\alpha.
\end{equation*}
\begin{lemma}\label{lemma_support_Z_f_image_non_zero_Villadsen}\cite[cf.][Lem. 14.1]{Popa_Schnell_Kodaira_dimension_zeros_1-forms}\cite[cf.][Proof of Thm. 1.1]{Villadsen_proof_Popa-Schnell}
    Define the graded $S$-module $L_\bullet^\alpha:=\im(M_\bullet^\alpha\rightarrow N_\bullet^\alpha)$. Then the following hold true.
    \begin{enumerate}[label=(\roman*)]
        \item The support of $L_\bullet^\alpha$ is contained in $q(Z_f)$.
        \item $L_\bullet^\alpha$ is non-zero.
    \end{enumerate}
\end{lemma}
\begin{proof}
    Denote for $v\in V$ its maximal ideal in $S$ by $\mathfrak{m}_v\subset S$. To prove the first assertion, note that there is a quasi-isomorphism
    \begin{equation*}
        R\Gamma(X,\mathcal{B}^\vee\otimes\mathcal{C}_{X,\bullet}^\alpha)\otimes^L_SS/\mathfrak{m}_v\cong R\Gamma(X,\mathcal{B}^\vee\otimes\mathrm{Dol}(f^*\alpha,f^*v)_\bullet)
    \end{equation*}
    for all $v\in V$. If $f^*v$ has no zero, then $\mathrm{Dol}(f^*\alpha,f^*v)$ is exact, and hence also $\mathcal{B}^\vee\otimes\mathrm{Dol}(f^*\alpha,f^*v)_\bullet$ is exact. Thus, $R\Gamma(X,\mathcal{B}^\vee\otimes\mathcal{C}_{X,\bullet}^\alpha)\otimes^L_SS/\mathfrak{m}_v=0$ for all $v\in V$ such that $f^*v$ has no zero. This shows that $M_\bullet^\alpha$ is supported on $q(Z_f)$. Since $L_\bullet^\alpha$ is the image of $M_\bullet^\alpha$ under the map $M_\bullet^\alpha\rightarrow N_\bullet^\alpha$, also $L_\bullet^\alpha$ is supported on $q(Z_f)$.\par
    For the second assertion, let $k=g-n$. Then $\mathcal{C}_{X,k}=\mathcal{K}_X$, and $\mathcal{C}_{Y,k}=\mathcal{K}_Y$. Since $\varphi$ is generically finite, $R\varphi_*\mathcal{K}_Y\cong\varphi_*\mathcal{K}_Y$ by \cite[][Thm. 1]{Takegoshi_relative_vanishing}. Thus, the map $\mathcal{B}^\vee\otimes\mathcal{C}_{X,k}^\alpha\rightarrow R\varphi_*\mathcal{C}_{Y,k}^\alpha$ from \Cref{Villadsen4} is just the map $\mathcal{B}^\vee\otimes\mathcal{K}_X\otimes f^*\alpha\rightarrow\varphi_*\mathcal{K}_Y\otimes f^*\alpha$.\par
    For the graded $S$-module $L_\bullet^\alpha$, the above identification implies that $L_k^\alpha\cong H^0(X,f^*(\mathcal{L}\otimes\alpha))$. If $H^0(X,f^*(\mathcal{L}\otimes\alpha))=0$ for general $\alpha$, then $H^0(A,f_*\OO_X\otimes\mathcal{L}\otimes\alpha)=0$ for general $\alpha$. This implies that $f$ maps $X$ to a general translate of a hyperplane section in $A$, a contradiction. Thus, $L_\bullet^\alpha$ is non-zero.
\end{proof}
\begin{lemma}\label{image_torsion_fee_Villadsen}\cite[cf.][Prop. 2.2]{Villadsen_proof_Popa-Schnell}
    $N_\bullet^\alpha$ is torsion free for general $\alpha$.
\end{lemma}
\begin{proof}
    Note that for all $v\in V$,
    \begin{equation*}
        R\Gamma(Y,\mathcal{C}_{Y,\bullet}^\alpha)\otimes^L_SS/\mathfrak{m}_v\cong R\Gamma(Y,\mathrm{Dol}(h^*\alpha,h^*v)_\bullet).
    \end{equation*}
    We will show that for general $\alpha$ the right hand side has cohomology ranks independent of $v\in V$. This implies that $N_\bullet^\alpha$ is a locally free module.\par
    To prove the claim, it suffices to show that the function
    \begin{align*}
        \varphi_\alpha\colon V&\longrightarrow\Z\\
        v&\longmapsto\dim H^i_{\mathrm{Dol}}(Y,h^*\alpha,h^*v)
    \end{align*}
    is constant in $v$ for general $\alpha$. Note that $\varphi_\alpha$ factors as
    \begin{align*}
        V&\stackrel{\psi_\alpha}{\longrightarrow} M_{\mathrm{Dol}}(Y)\stackrel{\varphi_\alpha}{\longrightarrow}\Z\\
        v&\longmapsto (h^*\alpha,h^*v).
    \end{align*}
    In particular, the function $\varphi_\alpha$ is non-constant if and only if $\im(\psi_\alpha)\cap\Sigma^k_m(Y)$ is a non-empty proper subset of $\im(\psi_\alpha)$ (the function $\varphi_\alpha$ is constant away from $\Sigma^k_m(Y)$ and jumps up on $\im(\psi_\alpha)\cap\Sigma^k_m(Y)$). By \Cref{structure_cohomology_jump_loci}, each irreducible component of $\Sigma^k_m(Y)$ is a linear subvariety. It hence suffices to show that for each linear subvariety $L\subset M_{\mathrm{Dol}}(Y)$ the image of $\psi_\alpha$ is either contained in $L$ or disjoint from $L$.\par
    Let $L\subset M_{\mathrm{Dol}}(Y)$ be a linear subvariety. Consider $h_{\mathrm{Dol}}\colon M_{\mathrm{Dol}}(A)\rightarrow M_{\mathrm{Dol}}(Y)$. Note that if $h_{\mathrm{Dol}}^{-1}(L)$ is non-empty, then $h_{\mathrm{Dol}}^{-1}(L)$ is a finite union of linear subvarieties in $M_{\mathrm{Dol}}(A)$. If $h_{\mathrm{Dol}}^{-1}(L)=M_{\mathrm{Dol}}(A)$, then $\im(\psi_\alpha)=\{\alpha\}\times h^*V\subset L$ for any $\alpha\in\pic^0(A)$. If $h_{\mathrm{Dol}}^{-1}(L)$ is a proper subset of $M_{\mathrm{Dol}}(A)$, then the forgetful morphism
    \begin{equation*}
        M_{\mathrm{Dol}}(A)\longrightarrow\pic^0(A)
    \end{equation*}
    sends $h_{\mathrm{Dol}}^{-1}(L)$ to a proper subset by \Cref{proper_linear_subvariety}. It hence suffices to take $\alpha$ to be outside the image of $h_{\mathrm{Dol}}^{-1}(L)$ in $\pic^0(A)$ under the forgetful morphism.
\end{proof}
\begin{proof}[Proof of \Cref{Popa_Schnell_Kaehler_case}]
    On the one hand, $L_\bullet^\alpha$ is non-zero by \Cref{lemma_support_Z_f_image_non_zero_Villadsen}. On the other hand, $L_\bullet^\alpha$ is torsion free for general $\alpha$ by \Cref{image_torsion_fee_Villadsen}. Thus, its support is all of $V$. By \Cref{lemma_support_Z_f_image_non_zero_Villadsen} this shows that $q(Z_f)=V$.
\end{proof}
\subsection{Application to the Iitaka fibration}
\begin{proposition}\label{Iitaka_fibers_are_mapped_to_subtorus_of_dim_g'_geq_g}
    Let $X$ be a compact connected K\"ahler manifold with $g$ everywhere linearly independent holomorphic 1-forms $\omega_1,\ldots,\omega_g$. Suppose that $X$ admits a good minimal model $X^{\mathrm{min}}$. Let $f\colon X^{\mathrm{min}}\rightarrow Y$ be the Iitaka fibration. Then the general fiber of $f$ is mapped to a translate of a subtorus $A\subset\mathrm{Alb}(X)$ of $\dim A\geq g$.
\end{proposition}
\begin{proof}
    Choose a resolution of singularities $\mu\colon \tilde{Y}\rightarrow Y$. Note that $X^{\mathrm{min}}\times_Y\tilde{Y}$ contains a unique irreducible component that is bimeromorphic to $X^{\mathrm{min}}$. Denote by $X_{\tilde{Y}}$ the normalization of this component, and denote by $\nu\colon X_{\tilde{Y}}\rightarrow X^{\mathrm{min}}$, and by $f_{\tilde{Y}}\colon X_{\tilde{Y}}\rightarrow\tilde{Y}$ the induced morphism. As $X^{\mathrm{min}}$ has terminal singularities, it has rational singularities \cite[][Thm. 5.22]{Kollar_Mori}. The proof of \cite[][Lem. 8.1]{Kawamata_minimal_models_Kodaira_dim_fiber_spaces} shows that the meromorphic map
    \begin{equation*}
        X^{\mathrm{min}}\longdashrightarrow X\stackrel{\mathrm{alb}_X}{\longrightarrow}\mathrm{Alb}(X)
    \end{equation*}
    extends to a morphism
    \begin{equation*}
        \mathrm{alb}_{X^{\mathrm{min}}}\colon X^{\mathrm{min}}\longrightarrow\mathrm{Alb}(X).
    \end{equation*}
    Moreover, $f$ induces a morphism $\pi\colon\mathrm{Alb}(X)\rightarrow\mathrm{Alb}(\tilde{Y})$. We also choose a compact connected K\"ahler manifold $X'$ together with proper bimeromorphic morphisms $X\stackrel{p}{\leftarrow}X'\stackrel{q}{\rightarrow}X_{\tilde{Y}}$. To summarize, we obtain the commutative diagram
    \[\begin{tikzcd}
	   X & {X'} & {X_{\tilde{Y}}} & {X^{\mathrm{min}}} & {\mathrm{Alb}(X)} \\
	   && {\tilde{Y}} & Y & {\mathrm{Alb}(\tilde{Y}).}
	   \arrow["p"', from=1-2, to=1-1]
	   \arrow["q", from=1-2, to=1-3]
	   \arrow["\nu", from=1-3, to=1-4]
	   \arrow["{f_{\tilde{Y}}}"', from=1-3, to=2-3]
	   \arrow["{\mathrm{alb}_{X^{\mathrm{min}}}}", from=1-4, to=1-5]
	   \arrow["f", from=1-4, to=2-4]
	   \arrow["\pi", from=1-5, to=2-5]
	   \arrow["\mu", from=2-3, to=2-4]
	   \arrow["{\mathrm{alb}_{\tilde{Y}}}"', curve={height=12pt}, from=2-3, to=2-5]
    \end{tikzcd}\]
    Since $Y$ is projective, $\mathrm{Alb}(\tilde{Y})$ admits an ample line bundle. Fix an ample line bundle $\mathcal{L}$ on $\mathrm{Alb}(\tilde{Y})$. After replacing $\mathcal{L}$ by a positive multiple, we can assume that $\mathcal{L}$ is globally generated. In particular, $\mathrm{alb}^*_{\tilde{Y}}\mathcal{L}$ can be represented by an effective Cartier divisor. By the construction of the Iitaka fibration, there is an ample line bundle $\mathcal{M}$ on $Y$ such that $f^*\mathcal{M}=\mathcal{K}_{X^{\mathrm{min}}}^d$ for some positive integer $d$. In particular, $f^*_{\tilde{Y}}\mu^*\mathcal{M}\cong\nu^*\mathcal{K}_{X^{\mathrm{min}}}^d$. Note that $\mu^*\mathcal{M}$ is big. Thus, Kodaira's lemma \cite[][Prop. 2.2.6]{Lazarsfeld1} shows that after replacing $\mathcal{M}$ by a positive multiple, we can assume that
    \begin{equation*}
        H^0(\tilde{Y},\mu^*\mathcal{M}\otimes\mathrm{alb}_{\tilde{Y}}^*\mathcal{L}^\vee)\neq 0.
    \end{equation*}
    Since $X^{\mathrm{min}}$ is a good minimal model, we can assume after replacing $d$ with a suitable positive multiple, that $\mathcal{K}_{X^{\mathrm{min}}}^d$ is globally generated. Moreover, as $X^{\mathrm{min}}$ has terminal singularities, we obtain isomorphisms
    \begin{equation*}
        H^0(X,\mathcal{K}_X^n)\cong H^0(X^{\mathrm{min}},\mathcal{K}_{X^{\mathrm{min}}}^n)
    \end{equation*}
    for all positive integers $n$. Thus, the map
    \begin{equation*}
        H^0(X',p^*\mathcal{K}_X^{d})\otimes\OO_{X'}\longrightarrow p^*\mathcal{K}_X^{d}
    \end{equation*}
    has $q^*\nu^*\mathcal{K}_{X^{\mathrm{min}}}^{d}$ as its image. This induces the map
    \begin{equation*}
        q^*\nu^*\mathcal{K}_{X^{\mathrm{min}}}^{d}\longhookrightarrow p^*\mathcal{K}_X^{d}.
    \end{equation*}
    We thus have the composition
    \begin{align*}
        0&\neq H^0(\tilde{Y},\mu^*\mathcal{M}\otimes\mathrm{alb}_{\tilde{Y}}^*\mathcal{L}^\vee)\\
        &\longhookrightarrow H^0(X_{\tilde{Y}},f^*_{\tilde{Y}}\mathcal{M}\otimes f^*_{\tilde{Y}}\mathrm{alb}_{\tilde{Y}}^*\mathcal{L}^\vee)\cong H^0(X_{\tilde{Y}},\nu^*\mathcal{K}_{X^{\mathrm{min}}}^d\otimes\nu^*\mathrm{alb}_{X^{\mathrm{min}}}^*\pi^*\mathcal{L}^\vee)\\
        &\longhookrightarrow H^0(X',q^*\nu^*\mathcal{K}_{X^{\mathrm{min}}}^d\otimes q^*\nu^*\mathrm{alb}_{X^{\mathrm{min}}}^*\pi^*\mathcal{L}^\vee)\cong H^0(X',q^*\nu^*\mathcal{K}_{X^{\mathrm{min}}}^d\otimes p^*\mathrm{alb}_{X}^*\pi^*\mathcal{L}^\vee)\\
        &\longhookrightarrow H^0(X',p^*\mathcal{K}_X^d\otimes p^*\mathrm{alb}_X^*\pi^*\mathcal{L}^\vee)\\
        & \stackrel{\sim}{\longrightarrow} H^0(X,p_*(p^*\mathcal{K}_X^d\otimes p^*\mathrm{alb}_X^*\pi^*\mathcal{L}^\vee))\cong H^0(X,\mathcal{K}_X^d\otimes\mathrm{alb}_X^*\pi^*\mathcal{L}^\vee).
    \end{align*}
    Here, the arrow in the second line is the pullback along $f_{\tilde{Y}}\colon X_{\tilde{Y}}\rightarrow\tilde{Y}$. It is injective since $f_{\tilde{Y}}$ is dominant. The isomorphism in the second line uses that $f^*_{\tilde{Y}}\mu^*\mathcal{M}\cong\nu^*\mathcal{K}_{X^{\mathrm{min}}}^d$, and that $\mathrm{alb}_{\tilde{Y}}\circ f_{\tilde{Y}}=\pi\circ\mathrm{alb}_{X^{\mathrm{min}}}\circ\nu$. The arrow in the third line is the pullback along $q\colon X'\rightarrow X_{\tilde{Y}}$. It is injective since $q$ is dominant. The isomorphism in the third line uses that $\mathrm{alb}_{X^{\mathrm{min}}}\circ\nu\circ q=\mathrm{alb}_X\circ p$. The arrow in the fourth line uses the injection $q^*\nu^*\mathcal{K}_{X^{\mathrm{min}}}^{d}\hookrightarrow p^*\mathcal{K}_X^{d}$. Finally, the last arrow is the push forward along $p$, and the last isomorphism follows from the projection formula.\par
    From \Cref{Popa_Schnell_Kaehler_case}, it follows that for every holomorphic 1-form $\omega\in H^0(\mathrm{Alb}(Y),\Omega_{\mathrm{Alb}(Y)}^1)$, the pullback $\mathrm{alb}_X^*\pi^*\omega$ has a zero on $X$. Since we assume that $X$ admits $g$ everywhere linearly independent holomorphic 1-forms $\omega_1,\ldots,\omega_g$, the cokernel of
    \begin{equation*}
        \pi^*\colon H^0(\mathrm{Alb}(Y),\Omega_{\mathrm{Alb}(Y)}^1)\longrightarrow H^0(\mathrm{Alb}(X),\Omega_{\mathrm{Alb}(X)}^1)
    \end{equation*}
    has to be of dimension $\geq g$. As the general fiber of $f$ has numerically trivial canonical sheaf, it is mapped to a translate of a subtorus $A\subset\mathrm{Alb}(X)$. This subtorus thus has to be of $\dim \geq g$, since otherwise the image of $\mathrm{alb}_X$ is contained in a proper subtorus of $\mathrm{Alb}(X)$, which is impossible.
\end{proof}
\section{Proof of the main result}
\begin{theorem}\label{precise_main_result_higher_codim}
    Let $X$ be a compact connected K\"ahler manifold, which admits a good minimal model $X^{\mathrm{min}}$. Let $f\colon X^{\mathrm{min}}\rightarrow Y$ be the Iitaka fibration, and suppose that a general fiber of $f$ is mapped to a translate of a subtorus in $\mathrm{Alb}(X)$ of dimension $g>0$. Then there is a finite \'etale cover $X'\rightarrow X$, a compact connected K\"ahler manifold $W$, a complex torus $A=\C^{g'}/\Lambda$ of dimension $g'\geq g$, and a cohomology class $\eta\in H^1(W,\OO_W^{\oplus g'}/\Lambda)$ with $c(\eta)=0$ such that $X'$ is bimeromorphic to $(W\times A)^\eta$.
\end{theorem}
\begin{corollary}\label{corollary_precise_main_result_higher_codim}
    Let $X$ be a compact connected K\"ahler manifold, which admits $g$ everywhere linearly independent holomorphic 1-forms $\omega_1,\ldots,\omega_g\in H^0(X,\Omega_X^1)$. Suppose that $X$ admits a good minimal model. Then there is a finite \'etale cover $X'\rightarrow X$, a compact connected K\"ahler manifold $W$, a complex torus $A=\C^{g'}/\Lambda$ of dimension $g'\geq g$, and a cohomology class $\eta\in H^1(W,\OO_W^{\oplus g'}/\Lambda)$ with $c(\eta)=0$ such that $X'$ is bimeromorphic to $(W\times A)^\eta$.
\end{corollary}
\begin{proof}
    Let $X^{\mathrm{min}}$ be the good minimal model, and let $f\colon X^{\mathrm{min}}\rightarrow Y$ be the Iitaka fibration. \Cref{Iitaka_fibers_are_mapped_to_subtorus_of_dim_g'_geq_g} shows that a general fiber of $f$ is mapped to a translate of a subtorus $A\subset\mathrm{Alb}(X)$ of $\dim A\geq g$. By \Cref{precise_main_result_higher_codim}, there is a finite \'etale cover $X'\rightarrow X$, a compact connected K\"ahler manifold $W$, a complex torus $A=\C^{g'}/\Lambda$ of dimension $g'\geq g$, and a cohomology class $\eta\in H^1(W,\OO_W^{\oplus g'}/\Lambda)$ with $c(\eta)=0$ such that $X'$ is bimeromorphic to $(W\times A)^\eta$.
\end{proof}
For the rest of this section, we fix the following notation. Let $X$ be a compact connected K\"ahler manifold, which admits a good minimal model $X^{\mathrm{min}}$. Let $f\colon X^{\mathrm{min}}\rightarrow Y$ be the Iitaka fibration, and suppose that a general fiber of $f$ is mapped to a translate of a subtorus $A\subset\mathrm{Alb}(X)$ of $\dim A=g>0$. Choose a resolution of singularities $\tilde{Y}\rightarrow Y$. Note that $X^{\mathrm{min}}\times_Y\tilde{Y}$ contains a unique irreducible component that is bimeromorphic to $X^{\mathrm{min}}$. Denote by $X_{\tilde{Y}}$ the normalization of this component, and denote by $f_{\tilde{Y}}\colon X_{\tilde{Y}}\rightarrow\tilde{Y}$ the induced morphism. As $X^{\mathrm{min}}$ has terminal singularities, it has rational singularities \cite[][Thm. 5.22]{Kollar_Mori}. The proof of \cite[][Lem. 8.1]{Kawamata_minimal_models_Kodaira_dim_fiber_spaces} shows that the meromorphic map
\begin{equation*}
    X^{\mathrm{min}}\longdashrightarrow X\stackrel{\mathrm{alb}_X}{\longrightarrow}\mathrm{Alb}(X)
\end{equation*}
extends to a morphism
\begin{equation*}
    \mathrm{alb}_{X^{\mathrm{min}}}\colon X^{\mathrm{min}}\longrightarrow\mathrm{Alb}(X),
\end{equation*}
which we call the Albanese morphism of $X^{\mathrm{min}}$. Denote the composition
\begin{equation*}
    X_{\tilde{Y}}\longrightarrow X^{\mathrm{min}}\stackrel{\mathrm{alb}_{X^{\mathrm{min}}}}{\longrightarrow}\mathrm{Alb}(X)
\end{equation*}
by $\mathrm{alb}_{X_{\tilde{Y}}}$. Note that $f$ induces a morphism $\pi\colon\mathrm{Alb}(X)\rightarrow\mathrm{Alb}(Y):=\mathrm{Alb}(\tilde{Y})$. We hence obtain the commutative diagram
\[\begin{tikzcd}
	{X_{\tilde{Y}}} & {\mathrm{Alb}(X)} \\
	{\tilde{Y}} & {\mathrm{Alb}(Y).}
	\arrow["{\mathrm{alb}_{X_{\tilde{Y}}}}", from=1-1, to=1-2]
	\arrow["{f_{\tilde{Y}}}"', from=1-1, to=2-1]
	\arrow["\pi", from=1-2, to=2-2]
	\arrow["{\mathrm{alb}_{\tilde{Y}}}", from=2-1, to=2-2]
\end{tikzcd}\]
Our first aim is to prove the following proposition
\begin{proposition}\label{key_diagram_higher_codim}
    Write $A=\C^{g}/\Lambda$. Then there is a cohomology class $\eta\in H^1(\tilde{Y},\OO_{\tilde{Y}}^{\oplus g}/\Lambda)$ such that $\mathrm{Alb}(X)\times_{\mathrm{Alb}(Y)}\tilde{Y}$ is isomorphic to $(\tilde{Y}\times A)^\eta$ over $\tilde{Y}$. Moreover, $c(\eta)=0$. In particular, we obtain the commutative diagram
    \begin{equation}\label{key_diagram_display_higher_codim}
        \begin{tikzcd}
	       {X_{\tilde{Y}}} && {(\tilde{Y}\times A)^\eta} \\
	       & {\tilde{Y}} & {.}
	       \arrow["h", from=1-1, to=1-3]
	       \arrow["{f_{\tilde{Y}}}"', from=1-1, to=2-2]
	       \arrow["{\pi^\eta}", from=1-3, to=2-2]
        \end{tikzcd}
    \end{equation}
\end{proposition}
In the following, we will show that after replacing $X$ by a finite \'etale cover if necessary, the morphism $h$ in diagram (\ref{key_diagram_display_higher_codim}) satisfies additional properties.\par
By our assumption, the general fiber of $f_{\tilde{Y}}$ is mapped to a translate of a subtorus $A\subset\mathrm{Alb}(X)$ of $\dim A= g$. We define the positive integer $d(X):=\dim A$. Note that $g\leq d(X)\leq\dim X$.
\begin{lemma}
    Let $X'\rightarrow X$ be a finite \'etale cover. Then the number $d(X')$ is well-defined, and $d(X')\geq d(X)$.
\end{lemma}
\begin{proof}
    Choose a proper bimeromorphic morphism $\tilde{X}\rightarrow X$ from a compact connected K\"ahler manifold $\tilde{X}$ such that the bimeromorphic map $X\dashrightarrow X^{\mathrm{min}}$ extends to a morphism $\tilde{X}\rightarrow X^{\mathrm{min}}$. The finite \'etale cover $X'\rightarrow X$ induces a finite \'etale cover $\tilde{X}'\rightarrow\tilde{X}$. By \Cref{good_minimal_model_ascend_finite_etale_cover}, $X'$ admits a good minimal model $X'^{\mathrm{min}}$, and there is a quasi-\'etale morphism $\tau\colon X'^{\mathrm{min}}\rightarrow X^{\mathrm{min}}$. Moreover, the Iitaka fibration of $X'^{\mathrm{min}}$ is given by the fibration part $X'^{\mathrm{min}}\rightarrow Y'$ of the Stein factorization of $X'^{\mathrm{min}}\rightarrow X^{\mathrm{min}}\rightarrow Y$. Consider the commutative diagram
    \[\begin{tikzcd}
	   {\mathrm{Alb}(X')} & {\mathrm{Alb}(X)} \\
	   {X'^{\mathrm{min}}} & {X^{\mathrm{min}}} \\
	   {Y'} & Y
	   \arrow[two heads, from=1-1, to=1-2]
	   \arrow[from=2-1, to=1-1]
	   \arrow[from=2-1, to=2-2]
	   \arrow[from=2-1, to=3-1]
	   \arrow[from=2-2, to=1-2]
	   \arrow[from=2-2, to=3-2]
	   \arrow[from=3-1, to=3-2]
    \end{tikzcd}\]
    This shows that $d(X')\geq d(X)$.
\end{proof}
As $d(X)$ grows with finite \'etale covers, and as $d(X)$ is bounded from above by $\dim X$, we can replace $X$ by a finite \'etale cover if necessary so that the following assumption is satisfied.
\begin{assumption}\label{assumption_on_dim_Albanese_image_of_Iitaka_fibers}
    \begin{equation*}
        d(X)=\max\{d(X')\mid X'\rightarrow X\;\text{finite \'etale}\}.
    \end{equation*}
\end{assumption}
\begin{lemma}\label{Albanese_morphism_surjective_fibers_connected}
    The following hold true.
    \begin{enumerate}[label=(\roman*)]
        \item\label{surjective_connected} Let $g\colon S\rightarrow T$ be a fibration between normal compact complex analytic varieties. Then the induced morphism $\mathrm{Alb}(S)\twoheadrightarrow\mathrm{Alb}(T)$ is surjective and has connected kernel, where $\mathrm{Alb}(S)$, resp. $\mathrm{Alb}(T)$ are the Albanese tori of a resolution of singularities $\tilde{S}\rightarrow S$, resp. $\tilde{T}\rightarrow T$.
        \item\label{description_Albanese_kernel} The induced morphism $\mathrm{Alb}(X)\twoheadrightarrow\mathrm{Alb}(Y)$ is surjective, and $A=\ker(\mathrm{Alb}(X)\rightarrow\mathrm{Alb}(Y))$.
    \end{enumerate}
\end{lemma}
\begin{proof}
    To show that the morphism $\mathrm{Alb}(S)\rightarrow\mathrm{Alb}(T)$ is surjective with connected fibers, note that we can choose bimeromorphic replacements for $S$ and $T$ to assume that $S$ and $T$ are both smooth. Recall that
    \begin{align*}
        \mathrm{Alb}(S)&=H^0(S,\Omega_S^1)^\vee/H_1(S,\Z),\\
        \mathrm{Alb}(T)&=H^0(T,\Omega_T^1)^\vee/H_1(T,\Z).
    \end{align*}
    Thus, the morphism $\mathrm{Alb}(S)\rightarrow\mathrm{Alb}(T)$ is surjective if and only if $g^*\colon H^0(T,\Omega_T^1)\rightarrow H^0(S,\Omega_S^1)$ is injective, and the morphism $\mathrm{Alb}(S)\rightarrow\mathrm{Alb}(T)$ has connected fibers if and only if the map $H_1(S,\Z)/\mathrm{torsion}\rightarrow H_1(T,\Z)/\mathrm{torsion}$ is surjective.\par
    As, $S$, and $T$ are smooth, the fibration $g\colon S\rightarrow T$ is generically smooth. Thus, the map $g^*\colon H^0(T,\Omega_T^1)\rightarrow H^0(S,\Omega_S^1)$ is injective. This proves the surjectivity of $\mathrm{Alb}(S)\rightarrow\mathrm{Alb}(T)$.\par
    To prove the surjectivity of $H_1(S,\Z)/\mathrm{torsion}\rightarrow H_1(T,\Z)/\mathrm{torsion}$, denote by $T^*\subset T$ the open subset over which $g$ is smooth. By shrinking $T^*$ if necessary, we can assume that the complement $T\setminus T^*$ is a divisor. As $g^*\colon S^*\rightarrow T^*$ is a smooth morphism, Ehresmann's theorem \cite[][Thm. 9.3]{Voisin_HT_1} implies that it is locally diffeomorphic to a trivial fiber bundle with differential geometric fiber $F$. Note that $F$ is connected, i.e. $\pi_0(F)=\{*\}$. By \cite[][Prop. 4.48]{Hatcher_AT}, $g^*\colon S^*\rightarrow T^*$ is a Serre fibration. The long exact sequence of homotopy groups associated to a Serre fibration \cite[][Thm. 4.41]{Hatcher_AT}
    \begin{equation*}
        \pi_1(F)\longrightarrow\pi_1(S^*)\longrightarrow\pi_1(T^*)\longrightarrow\pi_0(F)=\{*\}
    \end{equation*}
    shows that the map $\pi_1(S^*)\twoheadrightarrow\pi_1(T^*)$ is surjective. As $T\setminus T^*$ is a divisor, the map $\pi_1(T^*)\twoheadrightarrow\pi_1(T)$ is surjective as well. Thus, also the map $\pi_1(S)\rightarrow\pi_1(T)$ must be surjective. Taking the abelianization preserves the surjectivity. The isomorphisms $H_1(S,\Z)\cong\pi_1(S)^\mathrm{ab}$, $H_1(T,\Z)\cong\pi_1(T)^\mathrm{ab}$ then show that the map
    \begin{equation*}
        H_1(S,\Z)/\mathrm{torsion}\longrightarrow H_1(T,\Z)/\mathrm{torsion}
    \end{equation*}
    must be surjective as well. This proves \cref{surjective_connected}.\par
    To prove the second assertion, note that $f\colon X^{\mathrm{min}}\rightarrow Y$ is a fibration. Thus, by \cref{surjective_connected}, the morphism $\mathrm{Alb}(X)\twoheadrightarrow\mathrm{Alb}(Y)$ is surjective and has connected kernel. It remains to show that $A=\ker(\mathrm{Alb}(X)\rightarrow\mathrm{Alb}(Y))$. Note that by construction, $A$ is a subtorus in the connected torus $\ker(\mathrm{Alb}(X)\twoheadrightarrow\mathrm{Alb}(Y))$. Suppose this inclusion is proper. Then there exists a non-zero holomorphic 1-form $\omega$ on the subtorus
    \begin{equation*}
        \ker\left(\mathrm{Alb}(X)\twoheadrightarrow\mathrm{Alb}(Y)\right)\subset\mathrm{Alb}(X),
    \end{equation*}
    which restricts to zero on $A$. Extend the holomorphic 1-form $\omega$ on $\ker(\mathrm{Alb}(X)\twoheadrightarrow\mathrm{Alb}(Y))$ to a holomorphic 1-form $\tilde{\omega}$ on $\mathrm{Alb}(X)$. This shows, that $\mathrm{alb}_{X_{\tilde{Y}}}^*\tilde{\omega}$ restricts to zero on the fibers of $f_{\tilde{Y}}$. Thus, $\mathrm{alb}_{X_{\tilde{Y}}}^*\tilde{\omega}$ must be a pullback of a holomorphic 1-form on $\mathrm{Alb}(Y)$. However, as $\tilde{\omega}$ restricts to the non-trivial holomorphic 1-form $\omega$ on $\ker(\mathrm{Alb}(X)\twoheadrightarrow\mathrm{Alb}(Y))$, it cannot be a pullback from $\mathrm{Alb}(Y)$, a contradiction.
\end{proof}
\begin{proof}[Proof of \Cref{key_diagram_higher_codim}]
    \Cref{Albanese_morphism_surjective_fibers_connected} shows that $\mathrm{Alb}(X)\rightarrow\mathrm{Alb}(Y)$ is a locally trivial torus fiber bundle with trivial monodromy and fiber $A=\C^{g}/\Lambda$. Thus, by \Cref{classification_torus_fiber_bundles}, there is a cohomology class $\eta'\in H^1(\mathrm{Alb}(Y),\OO_{\mathrm{Alb}(Y)}^{\oplus g}/\Lambda)$ such that
    \begin{equation*}
        \mathrm{Alb}(X)\cong (\mathrm{Alb}(Y)\times A)^{\eta'}.
    \end{equation*}
    This implies that
    \begin{equation*}
        \tilde{Y}\times_{\mathrm{Alb}(Y)}\mathrm{Alb}(X)\cong (\tilde{Y}\times A)^\eta,
    \end{equation*}
    where $\eta\in H^1(\tilde{Y},\OO_{\tilde{Y}}^{\oplus g}/\Lambda)$ is the pullback of $\eta'$ to $\tilde{Y}$. Since $\mathrm{Alb}(X)$ is K\"ahler, $c(\eta')$ is torsion. Recall that the integral cohomology of a torus is torsion free. Thus, $c(\eta')=0$. Hence, also $c(\eta)=0$.
\end{proof}
Our next aim is to prove the following proposition.
\begin{proposition}\label{h_connected_fibers_c(eta)=0}
    After replacing $X$ by a finite \'etale cover if necessary, we can assume that $h$ has connected fibers, and that still $c(\eta)=0$ in diagram (\ref{key_diagram_display_higher_codim}).
\end{proposition}
Consider the Stein factorization
\[\begin{tikzcd}
	{X_{\tilde{Y}}} && {(\tilde{Y}\times A)^\eta} \\
	& S & {.}
	\arrow["h", from=1-1, to=1-3]
	\arrow[from=1-1, to=2-2]
	\arrow["{\overline{h}}"', from=2-2, to=1-3]
\end{tikzcd}\]
\begin{lemma}
    Let $D\subset (\tilde{Y}\times A)^\eta$ be the branch locus of $\overline{h}\colon S\rightarrow (\tilde{Y}\times A)^\eta$. Then the following assertions hold true.
    \begin{enumerate}[label=(\roman*)]
        \item $\pi^\eta(D)$ is a proper analytic subset of $\tilde{Y}$.
        \item Set $Y^*:=\tilde{Y}\setminus\pi^\eta(D)$. Then for all $y\in Y^*$, the morphism $\overline{h}$ induces a finite \'etale cover $S_y:=(\pi^\eta\circ\overline{h})^{-1}(y)\rightarrow (\pi^\eta)^{-1}(y)$.
        \item Moreover, if we identify $(\pi^\eta)^{-1}(y)\cong A$, then the induced finite \'etale cover $S_y\rightarrow A$ is isomorphic to a fixed finite \'etale cover $\sigma\colon B\rightarrow A$ for all $y\in Y^*$.
    \end{enumerate}
\end{lemma}
\begin{proof}
    By \cite[][Lem. 5.17]{Kollar_Mori}, there is a non-empty open subset $Y^\circ\subset Y$ such that the fibers of $f\colon X^{\mathrm{min}}\rightarrow Y$ are reduced and have terminal singularities over $Y^\circ$, and such that the complement $Y\setminus Y^\circ$ is an analytic subset. By shrinking $Y^\circ$ if necessary, we can assume that the proper bimeromorphic morphism $\tilde{Y}\rightarrow Y$ is an isomorphism over $Y^\circ$, i.e. $Y^\circ$ is smooth. In particular, we can view $Y^\circ$ as a connected open dense subset of $\tilde{Y}$. For $y\in Y^\circ$, consider the Stein factorization
    \[\begin{tikzcd}
	   {f_{\tilde{Y}}^{-1}(y)} && A \\
	   & {B(y)} & {.}
	   \arrow[from=1-1, to=1-3]
	   \arrow[from=1-1, to=2-2]
	   \arrow[from=2-2, to=1-3]
    \end{tikzcd}\]
    Note that $f_{\tilde{Y}}^{-1}(y)$ is a compact K\"ahler variety with terminal singularities and numerically trivial canonical sheaf.\par
    The decomposition theorem for numerically $K$-trivial compact K\"ahler varieties \cite[][Thm. A]{Bakker_Guenancia_Lehn}, shows that $B(y)$ is a complex torus. Therefore, the morphism $B(y)\rightarrow A$ is finite \'etale.\par
    Note that for all $y\in Y^\circ$, we can identify $B(y)\rightarrow A$ with $(\pi^\eta\circ\overline{h})^{-1}(y)\rightarrow (\pi^\eta)^{-1}(y)$. In particular, $\overline{h}$ is unramified over $(\pi^\eta)^{-1}(Y^\circ)$. Thus, $Y^*$ contains the non-empty open subset $Y^\circ$. This proves the first and second assertion.\par
    To prove the third assertion, note that since $Y^*$ is connected, and since $A$ admits only countably many finite \'etale covers, the morphism $S_y\rightarrow A$ does not depend on $y\in Y^*$.
\end{proof}
\begin{construction}\label{construction_Z}
    Write $B=\C^{g}/\Gamma$, and set $G:=\ker(\sigma\colon B\rightarrow A)$. As $c(\eta)=0$, \Cref{lemma_extension_of_etale_cover_of_elliptic_curves} gives a cohomology class $\theta\in H^1(\tilde{Y},\OO_{\tilde{Y}}^{\oplus g}/\Gamma)$ such that the torus fiber bundle $\pi^\theta\colon(\tilde{Y}\times B)^\theta\rightarrow\tilde{Y}$ is K\"ahler, $c(\theta)=0$, and a finite \'etale cover
    \begin{equation*}
        \rho\colon (\tilde{Y}\times B)^\theta\longrightarrow (\tilde{Y}\times A)^\eta.
    \end{equation*}
    Note that $G$ acts on $(\tilde{Y}\times B)^\theta$ as the group of deck transformations of $\rho$.\par
    For all open subsets $U\subset Y^*$, define
    \begin{equation*}
        \mathcal{Z}(U):=\mathrm{Isom}_{(\pi^\eta)^{-1}(U)}(S\vert_U,(\pi^\theta)^{-1}(U)):=\{\varphi\colon S\vert_U\stackrel{\sim}{\longrightarrow}(\pi^\theta)^{-1}(U)\mid \rho\circ\varphi=\overline{h}\}.
    \end{equation*}
    See \cite[][Sec. 3 Step 1]{Hao_Wang_Zhang_gmm} for a related approach. We claim that the assignment $U\mapsto\mathcal{Z}(U)$ defines a locally constant $G$-equivariant sheaf of sets.\par
    Indeed, the sheaf property is straightforward to verify. The $G$-action on $(\tilde{Y}\times B)^\theta$ induced a $G$-action on $\mathcal{Z}$ via
    \begin{align*}
        \mathcal{Z}(U)&\longrightarrow\mathcal{Z}(U)\\
        \varphi&\longmapsto g\circ\varphi,
    \end{align*}
    where $g\colon (\tilde{Y}\times B)^\theta\rightarrow (\tilde{Y}\times B)^\theta$ is the deck transformation corresponding to $g\in G$. To show that the sheaf $\mathcal{Z}$ is locally constant, let $U\subset Y^*$ be a polydisc. Then $(\pi^\eta)^{-1}(U)\cong U\times A$, and hence $(\pi^\eta\circ\overline{h})^{-1}(U)\cong U\times B$. Clearly, the set of isomorphisms $U\times B\stackrel{\sim}{\rightarrow} U\times B$ over $U\times A$ consists of translations by elements in $G$.\par
    Choose an open cover $\{U_i\}_{i\in I}$ over $Y^*$ that trivializes $(\tilde{Y}\times A)^\eta\vert_{Y^*}$, and $\mathcal{Z}$. Note that for all $i\in I$, there is a unique topology and complex structure on $\mathcal{Z}(U_i)$ such that the morphism $Z_i\rightarrow U_i$ from the associated complex manifold $Z_i$ is finite \'etale. Since $\mathcal{Z}$ is a sheaf, these local patches glue to a complex manifold $Z^*$, and the induced morphism $u^*\colon Z^*\rightarrow Y^*$ is finite \'etale with group of deck transformations $G$.\par
    Note that the diagonal morphism $Z^*\rightarrow Z^*\times_{Y^*}Z^*$ gives rise to a global section of $(u^*)^{-1}\mathcal{Z}$. In particular, we obtain a $G$-equivariant isomorphism
    \begin{equation*}
        S\times_{\tilde{Y}}Z^*\stackrel{\sim}{\longrightarrow} (\tilde{Y}\times B)^\theta\times_{\tilde{Y}}Z^*
    \end{equation*}
    over $(\tilde{Y}\times A)^\eta\times_{\tilde{Y}}Z^*$.\par
    The Grauert--Remmert theorem \cite[][Exp. XII, Thm. 5.4]{SGA} allows to extend $Z^*\rightarrow Y^*$ to a unique normal finite cover $Z\rightarrow\tilde{Y}$ with Galois group $G$. If $Z$ is not connected, we replace it by a connected component and replace $G$ by the stabilizer of this connected component. Note that $Z/G$ is normal and the morphism $Z/G\rightarrow\tilde{Y}$ is an isomorphism over $Y^*$. Since $\tilde{Y}$ is normal, $Z/G\cong\tilde{Y}$.\par
    Since functorial resolutions of singularities exists \cite[][Thm. 13.2]{Bierstone_Milman}, we can choose a resolution of singularities $\tilde{Z}\rightarrow Z$ such that the $G$-action on $Z$ lifts to a $G$-action on $\tilde{Z}$. Denote the composition by $u\colon\tilde{Z}\rightarrow Z\rightarrow\tilde{Y}$. Note that the $G$-actions on $(\tilde{Y}\times B)^\theta$ and on $\tilde{Z}$ induce a $G$-action on $(\tilde{Y}\times B)^\theta\times_{\tilde{Y}}\tilde{Z}\cong(\tilde{Z}\times B)^{u^*\theta}$. Since $G$ acts fixed point freely on $(\tilde{Y}\times B)^\theta$, the $G$-action on $(\tilde{Z}\times B)^{u^*\theta}$ is fixed point free as well.
\end{construction}
\begin{lemma}\label{description_Stein_factorization_higher_codim}
    Denote by $S_{\tilde{Z}}$ the normalization of the unique irreducible component of $S\times_{\tilde{Y}}\tilde{Z}$ that is bimeromorphic to $S\times_{\tilde{Y}}Z^*$, and denote by $S_{\tilde{Z}/G}$ the normalization of the unique irreducible component of $S\times_{\tilde{Y}}\tilde{Z}/G$ that is bimeromorphic to $S$. Then there is a canonical $G$-equivariant isomorphism
    \begin{equation*}
        S_{\tilde{Z}}\stackrel{\sim}{\longrightarrow}(\tilde{Z}\times B)^{u^*\theta}
    \end{equation*}
    over $(\tilde{Z}\times A)^{u^*\eta}$. Moreover, taking the quotient by $G$ results in the isomorphism
    \begin{equation*}
        S_{\tilde{Z}/G}\stackrel{\sim}{\longrightarrow}(\tilde{Z}\times B)^{u^*\theta}/G.
    \end{equation*}
\end{lemma}
\begin{proof}
    Recall from \Cref{construction_Z} there is a canonical $G$-equivariant isomorphism
    \begin{equation*}
        S\times_{\tilde{Y}}Z^*\stackrel{\sim}{\longrightarrow} (\tilde{Y}\times B)^\theta\times_{\tilde{Y}}Z^*\cong (\tilde{Z}\times B)^{u^*\theta}\times_{\tilde{Z}}Z^*
    \end{equation*}
    over $(\tilde{Y}\times A)^\eta\times_{\tilde{Y}}Z^*$.\par
    On the one hand, the composition
    \begin{equation*}
        S_{\tilde{Z}}\longrightarrow S\times_{\tilde{Y}}\tilde{Z}\longrightarrow (\tilde{Y}\times A)^\eta\times_{\tilde{Y}}\tilde{Z}\cong (\tilde{Z}\times A)^{u^*\eta}
    \end{equation*}
    is a finite morphism of degree $\vert G\vert$. On the other hand, the morphism
    \begin{equation*}
        (\tilde{Z}\times B)^{u^*\theta}\cong(\tilde{Y}\times B)^\theta\times_{\tilde{Y}}\tilde{Z}\stackrel{u^*\rho}{\longrightarrow}(\tilde{Y}\times A)^\eta\times_{\tilde{Y}}\tilde{Z}\cong (\tilde{Z}\times A)^{u^*\eta}
    \end{equation*}
    is a finite morphism of degree $\vert G\vert$.\par
    The above argument shows that both finite morphisms agree over $Z^*$. Thus, the above $G$-equivariant isomorphism extends to an isomorphism
    \begin{equation*}
        S_{\tilde{Z}}\stackrel{\sim}{\longrightarrow}(\tilde{Z}\times B)^{u^*\theta}.
    \end{equation*}
    Note that the $G$-action on $\tilde{Z}$ induces a $G$-action on $S_{\tilde{Z}}$. This makes the above isomorphism $G$-equivariant. Also note, that $S_{\tilde{Z}}/G\cong S_{\tilde{Z}/G}$. In particular, taking the quotient of the above $G$-equivariant isomorphism results in the isomorphism
    \begin{equation*}
        S_{\tilde{Z}/G}\stackrel{\sim}{\longrightarrow}(\tilde{Z}\times B)^{u^*\theta}/G.
    \end{equation*}
\end{proof}
We are now in position to prove \Cref{h_connected_fibers_c(eta)=0}.
\begin{proof}[Proof of \Cref{h_connected_fibers_c(eta)=0}]
    Denote by $X_{\tilde{Z}/G}$ the unique irreducible component of $X^{\mathrm{min}}\times_Y\tilde{Z}/G$ that is bimeromorphic to $X^{\mathrm{min}}$. Choose a resolution of singularities $\tilde{X}\rightarrow X_{\tilde{Z}/G}$, and define
    \begin{equation*}
        \tilde{X}':=\tilde{X}_{\tilde{Z}/G}\times_{S_{\tilde{Z}/G}}(\tilde{Z}\times B)^{u^*\theta}.
    \end{equation*}
    Since $G$ acts fixed point freely on $(\tilde{Z}\times B)^{u^*\theta}$ by \Cref{construction_Z} and since $(\tilde{Z}\times B)^{u^*\theta}/G\cong S_{\tilde{Z}/G}$ by \Cref{description_Stein_factorization_higher_codim}, the induced morphism $\tilde{X}'\rightarrow\tilde{X}$ is finite \'etale. Recall that $\tilde{X}$ is bimeromorphic to $X$. Thus, there is a finite \'etale cover $X'\rightarrow X$ such that $X'$ is bimeromorphic to $\tilde{X}'$.\par
    By \Cref{good_minimal_model_ascend_finite_etale_cover}, $X'$ admits a good minimal model $X'^{\mathrm{min}}$, which admits also a quasi-\'etale morphism $X'^{\mathrm{min}}\rightarrow X^{\mathrm{min}}$, and the Iitaka fibration $f'\colon X'^{\mathrm{min}}\rightarrow Y'$ of $X'^{\mathrm{min}}$ is given by the Stein factorization of $X'^{\mathrm{min}}\rightarrow X^{\mathrm{min}}\rightarrow Y$. Thus, $f'$ is bimeromorphically equivalent to the Stein factorization of $\tilde{X}'\rightarrow\tilde{X}\rightarrow\tilde{Z}/G\rightarrow Y$. From the commutative diagram
    \[\begin{tikzcd}
	   {\tilde{X}'} & {(\tilde{Z}\times B)^\theta} \\
	   {\tilde{X}} & {S_{\tilde{Z}/G},}
	   \arrow[from=1-1, to=1-2]
	   \arrow[from=1-1, to=2-1]
	   \arrow[from=1-2, to=2-2]
	   \arrow[from=2-1, to=2-2]
    \end{tikzcd}\]
    it follows that $f'$ is bimeromorphic to the Stein factorization of $\tilde{X}'\rightarrow\tilde{Z}\rightarrow Y$. Therefore, $Y'$ fits into the Stein factorization of $\tilde{Z}\rightarrow Y$, i.e. there is a proper bimeromorphic morphism $\tilde{Z}\rightarrow Y'$.\par
    Note that $X'^{\mathrm{min}}\times_{Y'}\tilde{Z}$ contains a unique irreducible component that is bimeromorphic to $X'^{\mathrm{min}}$. Denote by $X'_{\tilde{Z}}$ the normalization of this component, and denote by $f'_{\tilde{Z}}\colon X'_{\tilde{Z}}\rightarrow\tilde{Z}$ the induced morphism. The universal property of the fiber product gives rise to the morphism
    \begin{equation*}
        X'_{\tilde{Z}}\longrightarrow (\tilde{Y}\times A)^\eta\times_{\tilde{Y}}\tilde{Z}.
    \end{equation*}
    Note that the Stein factorization of this morphism is given by
    \begin{equation*}
        X'_{\tilde{Z}}\longrightarrow S_{\tilde{Z}}\longrightarrow(\tilde{Y}\times A)^\eta\times_{Y'}\tilde{Z}.
    \end{equation*}
    Thus, using the isomorphism in \Cref{description_Stein_factorization_higher_codim} results in the morphism with connected fibers
    \begin{equation*}
        h'\colon X'_{\tilde{Z}}\longrightarrow (\tilde{Z}\times B)^{u^*\theta}.
    \end{equation*}
    Finally, note that since $c(\theta)=0$, $c(u^*\theta)=u^*c(\theta)=0$.
\end{proof}
From now on, we assume that we are in the situation and notation from \Cref{h_connected_fibers_c(eta)=0}. Our next aim is to prove the following proposition.
\begin{proposition}\label{proposition_torus_action_h_equivariant}
    After replacing $X$ by a finite \'etale cover if necessary, the torus $A$ acts on $X_{\tilde{Y}}$ by translations $\mathrm{tr}(a)\colon X_{\tilde{Y}}\rightarrow X_{\tilde{Y}}$ such that the morphism $h$ in diagram (\ref{key_diagram_display_higher_codim}) becomes equivariant, where we endow $(\tilde{Y}\times A)^\eta$ with the natural $A$-action. In other words, the following diagram commutes for all $a\in A$
    \[\begin{tikzcd}
	   {X_{\tilde{Y}}} & {X_{\tilde{Y}}} \\
	   {(\tilde{Y}\times A)^\eta} & {(\tilde{Y}\times A)^\eta.}
	   \arrow["{\mathrm{tr}(a)}", from=1-1, to=1-2]
	   \arrow[from=1-1, to=2-1]
	   \arrow[from=1-2, to=2-2]
	   \arrow["{\mathrm{tr}(a)}", from=2-1, to=2-2]
    \end{tikzcd}\]
    Moreover, $h$ still has connected fibers and $c(\eta)=0$.
\end{proposition}
\begin{lemma}\label{isomorphism_between_Albanese_X_and_torus_fiber_bundle}
    The morphism $h\colon X_{\tilde{Y}}\rightarrow (\tilde{Y}\times A)^\eta$ induces an isomorphism
    \begin{equation*}
        \mathrm{Alb}(X)\stackrel{\sim}{\longrightarrow}\mathrm{Alb}((\tilde{Y}\times A)^\eta).
    \end{equation*}
\end{lemma}
\begin{proof}
    Since $h$ has connected fibers, \Cref{Albanese_morphism_surjective_fibers_connected} shows that the morphism $\mathrm{Alb}(X)\rightarrow\mathrm{Alb}((\tilde{Y}\times A)^\eta)$ is surjective with connected fibers. It hence suffices to show that the fibers are zero-dimensional. This follows from the commutative diagram
    \[\begin{tikzcd}
	   {\mathrm{Alb}(X)} && {\mathrm{Alb}((\tilde{Y}\times A)^\eta)} \\
	   & {\mathrm{Alb}(Y)} & {,}
	   \arrow[from=1-1, to=1-3]
	   \arrow[from=1-1, to=2-2]
	   \arrow[from=1-3, to=2-2]
    \end{tikzcd}\]
    since both diagonal arrows have kernel of dimension equal to $\dim A$ by \Cref{assumption_on_dim_Albanese_image_of_Iitaka_fibers}, resp. by \Cref{fundamental_properties_special_torus_bundles}.
\end{proof}
\begin{construction}\label{construction_relative_Albanese}
    In the situation and notation from \Cref{h_connected_fibers_c(eta)=0}, choose a functorial resolution of singularities $\tilde{X}\rightarrow X_{\tilde{Y}}$ \cite[][Thm. 13.2]{Bierstone_Milman}, and denote the induced fibration by $\tilde{f}\colon\tilde{X}\rightarrow\tilde{Y}$.\par
    By \cite[][Lem. 5.17]{Kollar_Mori}, there is a non-empty open subset $Y^*\subset Y$ such that the fibers of $f\colon X^{\mathrm{min}}\rightarrow Y$ are reduced and have terminal singularities over $Y^*$, and such that the complement $Y\setminus Y^*$ is an analytic subset. By shrinking $Y^*$ if necessary, we can assume that the proper bimeromorphic morphism $\tilde{Y}\rightarrow Y$ is an isomorphism over $Y^*$, i.e. $Y^*$ is smooth. In particular, we can identify $Y^*$ with a connected open dense subset $Y^*\subset\tilde{Y}$. By shrinking $Y^*$ further if necessary, we can assume that $Y^*$ is affine ($Y$ is a projective variety), that $\tilde{f}$ is smooth over $Y^*$, that $f_{\tilde{Y}}$ is flat over $Y^*$, and that the morphisms on fibers
    \begin{equation*}
        \tilde{f}^{-1}(y)\longrightarrow f^{-1}_{\tilde{Y}}(y)
    \end{equation*}
    are bimeromorphic for all $y\in Y^*$. Note that each fiber $f_{\tilde{Y}}^{-1}(y)$ is a numerically $K$-trivial K\"ahler variety with terminal singularities. Denote by $\tilde{f}^*\colon \tilde{X}^*:=\tilde{f}^{-1}_{\tilde{Y}}(Y^*)\rightarrow Y^*$ the restricted morphism.\par
    Let $n$ be the relative dimension of $\tilde{f}^*$. Consider the variation of Hodge structures $H$ associated to $\tilde{f}^*$, i.e. $H$ consists of the local system
    \begin{equation*}
        H_\Z:=R^{2n-1}(\tilde{f}^*)_*\underline{\Z}/\mathrm{torsion},
    \end{equation*}
    the holomorphic vector bundle $\mathcal{H}:=H_\Z\otimes\OO_{Y^*}$, the holomorphic subbundle
    \begin{equation*}
        \mathcal{F}^n:=R^{n-1}(\tilde{f}^*)_*\Omega^n_{\tilde{X}^*/Y^*}\subset\mathcal{H},
    \end{equation*}
    and the quotient $\mathcal{H}^{n-1,n}:=\mathcal{H}/\mathcal{F}^n$. Consider the Jacobian fibration
    \begin{equation*}
        \xi\colon J_H:=\mathcal{H}^{n-1,n}/H_\Z\longrightarrow Y^*,
    \end{equation*}
    and denote its sheaf of local sections by $\mathcal{J}_H=\mathcal{H}^{n-1,n}/H_\Z$.\par
    Since $\tilde{f}^*$ is smooth over $Y^*$, we can find an open cover $\{U_i\}_{i\in I}$ of $Y^*$ such that $\tilde{f}^*$ admits a local section $s_i\colon U_i\rightarrow\tilde{X}^*_{\tilde{Y}}$ over each $U_i$. These local sections give rise to local relative Albanese morphisms
    \begin{equation*}
        a_i\colon (\tilde{f}^*)^{-1}(U_i)\longrightarrow J_H\vert_{U_i}
    \end{equation*}
    over $U_i$. On $U_i\cap U_j$, define the local section $\zeta_{ij}$ of $J_H\rightarrow Y^*$ as
    \begin{equation*}
        \zeta_{ij}\colon U_i\cap U_j\stackrel{s_j\vert_{U_i\cap U_j}}{\longrightarrow}(\tilde{f}^*)^{-1}(U_i\cap U_j)\stackrel{a_i\vert_{U_i\cap U_j}}{\longrightarrow}J_H\vert_{U_i\cap U_j}.
    \end{equation*}
    Then the following diagram commutes
    \[\begin{tikzcd}
	   {(\tilde{f}^*)^{-1}(U_i\cap U_j)} & {J_H\vert_{U_i\cap U_j}} \\
	   {(\tilde{f}^*)^{-1}(U_i\cap U_j)} & {J_H\vert_{U_i\cap U_j}}
	   \arrow["{a_i\vert_{U_i\cap U_j}}", from=1-1, to=1-2]
	   \arrow["{\mathrm{id}}"', from=1-1, to=2-1]
	   \arrow["{\mathrm{tr}(\zeta_{ij})}", from=1-2, to=2-2]
	   \arrow["{a_j\vert_{U_i\cap U_j}}", from=2-1, to=2-2]
    \end{tikzcd}\]
    Moreover, a calculation shows that the local sections $\zeta_{ij}$ satisfy the cocycle condition, i.e. we obtain a class $\zeta\in H^1(Y^*,\mathcal{J}_H)$. Thus, the local relative Albanese morphisms glue to a global relative Albanese morphism
    \begin{equation*}
        \mathrm{alb}_{\tilde{X}^*/Y^*}\colon \tilde{X}^*_{\tilde{Y}}\longrightarrow J_H^\zeta
    \end{equation*}
    over $Y^*$. Moreover, since the fibers of $f^*_{\tilde{Y}}$ have terminal singularities the proof of \cite[][Lem. 8.1]{Kawamata_minimal_models_Kodaira_dim_fiber_spaces} shows that the global relative Albanese morphism factors through $\tilde{X}^*\rightarrow X^*_{\tilde{Y}}$:
    \begin{equation*}
        \mathrm{alb}_{\tilde{X}^*/Y^*}\colon \tilde{X}^*\longrightarrow X^*_{\tilde{Y}}\stackrel{\mathrm{alb}_{X^*_{\tilde{Y}}/Y^*}}{\longrightarrow} J_H^\zeta.
    \end{equation*}
\end{construction}
\begin{lemma}\label{relative_automorphisms_global_section_on_relative_Albanese_higher_codim}
    Define $H^0(Y^*,\mathcal{J}_H)^\circ\subset H^0(Y^*,\mathcal{J}_H)$ as the image of $H^0(Y^*,\mathcal{H}^{n-1,n})\rightarrow H^0(Y^*,\mathcal{J}_H)$. Denote the image of the exponential map $\exp\colon H^0(Y^*,(f^*_{\tilde{Y}})_*\mathcal{T}_{X^*_{\tilde{Y}}/Y^*})\rightarrow\mathrm{Aut}(X^*_{\tilde{Y}}/Y^*)$ by $\mathrm{Aut}^\circ(X^*_{\tilde{Y}}/Y^*)$. Then there is a surjective homomorphism with finite kernel
    \begin{equation*}
        \mathrm{Aut}^\circ(X^*_{\tilde{Y}}/Y^*)\longrightarrow H^0(Y^*,\mathcal{J}_H)^\circ.
    \end{equation*}
\end{lemma}
\begin{proof}
    Let us construct a group homomorphism
    \begin{equation*}
        \mathrm{Aut}^\circ(X^*_{\tilde{Y}}/Y^*)\longrightarrow H^0(Y^*,\mathcal{J}_H),
    \end{equation*}
    and show that its image is $H^0(Y^*,\mathcal{J}_H)^\circ$, and that its kernel is finite.\par
    As the resolution of singularities $\tilde{X}^*\rightarrow X^*_{\tilde{Y}}$ from \Cref{construction_relative_Albanese} was chosen to be a functorial, any automorphism $g\colon X^*_{\tilde{Y}}\rightarrow X^*_{\tilde{Y}}$ over $Y^*$ lifts to a unique automorphism $\tilde{g}\colon\tilde{X}^*\rightarrow\tilde{X}^*$ over $Y^*$. In the notation from \Cref{construction_relative_Albanese}, define the local section $\lambda_i$ of $J_H\rightarrow Y^*$ over $U_i$ as
    \begin{equation*}
        \lambda_i\colon U_i\stackrel{s_i}{\longrightarrow}(\tilde{f}^*)^{-1}(U_i)\stackrel{\tilde{g}_i}{\longrightarrow}(\tilde{f}^*)^{-1}(U_i)\stackrel{a_i}{\longrightarrow}J_H\vert_{U_i}.
    \end{equation*}
    Then the following diagram commutes
    \[\begin{tikzcd}
	   {(\tilde{f}^*)^{-1}(U_i)} & {J_H\vert_{U_i}} \\
	   {(\tilde{f}^*)^{-1}(U_i)} & {J_H\vert_{U_i}.}
	   \arrow["{a_i}", from=1-1, to=1-2]
	   \arrow["{\tilde{g}_i}"', from=1-1, to=2-1]
	   \arrow["{\mathrm{tr}(\lambda_i)}", from=1-2, to=2-2]
	   \arrow["{a_i}", from=2-1, to=2-2]
    \end{tikzcd}\]
    Consider the commutative diagram
    \[\begin{tikzcd}
	   & {(\tilde{f}^*)^{-1}(U_i\cap U_j)} && {J_H\vert_{U_i\cap U_j}} \\
	   {(\tilde{f}^*)^{-1}(U_i\cap U_j)} && {J_H\vert_{U_i\cap U_j}} \\
	   & {(\tilde{f}^*)^{-1}(U_i\cap U_j)} && {J_H\vert_{U_i\cap U_j}} \\
	   {(\tilde{f}^*)^{-1}(U_i\cap U_j)} && {J_H\vert_{U_i\cap U_j},}
	   \arrow[from=1-2, to=1-4]
	   \arrow[from=1-2, to=3-2]
	   \arrow["{\mathrm{tr}(\zeta_{ij})}"{description}, from=1-4, to=3-4]
	   \arrow["{\tilde{g}_{ij}}"{description}, from=2-1, to=1-2]
	   \arrow[from=2-1, to=2-3]
	   \arrow[from=2-1, to=4-1]
	   \arrow["{\mathrm{tr}(\lambda_i\vert_{U_i\cap U_j})}"{description}, from=2-3, to=1-4]
	   \arrow["{\mathrm{tr}(\zeta_{ij})}"{description}, from=2-3, to=4-3]
	   \arrow[from=3-2, to=3-4]
	   \arrow["{\tilde{g}_{ij}}"{description}, from=4-1, to=3-2]
	   \arrow[from=4-1, to=4-3]
	   \arrow["{\mathrm{tr}(\lambda_j\vert_{U_i\cap U_j})}"{description}, from=4-3, to=3-4]
    \end{tikzcd}\]
    where the upper horizontal arrows are  given by $a_i\vert_{U_i\cap U_j}$, and the lower horizontal arrows are given by $a_j\vert_{U_i\cap U_j}$. This implies
    \begin{equation*}
        \lambda_j\vert_{U_i\cap U_j}+\zeta_{ij}=\lambda_i\vert_{U_i\cap U_j}+\zeta_{ij}.
    \end{equation*}
    In particular, the local sections $\lambda_i$ and $\lambda_j$ agree on $U_i\cap U_j$, and hence they glue to a global section $\lambda\in H^0(Y^*,\mathcal{J}_H)$. This defines a group homomorphism
    \begin{equation*}
        \mathrm{Aut}^\circ(X^*_{\tilde{Y}}/Y^*)\longrightarrow H^0(Y^*,\mathcal{J}_H).
    \end{equation*}
    Consider the exponential map
    \begin{equation*}
        \exp\colon H^0(Y^*,(f^*_{\tilde{Y}})_*\mathcal{T}_{X^*_{\tilde{Y}}/Y^*})\longrightarrow\mathrm{Aut}^\circ(X^*_{\tilde{Y}}/Y^*).
    \end{equation*}
    We will construct an isomorphism $H^0(Y^*,(f^*_{\tilde{Y}})_*\mathcal{T}_{X^*_{\tilde{Y}}/Y^*})\stackrel{\sim}{\rightarrow} H^0(Y^*,\mathcal{H}^{n-1,n})$ such that the following diagram commutes
    \[\begin{tikzcd}
	   {H^0(Y^*,(f^*_{\tilde{Y}})_*\mathcal{T}_{X^*_{\tilde{Y}}/Y^*})} & {H^0(Y^*,\mathcal{H}^{n-1,n})} \\
	   {\mathrm{Aut}^\circ(X^*_{\tilde{Y}}/Y^*)} & {H^0(Y^*,\mathcal{J}_H)^\circ.}
	   \arrow["\sim", from=1-1, to=1-2]
	   \arrow[two heads, from=1-1, to=2-1]
	   \arrow[from=1-2, to=2-2]
	   \arrow[from=2-1, to=2-2]
    \end{tikzcd}\]
    Note that this implies that the image of $\mathrm{Aut}^\circ(X^*_{\tilde{Y}}/Y^*)\rightarrow H^0(Y^*,\mathcal{J}_H)$ is equal to $H^0(Y^*,\mathcal{J}_H)^\circ$.\par
    By \Cref{isogeny_from_Automorphisms_to_Albanese_numerically_K-trivial}, we have for all $y\in Y^*$
    \begin{equation*}
        \dim H^0(f^{-1}_{\tilde{Y}}(y),\mathcal{T}_{X^*_{\tilde{Y}}/Y^*}\vert_{f^{-1}_{\tilde{Y}}(y)})=\dim H^0(f^{-1}_{\tilde{Y}}(y),\mathcal{T}_{f^{-1}_{\tilde{Y}}(y)})=\dim\mathrm{Alb}(f^{-1}_{\tilde{Y}}(y))=\mathrm{rk}\mathcal{H}^{n-1,n},
    \end{equation*}
    which is constant. Since $f_{\tilde{Y}}$ is flat over $Y^*$, Grauert's theorem \cite[][§7, Satz 5]{Grauert_analytische_Garbentheorie_Modulraeume} implies that $(f^*_{\tilde{Y}})_*\mathcal{T}_{X^*_{\tilde{Y}}/Y^*}$ is locally free.\par
    Note that there is a canonical map
    \begin{equation*}
        (\tilde{f}^*)_*\mathcal{T}_{\tilde{X}^*/Y^*}\longrightarrow (f^*_{\tilde{Y}})_*\mathcal{T}_{X^*_{\tilde{Y}}/Y^*}.
    \end{equation*}
    As $\tilde{X}\rightarrow X_{\tilde{Y}}$ is a functorial resolution of singularities, the above map is an isomorphism on each fiber. Since $\tilde{f}^*$ is smooth, $(\tilde{f}^*)_*\mathcal{T}_{\tilde{X}^*/Y^*}$ is locally free. Hence, we conclude $(\tilde{f}^*)_*\mathcal{T}_{\tilde{X}^*/Y^*}\cong(f^*_{\tilde{Y}})_*\mathcal{T}_{X^*_{\tilde{Y}}/Y^*}$.\par
    Next, note that there is a canonical map
    \begin{equation*}
        (\tilde{f}^*)_*\mathcal{T}_{\tilde{X}^*/Y^*}\longrightarrow ((\tilde{f}^*)_*\Omega_{\tilde{X}^*/Y^*}^1)^\vee.
    \end{equation*}
    As $\tilde{f}^*$ is smooth, both sheaves are locally free. The fiber over $y\in Y^*$ on the left hand side can be identified with $H^0(\tilde{f}^{-1}(y),\mathcal{T}_{\tilde{f}^{-1}(y)})\cong H^0(f^{-1}_{\tilde{Y}}(y),\mathcal{T}_{f_{\tilde{Y}}^{-1}(y)})$, where we use that $\tilde{X}\rightarrow X_{\tilde{Y}}$ is a functorial resolution of singularities. The fiber over $y\in Y^*$ on the right hand side can be identified with $H^0(\tilde{f}^{-1}(y),\Omega_{\tilde{f}^{-1}(y)}^1)\cong H^0(f^{-1}_{\tilde{Y}}(y),\Omega_{f^{-1}_{\tilde{Y}}(y)}^{[1]})$ by \cite[][Cor. 1.7]{Kebekus_Schnell_pullback_reflexive_forms}. Moreover, \Cref{isogeny_from_Automorphisms_to_Albanese_numerically_K-trivial} shows that the natural map
    \begin{equation*}
        H^0(f^{-1}_{\tilde{Y}}(y),\mathcal{T}_{f_{\tilde{Y}}^{-1}(y)})\longrightarrow H^0(f^{-1}_{\tilde{Y}}(y),\Omega_{f^{-1}_{\tilde{Y}}(y)}^{[1]})^\vee
    \end{equation*}
    is an isomorphism. Therefore, the map of locally free sheaves
    \begin{equation*}
        (\tilde{f}^*)_*\mathcal{T}_{\tilde{X}^*/Y^*}\longrightarrow ((\tilde{f}^*)_*\Omega_{\tilde{X}^*/Y^*}^1)^\vee
    \end{equation*}
    is an isomorphism on each fiber, and thus it is an isomorphism.\par
    By Poincar\'e-duality, the natural pairing
    \begin{equation*}
        R^1(\tilde{f}^*)_*\underline{\C}\times R^{2n-1}(\tilde{f}^*)_*\underline{\C}\longrightarrow R^{2n}(\tilde{f}^*)_*\underline{\C}\cong\underline{\C}
    \end{equation*}
    is everywhere non-degenerate. This pairing is zero on $(\tilde{f}^*)_*\Omega_{\tilde{X}^*/Y^*}^1\times\mathcal{F}^n$. We hence obtain a canonical isomorphism
    \begin{equation*}
        ((\tilde{f}^*)_*\Omega_{\tilde{X}^*/Y^*}^1)^\vee\stackrel{\sim}{\longrightarrow}\mathcal{H}^{n-1,n}.
    \end{equation*}
    Composing these natural isomorphisms results in
    \begin{equation*}
        (f^*_{\tilde{Y}})_*\mathcal{T}_{X^*_{\tilde{Y}}/Y^*}\stackrel{\sim}{\longrightarrow}(\tilde{f}^*)_*\mathcal{T}_{\tilde{X}^*/Y^*}\stackrel{\sim}{\longrightarrow}((\tilde{f}^*)_*\Omega_{\tilde{X}^*/Y^*}^1)^\vee\stackrel{\sim}{\longrightarrow}\mathcal{H}^{n-1,n},
    \end{equation*}
    and taking global sections yields the desired isomorphism
    \begin{equation*}
        H^0(Y^*,(f^*_{\tilde{Y}})_*\mathcal{T}_{X^*_{\tilde{Y}}/Y^*})\stackrel{\sim}{\longrightarrow}H^0(Y^*,\mathcal{H}^{n-1,n}).
    \end{equation*}
     This concludes the proof of the surjectivity.\par
    The finiteness of the kernel follows from \Cref{isogeny_from_Automorphisms_to_Albanese_numerically_K-trivial}.
\end{proof}
\begin{lemma}\label{morphism_h'_factors_through_relative_Albanese}
    In the situation and notation from \Cref{relative_automorphisms_global_section_on_relative_Albanese_higher_codim}, denote the restriction of the morphism $h\colon X_{\tilde{Y}}\rightarrow (\tilde{Y}\times A)^\eta$ to $X^*_{\tilde{Y}}$ by $h^*\colon X^*_{\tilde{Y}}\rightarrow (Y^*\times A)^{\eta^*}$, where $\eta^*\in H^1(Y^*,\OO_{Y^*}^{\oplus g}/\Lambda)$ is the restriction of $\eta$ to $Y^*$. Then the morphism $h^*$ factors as
    \begin{equation*}
        X^*_{\tilde{Y}}\stackrel{\mathrm{alb}_{X^*_{\tilde{Y}}/Y^*}}{\longrightarrow} J_H^\zeta\longrightarrow (Y^*\times A)^{\eta^*}.
    \end{equation*}
\end{lemma}
\begin{proof}
    This is true locally by the universal property of the Albanese. Gluing the local factorizations together yields the global factorization.
\end{proof}
\begin{construction}\label{construction_retraction_local_systems}
    In the situation and notation from \Cref{morphism_h'_factors_through_relative_Albanese}, consider the variation of Hodge structures $\Lambda$ associated to the torus fiber bundle $\pi^\eta\colon (\tilde{Y}\times A)^\eta\rightarrow\tilde{Y}$ restricted to $Y^*$, i.e. $\Lambda$ consists of the local system
    \begin{equation*}
        \underline{\Lambda}=R^{2g-1}\pi^\eta_*\Z\vert_{Y^*},
    \end{equation*}
    the holomorphic vector bundle $\mathcal{G}:=\underline{\Lambda}\otimes\OO_{Y^*}$, the holomorphic subbundle
    \begin{equation*}
        \mathcal{E}^{g}:=R^{g-1}\pi^\eta_*\Omega^{g}_{(\tilde{Y}\times A)^\eta/\tilde{Y}}\vert_{Y^*}\subset\mathcal{G},
    \end{equation*}
    and the quotient $\mathcal{G}^{g-1,g}:=\mathcal{G}/\mathcal{E}^{g}$. The Jacobian fibration is given by $Y^*\times A\rightarrow Y^*$. Denote its sheaf of local sections by $\mathcal{I}_\Lambda=\mathcal{G}^{g-1,g}/\underline{\Lambda}$. Note that the morphism
    \begin{equation*}
        J_H^\zeta\longrightarrow (Y^*\times A)^{\eta^*}
    \end{equation*}
    from \Cref{morphism_h'_factors_through_relative_Albanese} induces a natural map
    \begin{equation*}
        \alpha\colon H_\Z\otimes\mathbb{R}\longrightarrow\underline{\Lambda}\otimes\mathbb{R}.
    \end{equation*}
    Note that this map is surjective, since $J_H^\zeta\rightarrow (Y^*\times A)^{\eta^*}$ is surjective.
    By Saito's decomposition theorem for proper K\"ahler morphisms \cite[][Thm. 0.5]{Saito_Kaehler_decomposition}, the local system $H_\Z\otimes\mathbb{R}$ is semi-simple. Therefore, the composition
    \begin{equation*}
        \underline{H^0(Y^*,H_\Z\otimes\mathbb{R})}\longhookrightarrow H_\Z\otimes\mathbb{R}\longrightarrow\underline{\Lambda}\otimes\mathbb{R}
    \end{equation*}
    is surjective. We claim that it is also injective.
\end{construction}
\begin{lemma}\label{lemma_retraction_local_system}
    The composition
    \begin{equation*}
        \underline{H^0(Y^*,H_\Z\otimes\mathbb{R})}\longrightarrow\underline{\Lambda}\otimes\mathbb{R}
    \end{equation*}
    from \Cref{construction_retraction_local_systems} is an isomorphism.
\end{lemma}
\begin{proof}
    We already know that this map is surjective. To prove injectivity, we show that the dual map
    \begin{equation*}
        (\Lambda\otimes\mathbb{R})^\vee\longrightarrow H^0(Y^*,H_\Z\otimes\mathbb{R})^\vee
    \end{equation*}
    is surjective.\par
    Let $\tilde{X}\rightarrow X_{\tilde{Y}}$ be the resolution of singularities from \Cref{construction_relative_Albanese}, and let $\tilde{f}\colon\tilde{X}\rightarrow\tilde{Y}$ be the induced morphism. Recall that $H_\Z$ was defined $H_\Z:=R^{2n-1}(\tilde{f}^*)_*\underline{\Z}/\mathrm{torsion}$, and hence
    \begin{equation*}
        H_\Z\otimes\mathbb{R}\cong R^{2n-1}(\tilde{f}^*)_*\underline{\mathbb{R}}.
    \end{equation*}
    Define the local system $L:=R^{1}(\tilde{f}^*)_*\underline{\mathbb{R}}$, and note that the Poincar\'e-duality induces a non-degenerate pairing
    \begin{equation*}
        L\times (H_\Z\otimes\mathbb{R})\longrightarrow\underline{\mathbb{R}}.
    \end{equation*}
    By Saito's decomposition theorem for proper K\"ahler morphisms \cite[][Thm. 0.5]{Saito_Kaehler_decomposition}, the local system $L$ is semi-simple. This shows that the induced pairing
    \begin{equation*}
        H^0(Y^*,L)\times H^0(Y^*,H_\Z\otimes\mathbb{R})\longrightarrow\mathbb{R}
    \end{equation*}
    is non-degenerate. In particular, we can identify $H^0(Y^*,H_\Z\otimes\mathbb{R})^\vee$ with $H^0(Y^*,L)$. Furthermore, Ehresmann's theorem \cite[][Thm. 9.3]{Voisin_HT_1} allows to identify $H^0(Y^*,L)$ with $H^1(\tilde{f}^{-1}(y),\mathbb{R})^{\pi_1(Y^*)}$ for a fixed $y\in Y^*$. We hence obtain
    \begin{equation*}
        H^0(Y^*,H_\Z\otimes\mathbb{R})^\vee\cong H^1(\tilde{f}^{-1}(y),\mathbb{R})^{\pi_1(Y^*)}.
    \end{equation*}
    Note that Deligne's global invariant cycle theorem \cite[][Thm. 4.24]{Voisin_HT_2} implies that the sequence
    \begin{equation*}
        0\longrightarrow H^1(\tilde{Y},\mathbb{R})\longrightarrow H^1(\tilde{X},\mathbb{R})\longrightarrow H^1(\tilde{f}^{-1}(y),\mathbb{R})^{\pi_1(Y^*)}\longrightarrow 0
    \end{equation*}
    is short exact. We can thus identify 
    \begin{equation*}
        H^0(Y^*,H_\Z\otimes\mathbb{R})^\vee\cong H^1(\tilde{X},\mathbb{R})/H^1(\tilde{Y},\mathbb{R})\cong H^1(\mathrm{Alb}(X),\mathbb{R})/H^1(\mathrm{Alb}(Y),\mathbb{R}).
    \end{equation*}
    Similarly, we can also identify
    \begin{equation*}
        (\Lambda\otimes\mathbb{R})^\vee\cong H^1(\mathrm{Alb}((Y\times A)^\eta),\mathbb{R})/H^1(\mathrm{Alb}(Y),\mathbb{R}),
    \end{equation*}
    and thus the map $(\Lambda\otimes\mathbb{R})^\vee\rightarrow H^0(Y^*,H_\Z\otimes\mathbb{R})^\vee$ can be identified with
    \begin{equation*}
        H^1(\mathrm{Alb}((\tilde{Y}\times A)^\eta),\mathbb{R})/H^1(\mathrm{Alb}(Y),\mathbb{R})\longrightarrow H^1(\mathrm{Alb}(X),\mathbb{R})/H^1(\mathrm{Alb}(Y),\mathbb{R}).
    \end{equation*}
    This is an isomorphism by \Cref{isomorphism_between_Albanese_X_and_torus_fiber_bundle}.
\end{proof}
\begin{construction}\label{construction_vector_fields_higher_codim}
    Applying the isomorphism from \Cref{lemma_retraction_local_system} to the commutative diagram of short exact sequences
    \[\begin{tikzcd}
	   0 & {H_\Z} & {H_\Z\otimes\mathbb{R}} & {(H_\Z\otimes\mathbb{R})/H_\Z} & 0 \\
	   0 & {\underline{\Lambda}} & {\underline{\Lambda}\otimes\mathbb{R}} & {(\underline{\Lambda}\otimes\mathbb{R})/\underline{\Lambda}} & 0
	   \arrow[from=1-1, to=1-2]
	   \arrow[from=1-2, to=1-3]
	   \arrow[from=1-2, to=2-2]
	   \arrow[from=1-3, to=1-4]
	   \arrow[from=1-3, to=2-3]
	   \arrow[from=1-4, to=1-5]
	   \arrow[from=1-4, to=2-4]
	   \arrow[from=2-1, to=2-2]
	   \arrow[from=2-2, to=2-3]
	   \arrow[from=2-3, to=2-4]
	   \arrow[from=2-4, to=2-5]
    \end{tikzcd}\]
    induces the commutative diagram of short exact sequences
    \[\begin{tikzcd}
	   0 & {H^0(Y^*,H_\Z)} & {H^0(Y^*,H_\Z\otimes\mathbb{R})} & {\ker(\delta)} & 0 \\
	   0 & {\Lambda} & {\Lambda\otimes\mathbb{R}} & {(\Lambda\otimes\mathbb{R})/\Lambda} & {0.}
	   \arrow[from=1-1, to=1-2]
	   \arrow[from=1-2, to=1-3]
	   \arrow[from=1-2, to=2-2]
	   \arrow[from=1-3, to=1-4]
	   \arrow[from=1-3, to=2-3]
	   \arrow[from=1-4, to=1-5]
	   \arrow[from=1-4, to=2-4]
	   \arrow[from=2-1, to=2-2]
	   \arrow[from=2-2, to=2-3]
	   \arrow[from=2-3, to=2-4]
	   \arrow[from=2-4, to=2-5]
    \end{tikzcd}\]
    Here, $\delta\colon H^0(Y^*,(H_\Z\otimes\mathbb{R})/H_\Z)\rightarrow H^1(Y^*,H_\Z)$ is the boundary map. Note that we showed that the vertical map in the middle is an isomorphism, and hence the vertical map on the left hand side is injective with finite cokernel. Thus, the map
    \begin{equation*}
        \ker(\delta)\longrightarrow(\Lambda\otimes\mathbb{R})/\Lambda
    \end{equation*}
    is surjective with finite kernel. Consider the composition
    \begin{equation*}
        \ker(\delta)\subset H^0(Y^*,(H_\Z\otimes\mathbb{R})/H_\Z)\longhookrightarrow H^0(Y^*,\mathcal{J}_H),
    \end{equation*}
    and note that this composition has image in $H^0(Y^*,\mathcal{J}_H)^\circ$. By \Cref{relative_automorphisms_global_section_on_relative_Albanese_higher_codim},
    the homomorphism
    \begin{equation*}
        \mathrm{Aut}^\circ(X^*_{\tilde{Y}}/Y^*)\longrightarrow H^0(Y^*,\mathcal{J}_H)^\circ
    \end{equation*}
    is surjective with finite kernel. Define $A'$ as the identity component of the preimage of $\ker(\delta)$ along this homomorphism. Then the composition
    \begin{equation*}
        A'\longrightarrow\ker(\delta)\longrightarrow(\Lambda\otimes\mathbb{R})/\Lambda\cong A
    \end{equation*}
    is surjective with finite kernel. This gives $A'$ the structure of a torus, and the map $\alpha\colon A'\rightarrow A$ becomes an isogeny. Thus, every element $a'\in A'$ defines an automorphism
    \begin{equation*}
        \mathrm{tr}(a')\colon X^*_{\tilde{Y}}\stackrel{\sim}{\longrightarrow}X^*_{\tilde{Y}}
    \end{equation*}
    such that the following diagram commutes
    \[\begin{tikzcd}
	   {X^*_{\tilde{Y}}} & {X^*_{\tilde{Y}}} \\
	   {(\tilde{Y}\times A)^\eta} & {(\tilde{Y}\times A)^\eta.}
	   \arrow["{\mathrm{tr}(a')}", from=1-1, to=1-2]
	   \arrow[from=1-1, to=2-1]
	   \arrow[from=1-2, to=2-2]
	   \arrow["{\mathrm{tr}(\alpha(a'))}", from=2-1, to=2-2]
    \end{tikzcd}\]
\end{construction}
\begin{lemma}\label{automorphisms_on_Xstar_are_bimeromorphic_on_X_higher_codim}
    The following assertions are true.
    \begin{enumerate}[label=(\roman*)]
        \item The induced map $A'\times X^*_{\tilde{Y}}\rightarrow X^*_{\tilde{Y}}$ is a morphism.
        \item For all $a'\in A'$, the automorphism $\mathrm{tr}(a')\colon X^*_{\tilde{Y}}\stackrel{\sim}{\rightarrow}X^*_{\tilde{Y}}$ defines a bimeromorphic map $\mathrm{tr}(a')\colon X_{\tilde{Y}}\stackrel{\sim}{\dashrightarrow}X_{\tilde{Y}}$.
    \end{enumerate}
\end{lemma}
\begin{proof}
    We make use of the notation in \Cref{construction_relative_Albanese}, \Cref{construction_retraction_local_systems} and in \Cref{construction_vector_fields_higher_codim}.\par
    To prove the first assertion, consider the commutative diagram
    \[\begin{tikzcd}
	   {H^0(Y^*,(f^*_{\tilde{Y}})_*\mathcal{T}_{X^*_{\tilde{Y}}/Y^*})} & {H^0(Y^*,\mathcal{H}^{n-1,n})} \\
	   {\mathrm{Aut}^\circ(X^*_{\tilde{Y}}/Y^*)} & {H^0(Y^*,\mathcal{J}_H)^\circ.}
	   \arrow["\sim", from=1-1, to=1-2]
	   \arrow[two heads, from=1-1, to=2-1]
	   \arrow[from=1-2, to=2-2]
	   \arrow[from=2-1, to=2-2]
    \end{tikzcd}\]
    from the proof of \Cref{relative_automorphisms_global_section_on_relative_Albanese_higher_codim}. Recall that the natural map $H^0(Y^*,H_\Z\otimes\mathbb{R})\hookrightarrow H^0(Y^*,\mathcal{H}^{n-1,n})$ is injective, and define $V_{\mathbb{R}}$ as the pre-image of $H^0(Y^*,H_\Z\otimes\mathbb{R})$ in $H^0(Y^*,(f^*_{\tilde{Y}})_*\mathcal{T}_{X^*_{\tilde{Y}}/Y^*})$. We then have the commutative diagram
    \[\begin{tikzcd}
	   {H^0(Y^*,(f^*_{\tilde{Y}})_*\mathcal{T}_{X^*_{\tilde{Y}}/Y^*})} & {H^0(Y^*,\mathcal{H}^{n-1,n})} & {H^0(Y^*,\mathcal{G}^{g-1,g})} \\
	   {V_{\mathbb{R}}} & {H^0(Y^*,H_\Z\otimes\mathbb{R})} & {\Lambda\otimes\mathbb{R}} \\
	   {A'} & {\ker(\delta)} & {A.}
	   \arrow["\sim", from=1-1, to=1-2]
	   \arrow[from=1-2, to=1-3]
	   \arrow[hook, from=2-1, to=1-1]
	   \arrow["\sim", from=2-1, to=2-2]
	   \arrow[from=2-1, to=3-1]
	   \arrow[hook, from=2-2, to=1-2]
	   \arrow["\sim", from=2-2, to=2-3]
	   \arrow[from=2-2, to=3-2]
	   \arrow[hook, from=2-3, to=1-3]
	   \arrow[from=2-3, to=3-3]
	   \arrow[from=3-1, to=3-2]
	   \arrow[from=3-2, to=3-3]
    \end{tikzcd}\]
    Recall that the lower row is surjective with finite kernel, and that the complex structure on $A'$ is induced by the complex structure on $A$ via $A'\rightarrow A$. Since the vertical map on the lower left hand side can be identified with the universal cover of $A'$, we obtain a complex structure on $V_{\mathbb{R}}$. To show that the map $A'\times X^*_{\tilde{Y}}\rightarrow X^*_{\tilde{Y}}$ is a morphism, it suffices to show that $V_{\mathbb{R}}$ is a complex subspace of $H^0(Y^*,(f^*_{\tilde{Y}})_*\mathcal{T}_{X^*_{\tilde{Y}}/Y^*})$. Since the upper row in the above diagram is a composition of complex linear maps, the diagram shows that $V_{\mathbb{R}}$ is a complex subspace of $H^0(Y^*,(f^*_{\tilde{Y}})_*\mathcal{T}_{X^*_{\tilde{Y}}/Y^*})$ if and only if $\Lambda\otimes\mathbb{R}$ is a complex subspace of $H^0(Y^*,\mathcal{G}^{g-1,g})$, where we endow $\Lambda\otimes\mathbb{R}$ with the complex structure induced by the universal covering map $\Lambda\otimes\mathbb{R}\rightarrow A$.\par
    Recall that $\mathcal{G}$ was defined as $\mathcal{G}:=\underline{\Lambda}\otimes\OO_{Y^*}$, and that $\mathcal{G}^{g-1,g}$ was defined as the quotient $\mathcal{G}/\mathcal{E}^{g}$, where
    \begin{equation*}
        \mathcal{E}^{g}:=R^{g-1}\pi^\eta_*\Omega^{g}_{(\tilde{Y}\times A)^\eta/\tilde{Y}}\vert_{Y^*}\subset\mathcal{G}.
    \end{equation*}
    Since translation on a torus acts trivial on $H^{g,g-1}$, we can identify $\mathcal{E}^{g}\cong H^{g-1}(A,\Omega_A^{g})\otimes\OO_{Y^*}$, and thus
    \begin{equation*}
        \mathcal{G}^{g-1,g}\cong H^{2g-1}(A,\C)/H^{g-1}(A,\Omega_A^{g})\otimes\OO_{Y^*}.
    \end{equation*}
    Since we can write
    \begin{equation*}
        A\cong\mathrm{Alb}(A)\cong (H^{2g-1}(A,\C)/H^{g-1}(A,\Omega_A^{g}))/H^{2g-1}(A,\Z),
    \end{equation*}
    the map
    \begin{equation*}
        \Lambda\otimes\mathbb{R}\longrightarrow H^{2g-1}(A,\C)/H^{g-1}(A,\Omega_A^{g})
    \end{equation*}
    is an isomorphism of complex vector spaces. As the right hand side naturally defines a complex subspace of $H^0(Y^*,\mathcal{G}^{g-1,g})$, the first assertion follows.\par
    To prove the second assertion, we have to show that the closure of the graph of $\mathrm{tr}(a')\colon X^*_{\tilde{Y}}\rightarrow X^*_{\tilde{Y}}$ defines an analytic subset of $X_{\tilde{Y}}\times X_{\tilde{Y}}$. Denote by $\tilde{X}\rightarrow X_{\tilde{Y}}$ the functorial resolution of singularities from \Cref{construction_relative_Albanese}.
    \begin{claim}\label{claim_meromorphic_vector_fields}
        All vector fields $v\in V_{\mathbb{R}}$ extend to meromorphic vertical vector fields on $\tilde{X}$.
    \end{claim}
    \begin{proof}[Proof of \Cref{claim_meromorphic_vector_fields}]
        Note that the fiber-wise bilinear pairing $\tilde{f}_*\mathcal{T}_{\tilde{X}/\tilde{Y}}\times\tilde{f}_*\Omega_{\tilde{X}/\tilde{Y}}^1\rightarrow\OO_{\tilde{Y}}$ induces a map
        \begin{equation*}
            \psi\colon\tilde{f}_*\mathcal{T}_{\tilde{X}/\tilde{Y}}\rightarrow(\tilde{f}_*\Omega_{\tilde{X}/\tilde{Y}}^1)^\vee.
        \end{equation*}
        Recall that $\tilde{f}$ is smooth over $Y^*$. As the resolution of singularities $\tilde{X}\rightarrow X_{\tilde{Y}}$ is functorial, vector fields on the fibers of $f_{\tilde{Y}}$ lift uniquely to vector fields on the fibers of $\tilde{f}$. This implies that $\psi$ is an isomorphism over $Y^*$. Note that $\tilde{f}_*\mathcal{T}_{\tilde{X}/\tilde{Y}}$ is a subsheaf of $\tilde{f}_*\mathcal{T}_{\tilde{X}}$, and hence torsion free. This shows that $\psi$ is injective.\par
        Let $D\subset\tilde{Y}$ be a polydisc, and let $U\subset D$ be an open subset such that $\overline{U}\subset D$. Let $\gamma$ be a local section of $\tilde{f}_*\Omega_{\tilde{X}/\tilde{Y}}^1$ over $D$. We can view $\gamma$ as a relative 1-form on $\tilde{f}^{-1}(D)$. Since the sheaves of smooth 1-forms on $\tilde{X}$, resp. $\tilde{Y}$ are acyclic, we can lift $\gamma$ to a smooth 1-form $\widehat{\gamma}$ on $\tilde{f}^{-1}(D)$ such that
        \begin{equation*}
            \widehat{\gamma}\vert_{\tilde{f}^{-1}(y)}=\gamma\vert_{\tilde{f}^{-1}(y)}
        \end{equation*}
        for all $y\in D\cap Y^*$.\par
        Fix a K\"ahler class $\omega$ on $\tilde{X}$. Let $v\in V_{\mathbb{R}}$, and denote its image on $H^0(Y^*,H_\Z\otimes\mathbb{R})$ by $w$ by the above isomorphism $V_{\mathbb{R}}\stackrel{\sim}{\rightarrow}H^0(Y^*,H_\Z\otimes\mathbb{R})$. Deligne's global invariant cycle theorem \cite[][Thm. 4.24]{Voisin_HT_2} implies that there is a cohomology class $\beta\in H^{2n-1}(\tilde{X},\mathbb{R})$ such that $\beta$ is mapped to $w$. We can hence calculate for all $y\in U\cap Y^*$
        \begin{equation*}
            \vert\langle v,\gamma\rangle(y)\vert=\left\vert\int\limits_{\tilde{f}^{-1}(y)}\widehat{\gamma}\wedge\beta\right\vert\leq \int\limits_{\tilde{f}^{-1}(y)}\vert\widehat{\gamma}\wedge\beta\vert\leq C\cdot\mathrm{vol}_\omega(\tilde{f}^{-1}(y)),
        \end{equation*}
        where $C$ is an upper bound of $\vert\widehat{\gamma}\wedge\beta\vert$ on the compact subset $\tilde{f}^{-1}(\overline{U})$. Since all fibers of $\tilde{f}$ are homologous,
        \begin{equation*}
            \mathrm{vol}_\omega(\tilde{f}^{-1}(y))=\langle[\omega]^n,[\tilde{f}^{-1}(y)]\rangle/n!
        \end{equation*}
        is independent of $y\in Y^*$. This shows that for all $v\in V_{\mathbb{R}}$, and $\gamma\in \Gamma(U\cap Y^*,\tilde{f}_*\Omega_{\tilde{X}/\tilde{Y}}^1)$ the function
        \begin{align*}
            \langle v,\gamma\rangle\colon U\cap Y^*&\longrightarrow\C\\
            y&\longmapsto\langle v,\gamma\rangle (y)
        \end{align*}
        is bounded. Moreover, since $v$ is a holomorphic vector field, and since $\gamma$ is a relative holomorphic 1-form, the function $\langle v,\gamma\rangle$ is holomorphic on $U\cap Y^*$. Riemann's extension theorem hence allows to extend this to a holomorphic function $\overline{\langle v,\gamma\rangle}$ on $U$.\par
        Define $\overline{v}$ by the assignment
        \begin{align*}
            \Gamma(U,(\tilde{f}_*\Omega_{\tilde{X}/\tilde{Y}}^1)^\vee)&\longrightarrow\Gamma(U,\OO_{\tilde{Y}})\\
            \gamma&\longmapsto\overline{\langle v,\gamma\rangle}.
        \end{align*}
        Note that this assignment is $\OO_{\tilde{Y}}$-linear, since the pairing $\langle v,\gamma\rangle$ of a holomorphic vector field and a relative holomorphic 1-form is $\OO_{\tilde{Y}}$-linear. Thus, since $U$ is an open subset of a polydisc, we can identify $\overline{v}$ with a local section
        \begin{equation*}
            \overline{v}\in\Gamma(U,(\tilde{f}_*\Omega_{\tilde{X}/\tilde{Y}}^1)^\vee).
        \end{equation*}
        By construction, $\overline{v}\vert_{U\cap Y^*}=\psi(v)$.\par
        Consider a point $y_0\in U\cap(\tilde{Y}\setminus Y^*)$. Choose a holomorphic function $u$ defined on $U$ that vanishes at $y_0$. As the cokernel of $\psi$ is supported on $\tilde{Y}\setminus Y^*$, there is a positive integer $k$ such that $u^k$ annihilates $\mathrm{coker}(\psi)$ near $y_0$. This implies that there is a local section $s$ of $\tilde{f}_*\mathcal{T}_{\tilde{X}/\tilde{Y}}$ over $U$ such that
        \begin{equation*}
            \psi(s)=u^k\overline{v}
        \end{equation*}
        on $U$. Recall that $\psi(v)=\overline{v}\vert_{U\cap Y^*}$, and hence $\psi(u^kv)=u^k\overline{v}\vert_{U\cap Y^*}$. As $\psi$ is injective, we conclude $u^kv=s\vert_{U\cap Y^*}$. In particular, $v=s/u^k$ is a meromorphic vector field on $U$.\par
        Since $\tilde{Y}$ is compact, there is an effective divisor $\Delta\subset\tilde{Y}$ supported on $\tilde{Y}\setminus Y^*$, and a non-negative integer $k$ such that $v$ defines a global section of $\tilde{f}_*\mathcal{T}_{\tilde{X}/\tilde{Y}}\otimes\OO_{\tilde{Y}}(k\Delta)$. The identification
        \begin{equation*}
            H^0(\tilde{Y},\tilde{f}_*\mathcal{T}_{\tilde{X}/\tilde{Y}}\otimes\OO_{\tilde{Y}}(k\Delta))\cong H^0(\tilde{X},\mathcal{T}_{\tilde{X}/\tilde{Y}}\otimes\OO_{\tilde{X}}(k\tilde{f}^*\Delta))
        \end{equation*}
        shows that $v$ extends to a meromorphic relative vector field on $\tilde{X}$.
    \end{proof}
    \begin{claim}\label{claim_orbit_relation_is_analytic}
        There is a closed analytic subset $R\subset\tilde{X}\times\tilde{X}$ such that for general $x\in\tilde{X}^*\subset\tilde{X}$,
        \begin{equation*}
            R\cap (\{x\}\times\tilde{X})=\{x\}\times A'x.
        \end{equation*}
    \end{claim}
    \begin{proof}[Proof of \Cref{claim_orbit_relation_is_analytic}]
        Choose a complex basis $v_1,\ldots,v_{g}$ of $V_{\mathbb{R}}$. \Cref{claim_meromorphic_vector_fields} shows that for each $v_i$ there is a non-negative integer $k_i$, and a divisor $\Delta_i$ supported on $\tilde{X}\setminus\tilde{X}^*$ such that $v_i$ defines a global section of $\mathcal{T}_{\tilde{X}/\tilde{Y}}\otimes\OO_{\tilde{X}}(k_i\Delta_i)$, i.e. we obtain a non-trivial map $\OO_{\tilde{X}}(-k_i\Delta_i)\rightarrow\mathcal{T}_{\tilde{X}/\tilde{Y}}$. Define $\mathcal{F}\subset\mathcal{T}_{\tilde{X}}$ as the saturation of the image of the map
        \begin{equation*}
            \bigoplus\limits_{i=1}^{g}\OO_{\tilde{X}}(-k_i\Delta_i)\longrightarrow\mathcal{T}_{\tilde{X}/\tilde{Y}}\longhookrightarrow\mathcal{T}_{\tilde{X}}.
        \end{equation*}
        Note that $\mathcal{F}\vert_{\tilde{X}^*}$ can be identified with $V_{\mathbb{R}}\otimes_\C\OO_{\tilde{X}^*}$. Consider the Lie-bracket
        \begin{equation*}
            [-,-]\colon\mathcal{F}\otimes\mathcal{F}\longrightarrow\mathcal{T}_{\tilde{X}}\longrightarrow\mathcal{T}_{\tilde{X}}/\mathcal{F},
        \end{equation*}
        which is $\OO_{\tilde{X}}$-linear since we quotient out $\mathcal{F}$. On $\tilde{X}^*$, this map is trivial, since the Lie bracket is trivial on $V_{\mathbb{R}}$. As $\mathcal{F}$ is saturated, the quotient $\mathcal{T}_{\tilde{X}}/\mathcal{F}$ is torsion-free. Thus, the Lie-bracket is trivial on all of $\tilde{X}$. This shows that $\mathcal{F}$ is a foliation of dimension $g$. Every point in $\tilde{X}^*$ is a regular point of the foliation $\mathcal{F}$, and the leaves are the orbits $A'x$.\par
        Let $Z\subset\tilde{X}$ be a subvariety, and denote its ideal sheaf by $\mathcal{I}_Z$. By spelling out the Leibniz rule, one sees that the differentiating along a holomorphic vector field induces a $\OO_{\tilde{X}}$-linear map
        \begin{equation*}
            \mathcal{F}\otimes\mathcal{I}_Z\longrightarrow\OO_Z.
        \end{equation*}
        We call a subvariety $Z\subset\tilde{X}$ invariant by $\mathcal{F}$, if the induced map $\mathcal{F}\otimes\mathcal{I}_Z\rightarrow\OO_Z$ is trivial.\par
        By Pereira's first theorem \cite[][Thm. 24]{Deroin_Guillot_foliations}, the union of the closed analytic subvarieties of $\tilde{X}$ of dimension $g$ that are invariant by $\mathcal{F}$ is the union of countably many closed analytic subvarieties. Since this union contains $\tilde{X}^*$, it is all of $\tilde{X}$, i.e. every point in $\tilde{X}$ lies in an invariant subvariety. We can hence apply Pereira's second theorem \cite[][Thm. 25]{Deroin_Guillot_foliations} to conclude that there exists a compact complex analytic variety $S$ of $\dim S=\dim\tilde{X}-g$ in the Fujiki class $\mathcal{C}$, and a dominant meromorphic map $p\colon\tilde{X}\dashrightarrow S$ such that the leaves of $\mathcal{F}$ are contained in the fibers of $p$.\par
        Denote by $\Gamma\subset\tilde{X}\times S$ the closure of the graph of $p$. As the leaves of $\mathcal{F}$ are connected, after taking the Stein factorization of $\Gamma\rightarrow S$ if necessary, we can assume that the fibers of $p'\colon\Gamma\rightarrow S$ are connected. Define $R\subset\tilde{X}\times\tilde{X}$ as the image of the morphism
        \begin{equation*}
            \Gamma\times_S\Gamma\longrightarrow\tilde{X}\times\tilde{X},
        \end{equation*}
        which is a closed analytic subset, since the morphism $\Gamma\times_S\Gamma\rightarrow\tilde{X}\times\tilde{X}$ is proper. To show that $R$ satisfies the property in \Cref{claim_orbit_relation_is_analytic}, let $x\in\tilde{X}^*$ be a general point. Note that we can identify $R\cap (\{x\}\times\tilde{X})$ with the image of
        \begin{equation*}
            \Gamma_x\times_S\Gamma\longrightarrow\{x\}\times\tilde{X},
        \end{equation*}
        where $\Gamma_x:=\Gamma\cap(\{x\}\times S)$. Since $x$ is general, $p$ is defined at $x$, i.e. $\Gamma_x=\{(x,p(x))\}$. We calculate
        \begin{equation*}
            \Gamma_x\times_S\Gamma=\{((x,p(x),(y,s))\in\Gamma_x\times\Gamma\mid p(x)=s\}=\{(x,p(x),y,p(x))\mid y\in(p')^{-1}(p(x))\}.
        \end{equation*}
        Since $x$ is general, we can assume that the fiber $(p')^{-1}(p(x))$ is irreducible. As it contains $A'x$, $(p')^{-1}(p(x))=A'x$. We can hence identify
        \begin{equation*}
            \Gamma_x\times_S\Gamma\cong\{(x,y)\in\tilde{X}\times\tilde{X}\mid y\in A'x\}.
        \end{equation*}
        This shows that for general $x\in\tilde{X}^*$, $R\cap(\{x\}\times\tilde{X})=\{x\}\times A'x$.
    \end{proof}
    To conclude the proof of the second assertion, note that since $\tilde{X}\rightarrow X_{\tilde{Y}}$ is a functorial resolution of singularities, the automorphisms $\mathrm{tr}(a')\colon X^*_{\tilde{Y}}\rightarrow X^*_{\tilde{Y}}$ lift to automorphisms $\mathrm{tr}(a')\colon\tilde{X}^*\rightarrow\tilde{X}^*$. Denote the graph of this morphism by $\Gamma^*_{a'}\subset\tilde{X}^*\times\tilde{X}^*$. It hence suffices to show that the closure
    \begin{equation*}
        \overline{\Gamma^*_{a'}}\subset\tilde{X}\times\tilde{X}
    \end{equation*}
    is a closed analytic subset.\par
    By \Cref{isomorphism_between_Albanese_X_and_torus_fiber_bundle} the morphism $\tilde{X}\rightarrow (\tilde{Y}\times A)^\eta$ induces an isomorphism
    \begin{equation*}
        \mathrm{Alb}(\tilde{X})\stackrel{\sim}{\longrightarrow}\mathrm{Alb}((\tilde{Y}\times A)^\eta).
    \end{equation*}
    Define the morphism $\tilde{\alpha}\colon A'\rightarrow\mathrm{Alb}(\tilde{X})$ as the composition
    \begin{equation*}
        A'\stackrel{\alpha}{\longrightarrow}A\stackrel{\mathrm{alb}_{(\tilde{Y}\times A)^\eta}}{\longrightarrow}\mathrm{Alb}((\tilde{Y}\times A)^\eta)\stackrel{\sim}{\longrightarrow}\mathrm{Alb}(\tilde{X}).
    \end{equation*}
    Note that $\tilde{\alpha}$ has finite kernel. Since for all $a\in A$ the diagram
        \[\begin{tikzcd}
	   {(\tilde{Y}\times A)^\eta} & {(\tilde{Y}\times A)^\eta} \\
	   {\mathrm{Alb}((\tilde{Y}\times A)^\eta)} & {\mathrm{Alb}((\tilde{Y}\times A)^\eta)}
	   \arrow["{\mathrm{tr}(a)}", from=1-1, to=1-2]
	   \arrow[from=1-1, to=2-1]
	   \arrow[from=1-2, to=2-2]
	   \arrow["{\mathrm{tr}(a)}", from=2-1, to=2-2]
    \end{tikzcd}\]
    commutes, it follows that
    \begin{equation*}
        \mathrm{alb}_{\tilde{X}}\vert_{\tilde{X}^*}\circ\mathrm{tr}(a')=\mathrm{alb}_{\tilde{X}}\vert_{\tilde{X}^*}+\tilde{\alpha}(a')
    \end{equation*}
    for all $a'\in A'$.\par
    Fix $a'\in A'$, and define
    \begin{equation*}
        G:=\{(x,y)\in R\mid \mathrm{alb}_{\tilde{X}}(y)=\mathrm{alb}_{\tilde{X}}(x)+\tilde{\alpha}(a')\},
    \end{equation*}
    which is a closed analytic subvariety of $\tilde{X}\times\tilde{X}$. By \Cref{claim_orbit_relation_is_analytic}, a general point in $R\cap(\tilde{X}^*\times\tilde{X}^*)$ is of the form $(x,\mathrm{tr}(b')(x))$. Hence, a general point in $G\cap(\tilde{X}^*\times\tilde{X}^*)$ is of the form $(x,\mathrm{tr}(b')(x))$ with $\tilde{\alpha}(b')=\tilde{\alpha}(a')$. Thus,
    \begin{equation*}
        G\cap\tilde{X}\times\tilde{X}=\bigcup\limits_{b'\in\ker(\tilde{\alpha})}\overline{\Gamma^*_{a'+b'}}.
    \end{equation*}
    As the left hand side is a closed analytic subset, so is the right hand side. This shows, that $\overline{\Gamma^*_{a'}}$ is an irreducible component of the closed analytic subset $G\cap\tilde{X}\times\tilde{X}$, and thus is a closed analytic subset itself. This proves the second assertion.
\end{proof}
We are now in position to prove \Cref{proposition_torus_action_h_equivariant}.
\begin{proof}[Proof of \Cref{proposition_torus_action_h_equivariant}]
    We first show that the morphisms $\mathrm{tr}(a')\colon X^*_{\tilde{Y}}\rightarrow X^*_{\tilde{Y}}$ extend to morphisms $\mathrm{tr}(a')\colon X_{\tilde{Y}}\rightarrow X_{\tilde{Y}}$. Recall that $X_{\tilde{Y}}$ is the normalization of the unique irreducible component in $X^{\mathrm{min}}\times_{Y}\tilde{Y}$, which is bimeromorphic to $X^{\mathrm{min}}$. \Cref{automorphisms_on_Xstar_are_bimeromorphic_on_X_higher_codim} gives rise to bimeromorphic maps
    \begin{equation*}
        \mathrm{tr}(a')\colon X^{\mathrm{min}}\longdashrightarrow X^{\mathrm{min}}.
    \end{equation*}
    It hence suffices to show that these bimeromorphic maps extend to morphisms $\mathrm{tr}(a')\colon X^{\mathrm{min}}\rightarrow X^{\mathrm{min}}$ over $Y$ for all $a'\in A'$.\par
    Define $X^*\subset X^{\mathrm{min}}$ as the image of $X^*_{\tilde{Y}}$ along the morphism $X_{\tilde{Y}}\rightarrow X^{\mathrm{min}}$, and note that translation by $a'\in A'$ is a morphism on $X^*$. Denote by $\Gamma^*\subset X^*\times X^*$ the graph of $\mathrm{tr}(a')\colon X^*\rightarrow X^*$, let $\overline{\Gamma}\subset X^{\mathrm{min}}\times X^{\mathrm{min}}$ be the closure. Choose a resolution of singularities $\tilde{\Gamma}\rightarrow\overline{\Gamma}$. We then have the commutative diagram of bimeromorphic morphisms
    \[\begin{tikzcd}
	   & {\tilde{\Gamma}} & \\
	   {X^{\mathrm{min}}} && {X^{\mathrm{min}}}
	   \arrow["{p_1}"', from=1-2, to=2-1]
	   \arrow["{p_2}", from=1-2, to=2-3]
    \end{tikzcd}\]
    such that $p_2\circ p_1^{-1}$ equals $\mathrm{tr}(a')$ as bimeromorphic maps. Denote the classes of $\mathcal{K}_{\tilde{\Gamma}}$, resp. $\mathcal{K}_{X^{\mathrm{min}}}$ in $H^{1,1}_{BC}(\tilde{\Gamma})$, resp. $H^{1,1}_{BC}(X^{\mathrm{min}})$ by $K_{\tilde{\Gamma}}$, resp. $K_{X^{\mathrm{min}}}$. Write
    \begin{equation*}
        K_{\tilde{\Gamma}}=p_1^*K_{X^{\mathrm{min}}}+\sum a_iE_i=p_2^*K_{X^{\mathrm{min}}}+\sum b_jF_j,
    \end{equation*}
    where $E_i$, resp. $F_j$ are exceptional prime divisors of $p_1$, resp. $p_2$. As $X^{\mathrm{min}}$ has terminal singularities, the rational coefficients $a_i$, resp. $b_j$ are all positive.\par
    Define
    \begin{equation*}
        D:=p_1^*K_{X^{\mathrm{min}}}-p_2^*K_{X^{\mathrm{min}}}.
    \end{equation*}
    As $K_{X^{\mathrm{min}}}$ is nef, the same holds for $p_2^*K_{X^{\mathrm{min}}}$. Since $p_1^*K_{X^{\mathrm{min}}}$ is trivial on the fibers of $p_1$, $-D$ is this $p_1$-nef. By the negativity lemma \cite[][Lem. 1.3]{Wang_Iitaka_conjecture_Kaehler_fiber_spaces}, $D$ is effective if and only if $(p_1)_*D$ is. We calculate
    \begin{equation*}
        (p_1)_*D=(p_1)_*(p_1^*K_{X^{\mathrm{min}}}-p_2^*K_{X^{\mathrm{min}}})=(p_1)_*(\sum b_jF_j-\sum a_iE_i)=\sum b_j(p_1)_*F_j,
    \end{equation*}
    since the $E_i$ are $p_1$-exceptional. Thus, $D$ is effective. Similarly, we can show that $-D$ is effective, by applying the negativity lemma to $p_2$. It follows, that $D$ is trivial, i.e.
    \begin{equation*}
        \sum a_iE_i=\sum b_jF_j.
    \end{equation*}
    In particular, $p_1$, and $p_2$ contract the same prime divisors.\par
    Let $I\subset X^{\mathrm{min}}$ be the indeterminacy locus of $\mathrm{tr}(a')$. Note that $I$ has codimension $\geq 2$. Suppose that there is a prime divisor $E\subset X^{\mathrm{min}}\setminus I$, on which $\mathrm{tr}(a')$ is not a local isomorphism. If $\mathrm{tr}(a')(E)$ is still a divisor, then $\mathrm{tr}(a')$ is generically finite of degree 1 on $E$, hence bimeromorphic. Since the target is normal, $\mathrm{tr}(a')$ is generically an isomorphism on $E$, a contradiction to the choice of $E$. This shows that $\mathrm{tr}(a')(E)$ cannot be a divisor. Therefore, its strict transformation $\tilde{E}$ on $\tilde{\Gamma}$ is a prime divisor, which is contracted by $p_2$ but not by $p_1$, a contradiction. Thus, there are closed analytic subsets $Z$, $Z'\subset X^{\mathrm{min}}$ of codimension $\geq 2$ such that
    \begin{equation*}
        \mathrm{tr}(a')\colon X^{\mathrm{min}}\setminus Z\stackrel{\sim}{\longrightarrow} X^{\mathrm{min}}\setminus Z'
    \end{equation*}
    is an isomorphism.\par
    Denote by $X^{\mathrm{sm}}\subset X^{\mathrm{min}}$ the smooth locus of $X^{\mathrm{min}}$. Define
    \begin{equation*}
        X^\circ:=(X^{\mathrm{min}}\setminus Z)\cap X^{\mathrm{sm}}\cap\mathrm{tr}(a')^{-1}(X^{\mathrm{sm}}),
    \end{equation*}
    and note that the complement $X^{\mathrm{min}}\setminus X^\circ$ has codimension $\geq 2$. Define $H:=H^2(X^{\mathrm{sm}},\Z)$. By \cite[][§3 Thm. 2]{Lojasiewicz_triangulation}, $X^{\mathrm{min}}$ can be triangulated such that each cell is either disjoint from $X^{\mathrm{sm}}$ or contained in $X^{\mathrm{sm}}$. As $X^{\mathrm{min}}$ is compact, this implies that $H$ is a finitely generated abelian group. Note that since the complements $X^{\mathrm{sm}}\setminus X^\circ$, $X^{\mathrm{sm}}\setminus\mathrm{tr}(a')(X^\circ)$ have codimension $\geq 2$, the restriction maps
    \begin{align*}
        H&\stackrel{\sim}{\longrightarrow} H^2(X^\circ,\Z),\\
        H&\stackrel{\sim}{\longrightarrow} H^2(\mathrm{tr}(a')(X^\circ),\Z)
    \end{align*}
    are isomorphisms. Thus, $\mathrm{tr}(a')$ induces an isomorphism
    \begin{equation*}
        \rho(a')\colon H\stackrel{\sim}{\longrightarrow} H.
    \end{equation*}
    We hence obtain a group homomorphism $A'\times H\rightarrow H$, which must be constant since $A'$ is a connected Lie group and since $H$ is discrete.\par
    Note that since $\tilde{\Gamma}$ is a compact K\"ahler manifold, $H^{1,1}_{BC}(\tilde{\Gamma})\cong H^{1,1}(\tilde{\Gamma})\subset H^2(\tilde{\Gamma},\mathbb{R})$. Fix a K\"ahler class $\omega$ on $X^{\mathrm{min}}$. We claim that $p_1^*[\omega]$ and $p_2^*[\omega]$ agree in $H^2(\tilde{\Gamma},\mathbb{R})$. Indeed, let $\delta:=p_1^*[\omega]-p_2^*[\omega]$, and let $\Gamma^\circ:=p_1^{-1}(X^\circ)$. Since $\mathrm{tr}(a')$ induces the identity on $H$, one sees that $\delta\vert_{\Gamma^\circ}=0$. Set $\Delta:=\tilde{\Gamma}\setminus\Gamma^\circ$. Recall that the kernel of the restriction map $H^{1,1}(\tilde{\Gamma})\rightarrow H^{1,1}(\Gamma^\circ)$ is generated by the divisorial components of $\Delta$. Note that each divisorial component of $\Delta$ is contracted by $p_1$, since $p_1(\Delta)\subset X^{\mathrm{min}}\setminus X^\circ$ is of codimension $\geq 2$. This allows to write $\delta=[B]$, where
    \begin{equation*}
        B=\sum \lambda_iB_i,
    \end{equation*}
    where each $B_i$ is $p_1$-exceptional. Let $C\subset\tilde{\Gamma}$ be a curve, which is contracted by $p_1$. Then
    \begin{equation*}
        B\cdot C=-p_2^*\omega\cdot C\leq 0.
    \end{equation*}
    In particular, $-B$ is $p_1$-nef. Since $(p_1)_*B=0$, the negativity lemma \cite[][Lem. 1.3]{Wang_Iitaka_conjecture_Kaehler_fiber_spaces} implies that $B$ is effective. Reversing the roles of $p_1$, $p_2$, $\mathrm{tr}(a')$, $\mathrm{tr}(a')^{-1}$ shows that $-B$ is effective. Thus, $B=0$, and hence $\delta=0$. This proves $p_1^*[\omega]=p_2^*[\omega]$.\par
    Let $C\subset\tilde{\Gamma}$ be a curve contracted by $p_1$. This is equivalent to $C\cdot p_1^*\omega=0$. Hence, the calculation
    \begin{equation*}
        C\cdot p_2^*\omega=C\cdot p_1^*\omega= 0
    \end{equation*}
    shows that the curve $C$ is also contracted by $p_2$. Therefore, we obtain a morphism $\mathrm{tr}(a')\colon X^{\mathrm{min}}\rightarrow X^{\mathrm{min}}$.\par
    To finish the proof, recall that $c(\eta)=0$ by \Cref{h_connected_fibers_c(eta)=0}. Write $A'=\C^{g}/\Lambda'$. We can by \Cref{lemma_extension_of_etale_cover_of_elliptic_curves} find a cohomology class $\eta'\in H^1(\tilde{Y},\OO_{\tilde{Y}}^{\oplus g}/\Lambda')$ such that there is a finite \'etale cover
    \begin{equation*}
        (\tilde{Y}\times A')^{\eta'}\longrightarrow(\tilde{Y}\times A)^\eta.
    \end{equation*}
    Moreover, $c(\eta')=0$. Choose a resolution of singularities $\tilde{X}\rightarrow X_{\tilde{Y}}$. Since $(\tilde{Y}\times A')^{\eta'}\rightarrow(\tilde{Y}\times A)^\eta$ is finite \'etale, the morphism
    \begin{equation*}
        \tilde{X}\times_{(\tilde{Y}\times A)^\eta}(\tilde{Y}\times A')^{\eta'}\rightarrow \tilde{X}
    \end{equation*}
    is finite \'etale as well. Since $X$ is bimeromorphic to $\tilde{X}$, there is a finite \'etale cover $X'\rightarrow X$ such that $X'$ is bimeromorphic to $\tilde{X}\times_{(\tilde{Y}\times A)^\eta}(\tilde{Y}\times A')^{\eta'}$.\par
    Consider the Stein factorization
    \[\begin{tikzcd}
	   {X_{\tilde{Y}}\times_{(\tilde{Y}\times A')^{\eta'}}(\tilde{Y}\times A')^{\eta'}} & {X_{\tilde{Y}}} & {X^{\mathrm{min}}} \\
	   & Z & {.}
	   \arrow[from=1-1, to=1-2]
	   \arrow[from=1-1, to=2-2]
	   \arrow[from=1-2, to=1-3]
	   \arrow[from=2-2, to=1-3]
    \end{tikzcd}\]
    The universal property of the fiber product induces an $A'$-action on $X_{\tilde{Y}}\times_{(\tilde{Y}\times A)^\eta}(\tilde{Y}\times A')^{\eta'}$ such that the upper row is $A'$-equivariant. This induces an $A'$-action on $Z$. The proof of \Cref{good_minimal_model_ascend_finite_etale_cover} shows that there is a small bimeromorphic morphism $X'^{\mathrm{min}}\rightarrow Z$ from a good minimal model of $X'$. We thus obtain for each $a'\in A'$ a bimeromorphic map $\mathrm{tr}(a')\colon X'^{\mathrm{min}}\dashrightarrow X'^{\mathrm{min}}$. Repeating the above argument, one shows that the bimeromorphic map $\mathrm{tr}(a')\colon X'^{\mathrm{min}}\dashrightarrow X'^{\mathrm{min}}$ is a morphism. In particular, we obtain an $A'$-action on $X'^{\mathrm{min}}$.\par
    By \Cref{good_minimal_model_ascend_finite_etale_cover}, the Iitaka fibration $f'\colon X'^{\mathrm{min}}\rightarrow Y'$ is given by the fibration part of the Stein factorization of $X'^{\mathrm{min}}\rightarrow X^{\mathrm{min}}\rightarrow Y$. Note that this composition factors bimeromorphically as
    \begin{equation*}
        X_{\tilde{Y}}\times_{(\tilde{Y}\times A)^\eta}(\tilde{Y}\times A')^{\eta'}\longrightarrow (\tilde{Y}\times A')^{\eta'}\longrightarrow\tilde{Y}\longrightarrow Y,
    \end{equation*}
    which is a composition of morphisms with connected fibers. Thus, $Y'=Y$.\par
    Note that $X'^{\mathrm{min}}\times_{Y}\tilde{Y}$ contains a unique irreducible component that is bimeromorphic to $X'^{\mathrm{min}}$. Denote by $X'_{\tilde{Y}}$ the normalization of this component, and denote by $f'_{\tilde{Y}}\colon X'_{\tilde{Y}}\rightarrow\tilde{Y}$ the induced morphism. Since $A'$ acts on $X'^{\mathrm{min}}$ over $Y$, we inherit an $A'$-action on $X'_{\tilde{Y}}$. Moreover, the induced morphism
    \begin{equation*}
        h'\colon X'_{\tilde{Y}}\longrightarrow (\tilde{Y}\times A')^{\eta'}
    \end{equation*}
    is $A'$-equivariant. Note that the base change along $(\tilde{Y}\times A')^{\eta'}\rightarrow (\tilde{Y}\times A)^\eta$ does not affect the fibers of $h$. Hence, the base change
    \begin{equation*}
        X_{\tilde{Y}}\times_{(\tilde{Y}\times A)^\eta}(\tilde{Y}\times A')^{\eta'}\longrightarrow (\tilde{Y}\times A')^{\eta'}
    \end{equation*}
    of $h$ is a morphism with connected fibers. Since the morphism $X_{\tilde{Y}}\times_{(\tilde{Y}\times A)^\eta}(\tilde{Y}\times A')^{\eta'}\rightarrow (\tilde{Y}\times A')^{\eta'}$ is bimeromorphic to $h'$, also $h'$ has connected fibers. This concludes the proof of \Cref{proposition_torus_action_h_equivariant}.
\end{proof}
From now on, we assume that we are in the situation and notation from \Cref{proposition_torus_action_h_equivariant}.
\begin{construction}\label{construction_W_higher_codim}
    Choose an open cover $\{U_i\}_{i\in I}$ of $\tilde{Y}$ and isomorphisms
    \begin{equation*}
        \varphi_i\colon (\pi^\eta)^{-1}(U_i)\stackrel{\sim}{\longrightarrow}U_i\times A
    \end{equation*}
    as in \Cref{special_local_trivialization}. For all $a\in A$, consider the subvariety of $f^{-1}_{\tilde{Y}}(U_i)$
    \begin{equation*}
        h^{-1}(\varphi_i^{-1}(U_i\times \{a\}))\subset f^{-1}_{\tilde{Y}}(U_i).
    \end{equation*}
    Since $h$ and $f_{\tilde{Y}}$ have connected fibers by \Cref{h_connected_fibers_c(eta)=0}, there is exactly one irreducible component of $h^{-1}(\varphi_i^{-1}(U_i\times \{a\})$, which dominates $U_i$. Denote this component by
    \begin{equation*}
        W_i(a)\subset h^{-1}(\varphi_i^{-1}(U_i\times \{a\})).
    \end{equation*}
    Fix a point $a_0\in A$. Note that over $U_i\cap U_j$, we have identifications of subvarieties in $f^{-1}_{\tilde{Y}}(U_i\cap U_j)$
    \begin{equation*}
        W_i(a_0)\vert_{U_i\cap U_j}= W_j(a_0+a_{ij})\vert_{U_i\cap U_j}\subset f^{-1}_{\tilde{Y}}(U_i\cap U_j).
    \end{equation*}
    Composing with $\mathrm{tr}(-a_{ij})$ thus yields the isomorphism
    \begin{equation*}
        \lambda_{ij}\colon W_i(a_0)\vert_{U_i\cap U_j}= W_j(a_0+a_{ij})\vert_{U_i\cap U_j}\stackrel{\sim}{\longrightarrow}W_j(a_0)\vert_{U_i\cap U_j}.
    \end{equation*}
    Note that the isomorphisms $\lambda_{ij}$ satisfy the cocycle condition. Thus, they allow to glue the local patches $W_i(a_0)\rightarrow U_i$ to $W\rightarrow\tilde{Y}$. Let $\tilde{W}\rightarrow W$ be a resolution of singularities, and denote the composition by $u\colon\tilde{W}\rightarrow\tilde{Y}$. Consider $(\tilde{W}\times A)^{u^*\eta}$.
\end{construction}
We are now in position to prove \Cref{precise_main_result_higher_codim}.
\begin{proof}[Proof of \Cref{precise_main_result_higher_codim}]
    In the notation of \Cref{construction_W_higher_codim}, consider the meromorphic map
    \begin{align*}
        f^{-1}_{\tilde{Y}}(U_i)&\longdashrightarrow W_i(a_0)\\
        x&\longmapsto\mathrm{tr}(a_0-p_2(\varphi_i(h(x))))(x),
    \end{align*}
    where $p_2\colon U_i\times A\rightarrow A$ is the projection onto the second factor. Note that this defines a meromorphic retraction of $W_i(a_0)\hookrightarrow f^{-1}_{\tilde{Y}}(U_i)$. Moreover, these local meromorphic maps glue to a meromorphic map
    \begin{equation*}
        X_{\tilde{Y}}\longdashrightarrow\tilde{W}.
    \end{equation*}
    The universal property of the fiber product gives rise to a meromorphic map
    \begin{equation*}
        X_{\tilde{Y}}\longdashrightarrow\tilde{W}\times_{\tilde{Y}}(\tilde{Y}\times A)^{\eta}\cong (\tilde{W}\times A)^{u^*\eta}.
    \end{equation*}
    It is not difficult to see that this meromorphic map is bimeromorphic.\par
    By Hironaka's resolution of indeterminacies, there is a proper bimeromorphic morphism $\tilde{X}\rightarrow X_{\tilde{Y}}$ from a compact connected K\"ahler manifold $\tilde{X}$ such that the meromorphic map $X_{\tilde{Y}}\dashrightarrow\tilde{W}$ extends to a morphism
    \begin{equation*}
        \tilde{X}\longrightarrow\tilde{W}.
    \end{equation*}
    This shows that $\tilde{W}$ is the holomorphic image of a compact connected K\"ahler manifold. By \cite[][Lem. 4.6]{Fujiki_closedness_Douady_space}, $\tilde{W}$ is hence bimeromorphic to a compact connected K\"ahler manifold. Thus, we can assume without loss of generality that $\tilde{W}$ is a compact connected K\"ahler manifold. Finally, $c(\eta)=0$ by \Cref{proposition_torus_action_h_equivariant}. Thus, also $c(u^*\eta)=0$. This concludes the proof of \Cref{precise_main_result_higher_codim}.
\end{proof}
\printbibliography

@book{Voisin_HT_1,
    title = {Hodge Theory and Complex Algebraic Geometry},
    volume = {1},
    author = {Voisin, Claire},
    editor = {Schneps, Leila},
    publisher = {Cambridge Univ. Press},
    series = {Cambridge Stud. Adv. Math.},
    number = {76},
    year={2002}
}

@book{Voisin_HT_2,
    title = {Hodge Theory and Complex Algebraic Geometry},
    volume = {2},
    author = {Voisin, Claire},
    editor = {Schneps, Leila},
    publisher = {Cambridge Univ. Press},
    series = {Cambridge Stud. Adv. Math.},
    number = {77},
    year={2003}
}

@article{Hao_Schreieder_3-folds,
    author = {Hao, Feng and Schreieder, Stefan},
	journal = {Geom. Topol},
	title = {Holomorphic one-forms without zeros on threefolds},
	volume = {25},
    number = {1},
    pages = {409--444},
	year = {2021}
}

@article{Kotschick_1-forms_fibrations_char-numbers,
    author = {Kotschick, Dieter},
    title = {Holomorphic one-forms, fibrations over the circle, and characteristic numbers of Kähler manifolds},
    journal = {Math. Proc. Cambridge Philos. Soc.},
    volume = {172},
    number = {1},
    pages = {95--103},
    year = {2022}
}

@article{Hao_Nowhere_Vanishing_Holomorphic_One-Forms_on_Varieties_of_Kodaira_Codimension_One,
    author = {Hao, Feng},
    title = {Nowhere Vanishing Holomorphic One-Forms on Varieties of Kodaira Codimension One},
    journal = {Int. Math. Res. Not. IMRN},
    volume = {2024},
    number = {6},
    pages = {4501--4515},
    year = {2023}
}

@article{Nakayama_elliptic-fibrations,
    author = {Nakayama, Noboru},
    title = {Local Structure of an Elliptic Fibration},
    journal = {Adv. Stud. Pure Math.},
    volume = {35},
    pages = {185--295},
    year = {2002}
}

@article{Claudon_Hoering_Lin_fundamental_group,
    author = {Claudon, Beno\^{i}t and Höring, Andreas and Lin, Hsueh-Yung},
    title = {The fundamental group of compact Kähler threefolds},
    journal = {Geom. Topol.},
    volume = {23},
    number = {7},
    pages = {3233--3271},
    year = {2019}
}

@article{Claudon_tori_fiber_bundles,
    author = {Claudon, Beno\^{i}t},
    title = {Smooth families of tori and linear Kähler groups},
    journal = {Ann. Fac. Sci. Toulouse Math.},
    volume = {27},
    number = {3},
    pages = {477--496},
    year = {2018}
}

@article{Popa_Schnell_Kodaira_dimension_zeros_1-forms,
    author = {Popa, Mihnea and Schnell, Christian},
    title = {Kodaira dimension and zeros of holomorphic one-forms},
    journal = {Ann. of Math.},
    volume = {179},
    number = {3},
    pages = {1109--1120},
    year = {2014}
}

@article{Varouchas_Kaehler_spaces_proper_open_morphisms,
    author = {Varouchas, Jean},
    title = {K\"ahler spaces and proper open morphisms},
    journal = {Math. Ann.},
    volume = {283},
    pages = {13--52},
    year = {1989}
}

@article{Catanese_quasi_etale,
    author = {Catanese, Fabrizio},
    title = {Q.E.D. for algebraic varieties},
    journal = {J. Differential Geom.},
    volume = {77},
    pages = {43--75},
    year = {2007}
}

@misc{church_gmm,
      title={Nowhere vanishing 1-forms on varieties admitting a good minimal model}, 
      author={Church, Benjamin},
      year={2024},
      eprint={2410.22753},
      archivePrefix={arXiv},
      primaryClass={math.AG},
      url={https://arxiv.org/abs/2410.22753}, 
}

@misc{Hao_Wang_Zhang_gmm,
      title={Good minimal models with nowhere vanishing holomorphic $1$-forms}, 
      author={Hao, Feng and Wang, Zichang and Zhang, Lei},
      year={2024},
      eprint={2412.12582},
      archivePrefix={arXiv},
      primaryClass={math.AG},
      url={https://arxiv.org/abs/2412.12582}, 
}

@misc{Pietig_1-forms_Kaehler_threefolds,
      title={Holomorphic 1-forms without zeros on K\"ahler threefolds}, 
      author={Pietig, Simon},
      year={2025},
      eprint={2506.22067},
      archivePrefix={arXiv},
      primaryClass={math.AG},
      url={https://arxiv.org/abs/2506.22067}, 
}

@book{Kollar_Mori,
    author = {Koll{\'a}r, J{\'a}nos and Mori, Shigefumi},
    title = {Birational Geometry of Algebraic Varieties},
    publisher = {Cambridge University Press},
    series = {Cambridge Tracts in Mathematics},
    mumber = {134},
    year = {2010}
}

@article{Kawamata_minimal_models_Kodaira_dim_fiber_spaces,
    author = {Kawamata, Yujiro},
    title = {Minimal models and the Kodaira dimension of algebraic fiber spaces.},
    journal = {J. reine angew. Math.},
    volume = {363},
    pages = {1--46},
    year = {1985}
}

@book{Hatcher_AT,
    title = {Algebraic topology},
	publisher = {Cambridge university press},
	author = {Hatcher, Allen},
	year = {2001},
}

@article{Fujiki_closedness_Douady_space,
    author = {Fujiki, Akira},
    title = {Closedness of the Douady Spaces of Compact K{\"a}hler Spaces},
    journal = {Publ. Res. Inst. Math. Sci.},
    volume = {14},
    pages = {1--52},
    year = {1978}
}

@misc{Correa_Kollar_Schreieder_Wang,
      title={Invisible singularities in complex algebraic geometry}, 
      author={Corr{\^e}a, Maur{\'i}cio and Koll{\'a}r, J{\'a}nos and Schreieder, Stefan and Wang, Botong},
      year={2026},
      eprint={2608.10973},
      archivePrefix={arXiv},
      primaryClass={math.AG},
      url={https://arxiv.org/abs/2608.10973}, 
}

@book{Lazarsfeld1,
    author = {Lazarsfeld, Robert},
    title = {Positivity in Algebraic Geometry I },
    publisher = {Springer},
    series = {Ergeb. Math. Grenzgeb.},
    number = {48},
    year = {2004}
}

@article{Bakker_Guenancia_Lehn,
  title={Algebraic approximation and the decomposition theorem for K{\"a}hler Calabi–Yau varieties},
  author={Bakker, Benjamin and Guenancia, Henri and Lehn, Christian},
  journal={Invent. Math.},
  year={2020},
  volume={228},
  pages={1255--1308},
}

@article{Saito_Kaehler_decomposition,
    author = {Saito, Morihiko},
    title = {Decomposition theorem for proper Kähler morphisms},
    journal = {Tohoku Math. J.},
    volume = {42},
    pages = {127--147},
    year = {1990}
}

@article{Villadsen_proof_Popa-Schnell,
    author = {Villadsen, Mads Bach},
    title = {Kodaira dimension and zeros of holomorphic one-forms, revisited},
    journal = {Math. Res. Lett.},
    volume = {29},
    pages = {1891–1898},
    year = {2022}
}

@article{Simpson_moduli_rank_1_local_systems,
    author = {Simpson, Carlos},
    title = {Subspaces of moduli spaces of rank one local systems},
    journal = {Ann. Sci. \'Ec. Norm. Sup\'er.},
    volume = {26},
    pages = {361--401},
    year = {1993}
}

@online{Schnell_generic_vanishing,
    author = {Schnell, Christian},
    title = {Generic vanishing theorem},
    url = {https://www.math.stonybrook.edu/~cschnell/pdf/notes/generic-vanishing.pdf},
    year = {2013}
}

@article{Takegoshi_relative_vanishing,
    author = {Takegoshi, Kensho},
    title = {Relative vanishing theorems in analytic spaces},
    journal = {Duke Math. J.},
    volume = {52},
    pages = {273--279},
    year = {1985}
}

@article{Wang_Iitaka_conjecture_Kaehler_fiber_spaces,
    author = {Wang, Juanyong},
    title = {On the Iitaka Conjecture $C_{n,m}$ for K{\"a}hler fiber spaces},
    journal = {Ann. Fac. Sci. Toulouse Math.},
    volume = {30},
    pages = {813--897},
    year = {2021}
}

@article{Bierstone_Milman,
    author = {Bierstone, Edward and Milman, Pierre},
    title = {Canonical desingularization in characteristic zero by blowing up the maximum strata of a local invariant},
    journal = {Invent. Math.},
    volume = {128},
    pages = {207–-302},
    year = {1997}
}

@misc{Hacon_Li_Xie_Q-factorialization,
      title={Fujiki Class $\mathcal C$ Varieties and a K\"ahler Criterion}, 
      author={Hacon, Christopher and Li, Yi and Xie, Lingyao},
      year={2026},
      eprint={2608.20588},
      archivePrefix={arXiv},
      primaryClass={math.AG},
      url={https://arxiv.org/abs/2608.20588}, 
}

@article{Grauert_analytische_Garbentheorie_Modulraeume,
    author = {Grauert, Hans},
    title = {Ein Theorem der analytischen Garbentheorie und die Modulr{\"a}ume komplexer Strukturen},
    journal = {Publ. Math.IH\'{E}S},
    volume = {5},
    pages = {5--64},
    year = {1960}
}

@book{SGA,
    author = {Grothendieck, Alexander and Raynaud, Mich{\`e}le},
    title = {Rev{\^e}tements {\'E}tales et Groupe Fondamental},
    publisher = {Springer},
    series = {Lecture Notes in Math},
    number = {224},
    year = {1971}
}

@article{Deroin_Guillot_foliations,
    author = {Deroin, Bertrand and Guillot, Adolfo},
    title = {Uniformizable foliated projective structures along singular foliations},
    journal = {Math. Z.},
    volume = {311},
    number = {85},
    year = {2025}
}

@article{Lojasiewicz_triangulation,
    author = {{\L}ojasiewicz, Stanis{\l}aw},
    title = {Triangulation of semi-analytic sets},
    journal = {Ann. Scuola Norm. Sup. Pisa Cl. Sci.},
    volume = {18},
    pages = {449--474},
    year = {1964}
}

@article{Kebekus_Schnell_pullback_reflexive_forms,
  title = {Extending holomorphic forms from the regular locus of a complex space to a resolution of singularities},
  author = {Kebekus, Stefan and Schnell, Christian},
  journal = {J. Am. Math. Soc.},
  year={2021},
  volume={34}
}

@misc{OAI:Log-abundance-for-compact-Kahler-spaces-under-logarithmic-Iitaka-subadditivity-October-4-2026,
  author = {{OpenAI}},
  title = {{Log abundance for compact K{\"a}hler spaces under logarithmic Iitaka subadditivity}},
  howpublished = {OpenAI Math Release preprint
                  \href{https://github.com/openai/math/blob/main/preprints/Log-abundance-for-compact-Kahler-spaces-under-logarithmic-Iitaka-subadditivity-October-4-2026/main.pdf}{OAI:Log-abundance-for-compact-Kahler-spaces-under-logarithmic-Iitaka-subadditivity-October-4-2026}},
  year = {2026}
}

@misc{OAI:Orbifold-and-logarithmic-Iitaka-subadditivity-September-26-2026,
  author = {{OpenAI}},
  title = {{Orbifold and logarithmic Iitaka subadditivity}},
  howpublished = {OpenAI Math Release preprint
                  \href{https://github.com/openai/math/blob/main/preprints/Orbifold-and-logarithmic-Iitaka-subadditivity-September-26-2026/paper.pdf}{OAI:Orbifold-and-logarithmic-Iitaka-subadditivity-September-26-2026}},
  year = {2026}
}
\end{document}